\documentclass[1p,authoryear]{elsarticle}
\usepackage{array}
\usepackage{booktabs} %

\usepackage{amsmath,amssymb,amsfonts,textcomp}
\usepackage{physics, mathtools}
\usepackage{mathrsfs}
\usepackage{natbib}
\usepackage{graphicx}
\usepackage{subfig}
\usepackage{enumitem}
\usepackage{xcolor}
\usepackage{soul}
\usepackage{hyperref}
\usepackage{tcolorbox}
\newcommand{\Chebyshevbasislow}{12}
\newcommand{\Chebyshevbasismed}{\the\numexpr(\Chebyshevbasislow*4)\relax}
\newcommand{\Chebyshevbasishigh}{\the\numexpr(\Chebyshevbasislow*8)\relax}

\newtheorem{thm}{Theorem}[section]
\newtheorem{lem}{Lemma}[section]
\newtheorem{prop}{Proposition}[section]
\newtheorem{cor}{Corollary}[section]
\newtheorem{asm}{Assumption}[section]
\newtheorem{rem}{Remark}[section]
\newproof{proof}{Proof}
\newcommand{\bunderline}[1]{\underline{#1\mkern-2mu}\mkern2mu}
\newcommand{\mylabel}[2]{#2\def\@currentlabel{#2}\label{#1}}
\newcommand{\E}{\mathbb{E}}

\begin{document}
\begin{frontmatter}

\title{Multidimensional Projection Filters via Automatic Differentiation and Sparse-Grid Integration}

\author[firstaddress,secondaddress]{Muhammad Fuady Emzir\corref{mycorrespondingauthor}}
\cortext[mycorrespondingauthor]{Corresponding author}
\ead{muhammad.emzir@kfupm.edu.sa}

\author[thirdaddress]{Zheng Zhao}
\author[fourthaddress]{Simo S\"arkk\"a}

\address[firstaddress]{Control and Instrumentation Engineering Department, King Fahd University of Petroleum and Minerals, Dhahran, Saudi Arabia}
\address[secondaddress]{Interdisciplinary Center of Smart Mobility and Logistics, King Fahd University of Petroleum and Minerals, Dhahran, Saudi Arabia}
\address[thirdaddress]{Uppsala University, Uppsala, Sweden}
\address[fourthaddress]{Aalto University, Espoo, Finland}

\begin{abstract}
    The projection filter is a technique for approximating the solutions of optimal filtering problems. In projection filters, the Kushner--Stratonovich stochastic partial differential equation that governs the propagation of the optimal filtering density is projected to a manifold of parametric densities, resulting in a finite-dimensional stochastic differential equation. Despite the fact that projection filters are capable of representing complicated probability densities, their current implementations are limited to Gaussian family or unidimensional filtering applications. This work considers a combination of numerical integration and automatic differentiation to construct projection filter algorithms for more generic problems. Specifically, we provide a detailed exposition of this combination for the manifold of the exponential family, and show how to apply the projection filter to multidimensional cases. We demonstrate numerically that based on comparison to a finite-difference solution to the Kushner--Stratonovich equation and a bootstrap particle filter with systematic resampling, the proposed algorithm retains an accurate approximation of the filtering density while requiring a comparatively low number of quadrature points. Due to the sparse-grid integration and automatic differentiation used to calculate the expected values of the natural statistics and the Fisher metric, the proposed filtering algorithms are highly scalable. They therefore are suitable to many applications in which the number of dimensions exceeds the practical limit of particle filters, but where the Gaussian-approximations are deemed unsatisfactory.
\end{abstract}

\begin{highlights}%
        \item We use automatic differentiation and sparse-grid integration to automate the construction of the projection filter.
        \item We present methods for constructing projection filters for multidimensional filtering problems using a non-Gaussian parametric density. 
        \item We show that the practical performance of the filter is comparable to the particle filter and finite difference based solutions to the Kushner--Stratonovich equation.
        \item An open-source implementation of the method is available.
\end{highlights}

\end{frontmatter}

\section{Introduction}
Obtaining the statistically best estimate of an unobserved stochastic process based on a stochastic observation process generated by it is known as the optimal filtering problem \citep{Jazwinski1970,Liptser2001,Bain2009,Sarkka2013}. The optimal filtering problem is equivalent to determining the evolution of the conditional probability density of the state of the unobserved process given the observation process up to the current time. When the stochastic processes are modeled as It\^o type of stochastic differentiation equations (SDEs), under suitable regularity conditions, the evolution of the conditional density is described by a stochastic partial differential equation (SPDE) known as the Kushner--Stratonovich equation  \citep{Jazwinski1970,RobertS.Liptser2010,Kushner1967,Kushner1967a,Wonham1963}. The Kushner--Stratonovich equation is a nonlinear SPDE that has a complicated structure which makes it hard to solve \citep{Ceci2013}. \citet{Zakai1969} introduced an alternative representation to the Kushner--Stratonovich equation, resulting in a linear SPDE for an unnormalized conditional density of the state. The solutions of the Zakai equation are, nevertheless, still hard to obtain as they are generally infinite-dimensional except when the Lie algebra naturally associated with it is finite-dimensional \citep{Maurel1984,Brigo1999}. In most cases, we can only hope for approximations of the solutions.

Among approximations to the solution of the Kushner--Stratonovich equation is the projection filter which has been proposed in the late 80s  \citep{Hanzon1991, Brigo1995,Brigo1998,Brigo1999}. Essentially, it projects the dynamics of the conditional probability density given by the Kushner--Stratonovich equation onto a finite-dimensional manifold of parametric densities, resulting in a finite-dimensional stochastic differential equation (SDE). 
So far, in practice, the use of the projection filter has been very limited; outside of the Gaussian density family, projection filters have mainly been developed for unidimensional dynamical systems \citep{Brigo1995,Brigo1999,AzimiSadjadi2005,Armstrong2015,Koyama2018, Tronarp2019}. One of the main issues in projection filter is that the filter equation is model-specific, and depends on the chosen parametric family and the natural statistics used. This means that in order to construct the filter, one needs to derive the filtering equation analytically. This manual procedure is prone to mistakes \citep{Armstrong2015}. Another challenge is that in order to propagate the parameters of SDEs, the log-partition function and several expected statistics need to be computed at every step via numerical integration. These numerical integrals are difficult to calculate since one needs to take care of different integrands with different supports simultaneously. To reduce the computational burden, so far, projection filter's numerical implementation has relied upon a recursive procedure to compute the expected values of the natural statistics and the Fisher information matrix \citep{Brigo1995}. This recursion is, in general, only feasible for unidimensional problems. %

Previously, \citet{Armstrong2015} proposed a functional ring as a symbolic layer for abstraction of the operations in a function space including addition, multiplication, scalar multiplication, and differentiation. Whilst this abstraction provides some degree of automation to derivation of the filter equation, numerical integrations for computing the expected values of the natural statistics are still necessary. These integrations can be time-consuming and repetitive, as the expected values of the natural statistics need to be calculated by taking their different supports into account. Moreover, the implementation in \citet{Armstrong2015} is restricted to unidimensional problems. 

Although the projection filter has not recently been used in many applications, there are recent theoretical results for it. For example, \citet{Armstrong2016,Armstrong2018,Armstrong2018a} study alternative projections than the standard Stratonovich projection. Different distance measures than the Hellinger distance have also been recently investigated in \citet{Armstrong2015}. Apart from the exponential family, the mixture family has been considered in \citet{Armstrong2015}, and it was proven that if the mixture components are chosen to be the Galerkin basis functions, then the projection filter under the $L^2$ metric is equivalent to a Galerkin method applied to the Kushner--Stratonovich SPDE. The projection filter is also gaining popularity as an alternative approach to solve finite-dimensional quantum filtering problems \citep{Handel2005,Gao2019,Gao2020,Gao2020a}. In these formulations, the original projection filter formulation is modified using a different type of projection from \citet[Chapter 7]{Amari2000}. For a recent review of these topics, we refer the reader to \citet{Brigo2022}.

There is also another line of research that offers methods to approximate the filtering densities via random samples. Among these methods is particle filtering which is also known as sequential Monte Carlo \citep{Doucet2001}. Although these methods can approximate the optimal filters for almost all types of filtering problems, to avoid particle weight collapse, the number of particles in these methods often scales exponentially with the dimension of the problem \citep{Snyder2008,Poterjoy2015,Beskos2017}. Another class of filters are so-called assumed density filters \citep{Kushner1967b} which include sigma-point filters such as unscented Kalman filter \citep{Julier1997} and other Gaussian approximations of the optimal filter \citep{Ito2000, Saerkkae2013}. These filters can be seen as special cases of projection filters \citep{Brigo1999}. For comprehensive surveys of optimal filtering, we refer the reader to \citet{Bain2009,Crisan2011,Sarkka2013,candy2016}.

The contribution of this paper is to introduce an alternative way to implement projection filters via polynomial expansions and automatic differentiation. Automatic differentiation is a method for determining the exact derivatives of a function (to machine precision) by tracing basic mathematical operations involved in the function \citep{Rall1981,Griewank2008}. It has a computing cost that is proportional to the cost of function evaluation, depending on whether forward, backward, or hybrid accumulations are utilized, as well as on the dimensionality of the Jacobian. In particular, we use a Chebyshev polynomial expansions for unidimensional problems and sparse-grid integrations for higher dimensions in order to efficiently compute the log-partition function. The numerical integration is only used for computing the log-partition function, and the rest of the computations in the filter are computed via automatic differentiation, which significantly simplifies the filter algorithms and eliminates the repetitive numerical integrations for calculating the expected values of the statistics of the exponential family at hand. In order to do so, we use a smooth bijection from the state sample space $\mathcal{X}$ to the canonical open hypercube $(-1,1)^{\otimes d}$ ($\mathcal{X}$ is assumed to be fixed, see Section \ref{sec:log_partition_function_1d}). This setting ensures that the integrations can be performed in a fixed domain. To demonstrate the use of the projection filter, we focus on a class of models, where the coefficients of the dynamic model SDE can be represented by polynomials, the natural statistics are monomials of the states, and the drift coefficient of the observation model SDE is in the span of the natural statistics. In all, our method allows for an automatic implementation of projection filters. That is to say, one only needs to specify the model (symbolically) and a bijection function to carry out the filter computations -- the manual derivation for the filter equation is no more needed. %

Using sparse-grid techniques in optimal filtering is not entirely novel. For instance, there are Gaussian approximation-based filters that use sparse-grid quadratures \citep{Winshcel2010,Jia2012,Baek2013,Radhakrishnan2016,Singh2018}. Furthermore, \citet{Bao2014} employ an adaptive sparse-grid and a finite difference approach to solve the Zakai equation. We propose an alternative approximation for nonlinear optimal filtering problems compared to \citet{Bao2014}. Specifically, we employ a sparse-grid method to calculate the log-partition function that is used in the projection filter approximation. Compared to \citet{Bao2014}, our approach requires much less variables to be propagated at each time step. This is due to that we only propagate the natural parameters of the exponential family and not the quadrature points. Moreover, our approach requires the quadrature points to be calculated only once at the beginning of the computations, whereas the approach of \citet{Bao2014} requires the quadrature points to be calculated at each time step at various grid-refinement phases. Despite the fact that adaptive sparse grids may have fewer nodes while retaining excellent accuracy, the grid-refinement technique is problem-dependent and requires additional computing steps to generate new quadrature points. These procedures involve repeated evaluations of the integrand \citep{Bungartz2004,Ma2009}. Introducing these additional calculations for every time step incurs a substantial computational expense. Similarly to our approach, \citet{Bao2014} employ a hypercube as the sparse-grid domain. Nevertheless, the placement of the hypercube is adjusted each time using a heuristic technique based on a finite number of samples derived from calculating the quadrature point trajectory using the state equations. Our approach is free of these extra computations.
 
Even if we can compute the moments and Fisher metric perfectly, the projection filter still has local projection errors. Fortunately, the estimates for these local projection errors are well-established for certain classes of state-space models and parametric families (with different natural statistics) \citep[Propositions 6.2 and 6.4]{Brigo1999}. We believe that our approximation will be useful in problems where obtaining approximate solutions to Kushner--Stratonovich equations with particle filters is deemed too costly, numerically unstable, or time-consuming, and where the Gaussian approximation is not adequate. 

The article is organized as follows. In Section~\ref{sec:SimplificationOfExponentialFamily}, we present the problem formulation. In particular, we define the class of state-space models (i.e., the SDEs of the dynamic and observation models) for which the proposed projection filter can be constructed. Section \ref{sec:log_partition_function_1d} describes how to approximate the log-partition function via polynomial interpolation and automatic differentiation for unidimensional problems. Section \ref{sec:log_partition_function_nd} describes the use of sparse-grid methods to compute the log-partition function in multidimensional cases. Section \ref{sec:Extension_to_dx_greater_than_1} describes methods to obtain the expected values of the natural statistics and the Fisher metric via automatic differentiation in multidimensional problems. In Section \ref{sec:NumericalExamples}, we conduct numerical experiments to show the efficiency of the proposed projection filter in both unidimensional and multidimensional filtering problems. Section \ref{sec:conclusion} concludes our findings\footnote{We have made the code for the approach described in this work accessible for download. The source code is available at \url{https://to_be_revealed}.}.

\section{Projection filter construction for an exponential family}\label{sec:SimplificationOfExponentialFamily}
In this paper, we consider optimal filtering problems consisting of the following dynamic and observation models:
\begin{subequations}
	\label{eqs:nonlinear_SDE}
	\begin{align}
		dx_t &= f(x_t,t) dt + \sigma(x_t,t) dW_t,\\
		dy_t &= h(x_t,t) dt + dV_t,
	\end{align}
\end{subequations}
where $x_t \in \mathbb{R}^{d_x}$ and $y_t\in \mathbb{R}^{d_y}$. The processes $\left\{ W_t \right\}_{t\geq0}, \left\{ V_t \right\}_{t\geq0}$ are Brownian motions taking values in $\mathbb{R}^{d_w}$ and $\mathbb{R}^{d_y}$ with invertible spectral density matrices $Q_t$ and $R_t$ for all $t\geq0$, respectively. We also define $\alpha(x_t,t) \coloneqq \sigma(x_t,t) Q_t \sigma(x_t,t)^\top$. For the sake of exposition, from now on, we assume that $R_t = I$ for all $t\geq0$. For any vector $v_t \in \mathbb{R}^n$ that is a function of time $t$, we denote the $k$-th element of $v_t$ by $v_{t,k}$. Similarly, for any matrix $X_t \in \mathbb{R}^{m\times n}$ that is a function of time $t$, we denote the $i,j$ entry of $X_t$ as $X_{t,(i,j)}$.

Let the $m$-dimensional exponential family EM$(c)$ be the set of probability densities of the form $p_\theta(x) = \exp(\sum_{i=1}^m c_i (x) \theta_i  - \psi(\theta)), x \in \mathcal{X}$, where $\theta \in \Theta \subset \mathbb{R}^m$ and $\mathcal{X}\subseteq \mathbb{R}^{d_x}$ are the natural parameters and the sample space of the state $x$, respectively. We also assume that the natural statistics $\left\{ c_i \right\}_{i=1,\ldots,m}$ are linearly independent. The log-partition function $\psi$ is given by
\begin{align}
	\psi(\theta) &= \log \int_{\mathcal{X}} \exp\left( \sum_{i=1}^m c_i(x) \theta_i \right) dx. \label{eq:psi}
\end{align}
The inclusion of the log-partition function $\psi$ inside the exponential ensures that the density $p_\theta$ is normalized. Moreover, by definition of the exponential family, $\Theta \subset \mathbb{R}^m$ is selected such that for any $\theta \in \Theta$, $\psi < \infty$  \citep[Section 1.4]{Calin2014}.

Under suitable regularity conditions to guarantee the well-definedness of the system (i.e., the solutions to the SDEs in Equation~\eqref{eqs:nonlinear_SDE} exist, and the projection from the space of probability densities to the finite-dimensional manifold exists), the projection filter using Stratonovich projection to finite-dimensional manifold is given by \citep{Brigo1999}:
\begin{align}
	d\theta_t &= g(\theta_t)^{-1}\E_{\theta_t}\left[ \mathcal{L}_t \left[c\right]  - \frac{1}{2}\abs{h_t}^2 \left[ c - \eta(\theta_t) \right] \right] dt \nonumber\\
	& \; + g(\theta_t)^{-1} \sum_{k=1}^{d_y} \E_{\theta_t}\left[ h_{t,k} \left[ c-\eta(\theta_t) \right] \right] \circ dy_{t,k} .\label{eq:Brigo_dTheta_EM}
\end{align}
In the equation above, $g$ is the Fisher metric corresponding to the probability density family, $\E_{\theta}$ is the expectation with respect to the parametric probability density $p_\theta$, $h_t:= h(\cdot, t)$, $\theta_t\mapsto\eta(\theta_t):=\E_{\theta_t}[c]$, $\mathcal{L}_t$ is the backward diffusion operator, and
$\circ$ denotes the Stratonovich multiplication. 

In what follows, we restrict the class of the dynamics in \eqref{eqs:nonlinear_SDE} with the following assumptions.
\begin{asm}\label{asm:f_h_sigma_is_polynomia}
	The functions $f, h,$ and $\sigma$ are polynomial functions of $x_t$. 
\end{asm}
\begin{asm}\label{asm:c_monomials}
	The natural statistics $\left\{ c_i \right\}$  are selected as monomials of the elements of $x_t$. %
\end{asm}
\begin{asm}\label{asm:EM_c_ast}
	Each component of $h_t$, that is, $h_{t,k}$ is in the span of $\left\{ 1, c_1, \ldots, c_m \right\}$. that is, there exists collection of real numbers $\big\lbrace\lambda^{(i)}_k \colon i=1,\ldots,m, \, k=1,\ldots,d_x\big\rbrace$ such that
	\begin{align*}
		h_{t,k} &= \lambda_k^{(0)} + \sum_{i=1}^m \lambda_k^{(i)} c_i,
	\end{align*}
	for all $k=1,\ldots,d_x$.
\end{asm}

Under Assumption \ref{asm:EM_c_ast}, the resulting exponential family of probability densities is known as the EM$(c^\ast)$ family defined in~\citet{Brigo1999}. For this specific family EM$(c^\ast)$, the projection filter equation reduces to \cite[Theorem 6.3]{Brigo1999}
\begin{align}
	d\theta_t &= g(\theta_t)^{-1}\E_{\theta_t}\left[ \mathcal{L}_t \left[c\right]  - \frac{1}{2}\abs{h_t}^2 \left[ c - \eta(\theta_t) \right] \right] dt + \sum_{k=1}^{d_y} \lambda_k  dy_{t,k}. \label{eq:Brigo_dTheta_EM_c_ast}
\end{align}

Furthermore, \eqref{eq:Brigo_dTheta_EM_c_ast} can be simplified further with Assumptions \ref{asm:f_h_sigma_is_polynomia} and~\ref{asm:c_monomials}. The idea is to express the terms in Equation \eqref{eq:Brigo_dTheta_EM_c_ast} as products of matrices and expectations of statistics under these assumptions. Specifically, we obtain $\mathcal{L}_t[c] - \frac{1}{2}|h_t|^2c = a_0 + A_0 \tilde{c}$ and $\frac{1}{2}|h_t|^2 = b_0 + b_h^\top\tilde{c}$, where $a_0 \in \mathbb{R}^{m}$, $A_0 \in \mathbb{R}^{m\times(m+m_h)}$,  $b_0 \in \mathbb{R}$,  $b_h \in \mathbb{R}^{(m+m_h)}$, and $\tilde{c}^\top = [c^\top, c_h^\top]$. Note that $c_h:\mathbb{R}^{d_x}\rightarrow \mathbb{R}^{m_h}$ is the vector of the remaining monomials of $x_t$ that are not included in $c$. Therefore, we can express \eqref{eq:Brigo_dTheta_EM_c_ast} as
\begin{align}
	d\theta_t = g(\theta_t)^{-1}\left[a_0 + b_0 \eta(\theta_t) + M(\theta_t) \mathbb{E}_{\theta_t}[\tilde{c}] \right]dt + \lambda dy_t, 
\end{align}
where the matrix of the polynomial coefficients $M$ is given by %
\begin{align}
	M &= A_0 + \eta(\theta_t) b_h^\top.
\end{align}
Denoting $\tilde{\eta}:= \mathbb{E}_{\theta_t}[\tilde{c}]$ we arrive at
\begin{align}
	d\theta_t = g(\theta_t)^{-1}\left[ a_0 + b_0 \eta(\theta_t) + M(\theta_t) \tilde{\eta}(\theta_t)\right] dt + \lambda dy_t.\label{eq:Brigo_dTheta_EM_c_ast_polynomials}
\end{align}
Equation \eqref{eq:Brigo_dTheta_EM_c_ast_polynomials} can also be used with other exponential families where the natural statistics are not necessarily monomials. For instance, \eqref{eq:Brigo_dTheta_EM_c_ast_polynomials} can be used in the case where the natural statistics are allowed to have negative powers of $x_t$. However, the automation of implementation for the polynomials is much easier because the polynomials are closed under addition, multiplication, differentiation, and integration. During the implementation, we use SymPy \citep{Meurer2017}, which is a symbolic mathematical package for Python programming language, to automatically calculate the matrices and vectors $a_0$, $b_0$, and $A_0$ in (6) and (7).

In order to solve Equation \eqref{eq:Brigo_dTheta_EM_c_ast_polynomials}, we need to compute the expectation $\tilde{\eta}$ and the Fisher metric $g$. These quantities can be computed from the log-partition function. Specifically, \citet{Amari1985,Calin2014,Brigo1999} show that

\begin{subequations}
	\begin{align}
		\eta_i = \partial_i\psi \quad \text{and} \quad g_{ij}(\theta) = \partial_i \partial_j \psi \label{eq:eta} %
	\end{align}    
\end{subequations}
hold for all $i,j=1,\ldots,d_x$. In the equation above, we denote $\partial_i\psi:=\pdv{\psi}{\theta_i}$ which is a commonly used shorthand in information geometry. Provided that the log-partition function $\psi$ is evaluated numerically, evaluating $\eta$ and $g$ is straightforward with any numerical package that supports automatic differentiation. We will show in the subsequent sections that the approximation error arising from these computations is predictable. Furthermore, for the remaining statistics $c_h$, a simple trick using automatic differentiation can be used to calculate their expectations. We will describe this in Section \ref{sec:Extension_to_dx_greater_than_1}.

\section{Calculating log-partition function in unidimensional case}
\label{sec:log_partition_function_1d}
The previous discussion shows that under Assumptions \ref{asm:f_h_sigma_is_polynomia}, \ref{asm:c_monomials}, and \ref{asm:EM_c_ast}, after re-arranging the filter equation to \eqref{eq:Brigo_dTheta_EM_c_ast_polynomials}, to construct the projection filter, it remains to calculate $\psi$. In this section, we will describe how to compute $\psi$ efficiently in the unidimensional case ($d_x=1$), while the next section covers the multidimensional case. We will start by defining a bijection that will be used in the numerical integration of the log-partition function \eqref{eq:psi}. %
In the most general setting, the support of the parametric probability $p(\cdot, \theta)$ depends on $\theta$. However, \citet{Amari1985} pointed out that this condition poses a substantial difficulty to analysis. Therefore, following \citet[Section 2.1]{Amari1985}, we assume that the support of $p(\cdot,\theta)$ is uniform for all $\theta \in \Theta$. For notational simplicity let us set $\Omega = (-1,1)$, and denote the class of bounded variation functions on $\Omega$ as $BV(\bar{\Omega})$, where $\bar{\Omega}$ is the closure of $\Omega$.

Equation \eqref{eq:psi} can be written as
\begin{subequations}
    \begin{align}
        \psi(\theta) &= \log \int_\Omega z(\tilde{x}) d\tilde{x}, \label{eq:psi_in_tilde} \\
        z(\tilde{x}) &= \phi'(\tilde{x}) \exp\left( \sum_i c_i(\phi(\tilde{x})) \theta_i \right).\label{eq:z_in_tilde}
    \end{align}
\end{subequations}
Let $L^2_w(\Omega)$ be the set of square integrable functions on $\Omega$ with respect to a positive weight function $w(\tilde{x}) = (1-\tilde{x})^{-1/2}$. Let us then define inner product by $\expval{f, g}:= \int_\Omega w(\tilde{x})f(\tilde{x})g(\tilde{x})d\tilde{x}$. The space $L^2_w(\Omega)$ equipped with the inner product is a Hilbert space, and the Chebyshev polynomials are a complete orthogonal basis of $L_w^2(\Omega)$. Assuming that $z\in L^2_w(\Omega)$, we can expand $z(\tilde{x})$ on the Chebyshev polynomial basis:  
    \begin{align}
        z(\tilde{x}) &= \sum_{i=0}^\infty \hat{z}_i(\theta) T_i(\tilde{x})  \nonumber,\\
        \hat{z}_i(\theta) &= \expval{T_i, z}. \label{eq:Chebyshev_z}
    \end{align}
    In this equation, $T_i$ is the $i$-th Chebyshev polynomial of the first kind, and $\hat{z}_i(\theta)$ is the corresponding Chebyshev coefficient. Since the Chebysev series is related to the Fourier series by substituting $\tilde{x}=\cos(\omega)$, $\hat{z}_i(\theta)$ can be numerically computed using the fast Fourier Transform (FFT).  
    Let $S_n^T f := \sum_{j=0}^n \expval{T_j, f} T_j$ be the truncated Chebyshev sum of $f$ up to $n$ first terms. Moreover, let $\psi(\theta)^{(n)}$ be the approximation of $\psi$ where $z$ is replaced by $S_n^T z$:
    \begin{align}
        \psi(\theta)^{(n)} &:=  \log \left( \sum_{j=0}^n \int_\Omega \hat{z}_j(\theta) T_j(\tilde{x})d\tilde{x} \right), \label{eq:psi_in_z_hat} \\
        &= \log \left( \sum_{j=0}^n \hat{z}_j(\theta) \bar{T}_j \right) \nonumber ,\\
        \bar{T}_j &:=\int_\Omega T_j(\tilde{x}) d\tilde{x} = 
        \begin{cases}
            \frac{-1^j + 1}{1- j^2} &, j \neq 0,\\
            0 &, j = 0.
        \end{cases}
    \end{align} 
    In the following proposition, we show that the approximated log-partition function $\psi(\theta)^{(n)}$ converges to $\psi$.
\begin{prop}\label{prp:psi_N_converges}
    Suppose that the natural statistics $\left\{ c_i \right\}$ and the bijection $\phi$ are chosen so that the $z$ in \eqref{eq:z_in_tilde} belongs to $L^2_w(\Omega)$. Then, the approximated log-partition $\psi(\theta)^{(n)}$ converges to $\psi(\theta)$ as $n$ approaches infinity.
\end{prop}
\begin{proof}
    From the definition of $\psi$ and $\psi^{(n)}$ we obtain:
    \begin{align*}
        \abs{\psi(\theta)- \psi(\theta)^{(n)}} &= \abs{\log \left[ \dfrac{\int_\Omega z d\tilde{x}}{\int_\Omega S_n^T z d\tilde{x}} \right]}\\
        &= \abs{\log \left[ \dfrac{\int_\Omega z-S_n^T z d\tilde{x} + \int_\Omega S_n^T z d\tilde{x} }{\int_\Omega S_n^T z d\tilde{x}} \right]}\\
        &\leq \abs{\log \left[ \dfrac{\int_\Omega\abs{z-S_n^T z} d\tilde{x} + \int_\Omega S_n^T z d\tilde{x} }{\int_\Omega S_n^T z d\tilde{x}} \right]}\\
        &\leq \abs{\log \left[  \dfrac{\left( \int_\Omega 1^2 d\tilde{x} \right)^{1/2} \left( \int_\Omega\abs{z-S_n^T z}^2 d\tilde{x}\right)^{1/2} + \int_\Omega S_n^T z d\tilde{x} }{\int_\Omega S_n^T z d\tilde{x}} \right]}\\
        &\leq \abs{\log \left[  \dfrac{ \sqrt{2} \left( \int_\Omega\abs{z-S_n^T z}^2 w(\tilde{x}) d\tilde{x}\right)^{1/2} + \int_\Omega S_n^T z d\tilde{x} }{\int_\Omega S_n^T z d\tilde{x}} \right]}\\
        &= \abs{ \log \left[  \dfrac{ \sqrt{2} \norm{z-S^T_n z}_2 + \int_\Omega S_n^T z d\tilde{x} }{\int_\Omega S_n^T z d\tilde{x}} \right] },
    \end{align*}
    where $\norm{\cdot}_2$ is $L^2_w \left( \Omega \right)$ norm. The result follows immediately. \qed
\end{proof}
\begin{rem}
    Proposition \ref{prp:psi_N_converges} shows that approximation \eqref{eq:psi} converges. Indeed, since $z$ in \eqref{eq:z_in_tilde} is likely to be several times continuously differentiable or even analytic, even a stronger convergence result can be expected since the Chebyshev approximation error $\abs{f- S_n^T f}$ has $\mathcal{O}(r^{-n})$ rate when $f$ is analytic for some $r>1$ \citep[Theorem 5.16]{Mason2003}. In this case $\abs{\psi(\theta)- \psi(\theta)^{(n)}}$ also decays exponentially as $\mathcal{O}(r^{-n})$.  
\end{rem}

    In order to use the Chebyshev sum for approximating $\psi$, it is inevitable to compute its inner products. This can be done either by calculating the inner products or the FFT, at the cost of complicated derivations and demanding computations. Fortunately, by using the orthogonal interpolation approach \citep{Mason2003} for approximating $z$, computing the inner products is no longer needed. This is can be done follows. An interpolant of a function $f$ with polynomial of order $n-1$ is a polynomial $p_n(x) = J_{n-1}f (x)$, such that when it is evaluated at the $n$ distinct interpolation points $\left\{ x_i \right\}$ it has exactly the same value as $f(x_i)$. That is, for any $i = 1,\ldots, n$:
    \begin{align*}
        p_n(x_i) = c_0 + c_1 x_i + \cdots + c_n x_i^{n-1} &= f(x_i).
    \end{align*}
    The resulting polynomial can be written in terms of the Lagrange polynomials \citep[Lemma 6.3]{Mason2003} as 
    \begin{align*}
        J_{n-1} f(x) &= \sum_{i=1}^{n}\ell_i(x)f(x_i).
    \end{align*}
    When the interpolation points are selected to be the zeros of the Chebyshev polynomial of the first kind, $\ell_i$ is given by \citep[Corollary 6.4 A]{Mason2003}
    \begin{align}
        \ell_i(x) &= \dfrac{T_{n}(x)}{(n+1)(x-x_i)U_{n-1}(x_i)},
    \end{align}
    where $U_n$ is the Chebyshev polynomial of the second kind  of order $n$.
    When approximating $f$ by its $n$-th order Chebyshev interpolant $J_{n-1} f$, we can approximate the definite integral $\int_\Omega f w dx$ by what is known as Gauss--Chebyshev quadrature (of the first kind) \citep[Theorem 8.4]{Mason2003}:
    \begin{align}
        \int_\Omega f(\tilde{x}) w(\tilde{x}) d\tilde{x} &\approx Q^1_n f, \nonumber \\
        Q^1_n f &:= \int_\Omega J_{n-1} f(\tilde{x}) w(\tilde{x}) d\tilde{x} =  \sum_{i=1}^n \dfrac{\pi}{n}f(\tilde{x}_i),\nonumber \\
        w(\tilde{x}) &= (1-\tilde{x}^2)^{-1/2}, \nonumber\\
        \tilde{x}_i &= \cos\left( \dfrac{(i-\frac{1}{2})\pi}{n} \right). \label{eq:Gauss_Chebyshev_quadrature}
    \end{align}
    This approximation is based on the continuous orthogonality of the Chebyshev polynomials (that is if $i\neq j$, $\langle T_i,T_j\rangle = 0$)  and is exact if $f$ is a polynomial of degree $2n-1$ or less \citep[Theorem 8.2]{Mason2003}. It is known that for a function $f \in BV(\bar{\Omega})$, the integration error satisfies $E_n f \leq \pi V_{\bar{\Omega}}(f)/(2n)$ where
    \begin{align}
        E_n f &:= \abs{\int_\Omega f(\tilde{x}) w(\tilde{x}) d\tilde{x} - Q^1_n f}, \label{eq:integration_error}
    \end{align} 
    and $V_{\bar{\Omega}}(f)$ stands for the total variation of $f$ on $\bar{\Omega}$ \citep{Riess1971,Chui1972}. 
    
    To compute the integral of $z(\tilde{x})$ given in \eqref{eq:z_in_tilde}, we will apply the quadrature rule \eqref{eq:Gauss_Chebyshev_quadrature}. We must be cautious, though, because not every smooth bijection $\phi$ used to compute the integral of $z$ on $\Omega$ will be permissible. In fact, for some $\theta>0$, let $p_\theta = \frac{1}{\sqrt{2\pi}} \exp(- (x-\theta)^2/2)$. If we choose $\phi = \sqrt{2} \erf(\tilde{x})^{-1}$, then $z(\tilde{x}) = \frac{1}{2}\exp(\frac{2\phi(\tilde{x})\theta-\theta^2}{2})$, which tends to infinity as  $\tilde{x} \to 1$, meaning that $w^{-1} z$ may not be in $BV(\bar{\Omega})$, resulting in numerical integration failure. We offer conditions in the following lemma that may be used to verify that the numerical quadrature \eqref{eq:Gauss_Chebyshev_quadrature} produces a consistent integration result, that is, the error decreases with increasing $n$. Furthermore, because the domain of the bijection $\phi$ is $\Omega$, we assign $z(1) = \lim_{\tilde{x}\uparrow 1 } z(\tilde{x})$, and $z(-1) = \lim_{\tilde{x}\downarrow -1 } z(\tilde{x})$.

    In the following, we present the main results of this section that give upper-bound estimates of the approximation errors of the log-partition functions, and the first and second moments of $c(x)$.  

    \begin{lem}\label{lem:boundedVariation}
        Let $\phi: \Omega \rightarrow \mathbb{R}$ be a smooth bijection, $w(\tilde{x}) = (1-\tilde{x}^2)^{-1/2}$, and Assumption \ref{asm:c_monomials} be satisfied. Let also $\zeta := \phi^{-1}: \mathbb{R}\rightarrow \Omega$.
        If for any $\theta \in \Theta$,
        \begin{subequations}
            \begin{align}
                \E_\theta \left[ \dfrac{1}{\zeta'(x)^2} \right] &< \infty \label{eq:lemma_requirement_zeta},\\  \E_\theta \left[ \dfrac{\zeta''(x)^2}{\zeta'(x)^4} \right] &< \infty \label{eq:lemma_requirement_zeta_ratio}, 
            \end{align}    
        \end{subequations}
        then $r$,$s_k$, and $t_{l,k}$, defined by
        \begin{align*}
            r(\tilde{x}) &:= z(\tilde{x}) w(\tilde{x})^{-1}, &
            s_k(\tilde{x}) &:= c_k(\phi(\tilde{x})) r(\tilde{x}) ,  & 
            t_{l,k}(\tilde{x}) &:= c_l(\phi(\tilde{x})) s_k(\tilde{x}),
        \end{align*} 
        belong to $BV(\bar{\Omega})$. 
        \end{lem}
        \begin{proof}
            First we will prove that $r\in BV(\bar{\Omega})$. To see this, it is enough to show that $z\in BV(\bar{\Omega})$ since $w^{-1} \in BV(\bar{\Omega})$ and $BV(\bar{\Omega})$ is an algebra. 
        
            Before showing that $z\in BV(\bar{\Omega})$, we claim that for $\theta \in \Theta$, $\E_\theta[\abs{x^k}]< \infty, \forall k \in \mathbb{N}$ when Assumption \ref{asm:c_monomials} holds. To see this, we observe that the moment generating function $\varphi$ with respect to the natural statistics $\left\{ c_i \right\}$ is given by
            \begin{align*}
                \varphi &:= \E_\theta \left[ \exp(v^\top c(x)) \right] \\
                &= \frac{1}{\exp(\psi)}\int_{\mathbb{R}} \exp(c^\top(x)\theta) \exp(v^\top c(x)) dx \\
                &= \exp(\psi(\theta+v) - \psi(\theta)).
            \end{align*}
            Therefore, $\E_\theta[c_i^k(x)] = \left.\pdv[k]{\varphi}{v_i}\right|_{v=0}$ since $\psi$ is infinitely differentiable on $\Theta$ \citep[Lemma 2.3]{Brigo1995}. Furthermore, since $c_i$ is a monomial, we have proven the claim.
            
            Since $z$ is differentiable on $\Omega$, then the total variation of $z$ is given by $\int_\Omega \abs{z'}d\tilde{x}$. Let $\beta(x) = \exp(\sum_{i=1}^m c_i(x)\theta_i)$. Without loss of generality, assume that $\phi$ is monotonically increasing, otherwise, choose the bijection as $-\phi$. Since $\phi$ is monotonically increasing, the total variation of $z$ is given by
            \begin{align*}
                V_{\bar{\Omega}}(z) &= \int_\Omega \abs{z'}d\tilde{x} = \int_\Omega \abs{\left(  \phi'(\tilde{x}) \beta(\phi(\tilde{x}))\right)'}d\tilde{x}\\
                &\leq  \int_\Omega \abs{
                     \dfrac{\phi''(\tilde{x})}{\phi'(\tilde{x})} 
                } z(\tilde{x})d\tilde{x}
                + 
                \int_\Omega \sum_{i=1}^m \abs{\phi'(\tilde{x}) c'_i(x)\theta_i} z(\tilde{x}) d\tilde{x}.
            \end{align*}
            Using $x=\phi(\tilde{x})$ and 
            \begin{align*}
                \zeta'(x) &= \frac{1}{\phi'(\tilde{x})}, & \dfrac{\phi''(\tilde{x})}{\phi'(\tilde{x})} &= -\dfrac{\zeta''(x)}{\zeta'(x)^2},
            \end{align*}
            and by introducing a positive measure on $\mathbb{R}$, $d\mu(x) = \beta(x)dx = z(\tilde{x})d\tilde{x}$, we can further simplify
            \begin{align*}
                \int_\Omega \abs{
                     \dfrac{\phi''(\tilde{x})}{\phi'(\tilde{x})} 
                } z(\tilde{x})d\tilde{x}
                &+
                \int_\Omega \sum_{i=1}^m \abs{\phi'(\tilde{x}) c'_i(x)\theta_i} z(\tilde{x})d\tilde{x}\\
                &= \int_\mathbb{R} \abs{
                    \dfrac{\zeta''(x)}{\zeta'(x)^2} 
               } d\mu
               +
               \int_\mathbb{R} \sum_{i=1}^m \abs{\frac{1}{\zeta'(x)} c'_i(x)\theta_i} d\mu.
            \end{align*}
            By Jensen inequality and \eqref{eq:lemma_requirement_zeta_ratio}, the first part of the righthand side is finite,
            \begin{align*}
                \int_\mathbb{R} \abs{
                    \dfrac{\zeta''(x)}{\zeta'(x)^2} 
               } d\mu = \exp(\psi) \E_\theta \left[ \abs{ \dfrac{\zeta''(x)}{\zeta'(x)^2}} \right]  < \infty. 
            \end{align*}
            The second part is also finite, since
            \begin{align*}
                \int_\mathbb{R} \sum_{i=1}^m \abs{\frac{1}{\zeta'(x)} c'_i(x)\theta_i} d\mu &= \exp(\psi) \E_\theta \left[ \sum_{i=1}^m \abs{\frac{1}{\zeta'(x)} c'_i(x)\theta_i} \right] \\
                &\leq \exp(\psi) \sum_{i=1}^m \E_\theta \abs{\frac{1}{\zeta'(x)} c'_i(x)\theta_i} \\
                &\leq \exp(\psi) \sum_{i=1}^m \left[ \E_\theta \left[ \frac{1}{\zeta'(x)^2} \right]  \right]^{1/2} \left[ \E_\theta \left[ c'_i(x)^2 \right] \right]^{1/2}\abs{\theta_i} 
            \end{align*}
            which is less than infinity by \eqref{eq:lemma_requirement_zeta}, Assumption \ref{asm:c_monomials}, and since any moments of $x$ are finite.
            To show that $s_k \in BV(\bar{\Omega})$ it is enough to check that $c_k(\phi) z$ is of bounded variation on $\Omega$. As before, we get
            \begin{align*}
                V_{\bar{\Omega}}(c_k(\phi) z ) &= \int_\Omega \abs{ \left( c_k(\phi(\tilde{x})) z(\tilde{x}) \right)' }d\tilde{x} \nonumber\\
                &\leq \int_\mathbb{R} \abs{\dfrac{c_k'(x)}{\zeta'(x)}} d\mu
                + \int_\mathbb{R} \abs{\dfrac{c_k(x)\zeta''(x)}{\zeta'(x)^2}}d\mu 
                + \int_\mathbb{R} \sum_{i=1}^m \abs{\dfrac{c_k(x)c_i'(x)}{\zeta'(x)}\theta_i}d\mu.
            \end{align*}
            The first and last terms of the right hand side of the last equation are finite because of condition \eqref{eq:lemma_requirement_zeta} and since any moment of $x$ is finite. The second part follows from H\"older's inequality and from condition \eqref{eq:lemma_requirement_zeta_ratio}.
            
            By Assumption \ref{asm:c_monomials}, there exists $q\in \mathbb{N}$ such that $c_l(x)c_k(x)=x^q$. Therefore, the result for $t_{l,k}\in BV(\bar{\Omega})$ follows similarly as $s_k \in BV(\bar{\Omega})$, which completes the proof. \qed
        \end{proof}
        
        Lemma~\ref{lem:boundedVariation} shows that it is sufficient to have the conditions~\eqref{eq:lemma_requirement_zeta} and~\eqref{eq:lemma_requirement_zeta_ratio} hold in order to have $r, s_k, t_{l,k} \in BV(\bar{\Omega})$.  Consider two bijections $\phi=\tan(\frac{\pi}{2}\tilde{x})$ and $\phi = \frac{\tilde{x}}{1-\tilde{x}^2}$. Since their  corresponding $\abs{\frac{1}{\zeta'(x)^2}}$ and $\abs{\frac{\zeta''(x)}{\zeta'(x)^2}}$ are less than absolute values of some polynomials, the conditions ~\eqref{eq:lemma_requirement_zeta} and ~\eqref{eq:lemma_requirement_zeta_ratio} hold for any density from the exponential families with monomial natural statistics. On the other hand, the bijection $\phi = \tanh^{-1}(\tilde{x})$ under $p_\theta=\frac{1}{\sqrt{2\pi}} \exp(- (x-\theta)^2/2)$ satisfies these conditions for $\theta \in (-\pi,\pi)$. For related discussion on the effect of bijection in the context of sparse-grid integration, see \citet{Griebel2014}. 
        
        As a direct consequence of Lemma \ref{lem:boundedVariation} and results from \citet{Chui1972}, we have the following upper bounds on the integration errors of $r, s_i$ and $t_{l,k}$.

        \begin{cor}
            The integration errors of $r, s_i$ and $t_{l,k}$ for $i,l,k=1,\ldots,m$ satisfy 
            \begin{subequations}
                \begin{align}
                E_n r &\leq \pi V_{\bar{\Omega}}(r)/(2n),\\
                E_n s_i &\leq \pi V_{\bar{\Omega}}(s_i)/(2n),\\
                E_n t_{l,k} &\leq \pi V_{\bar{\Omega}}(t_{l,k})/(2n),
            \end{align}\label{eqs:integration_error_bound_rst}
            \end{subequations}
            respectively.
        \end{cor}
        We now present the main result for error bounds of $\psi(\theta)^{(n)}$, $\partial_i \psi(\theta)^{(n)}$, and $\partial_i\partial_j \psi(\theta)^{(n)}$.

    \begin{thm}\label{thm:error_estimation}
        Let Assumption \ref{asm:c_monomials} be satisfied. Let $\phi$ be a smooth bijection satisfying the conditions in Lemma \ref{lem:boundedVariation}, and $r, \left\{ s_k \right\}, \left\{ t_{l,k} \right\}$ be functions defined as in Lemma \ref{lem:boundedVariation} whose
         integrations errors are given by \eqref{eqs:integration_error_bound_rst}.
        Denote by $\psi(\theta)^{(n)}$ the approximation of log-partition function $\psi$ where $\int_\Omega z d\tilde{x}$ in \eqref{eq:psi_in_tilde} is replaced by $Q_1^n r$. The log-partition function approximation error satisfies, 
        \begin{align*}
            \abs{\psi-\psi^{(n)}} &\leq \dfrac{1}{Q^1_n r} E_n r,
        \end{align*}
        where $Q^1_n r>0$. Furthermore, the approximation errors of the first and the second derivatives of $\psi$ have the following upper bounds:
            \begin{align*}
                \abs{\partial_i\psi-\partial_i\psi^{(n)}}&\leq \dfrac{1}{\exp(\psi)} E_n s_i + \abs{ \dfrac{Q^1_n s_i}{\exp(\psi)Q^1_n r }} E_n r ,\\
                \abs{\partial_j\partial_i\psi(\theta)- \partial_j\partial_i\psi(\theta)^{(n)}} &\leq 
                \dfrac{1}{\exp(\psi)} E_n t_{j,i} + \abs{ \dfrac{Q^1_n t_{j,i}}{\exp(\psi)Q^1_n r }} E_n r  \nonumber\\ 
            &+ \dfrac{1}{\exp(\psi)} E_n s_j + \abs{ \dfrac{Q^1_n s_j}{\exp(\psi)Q^1_n r }} E_n r.
            \end{align*}
    \end{thm}
    \begin{proof}
        Without loss of generality, we can again assume that $\phi$ is monotonically increasing (otherwise choose $-\phi$). Observe that on $\Omega$, $r(\tilde{x}) > 0$ since the bijection $\phi$ is monotonically increasing. Therefore, by \eqref{eq:Gauss_Chebyshev_quadrature}, we have $Q^1_n r > 0$. Furthermore, as in Proposition~\ref{prp:psi_N_converges}, we have
        \begin{align*}
            \abs{\psi(\theta)- \psi(\theta)^{(n)}} &= \abs{\log \left[ \dfrac{\int_\Omega z(\tilde{x}) d\tilde{x}}{Q^1_n r} \right]} =\abs{\log \left[ \dfrac{\int_\Omega z(\tilde{x}) d\tilde{x} - Q^1_n r}{Q^1_n r} + 1\right]} \\ &\leq \log \left[ \dfrac{E_n r}{Q^1_n r} + 1\right] 
            \leq  \dfrac{E_n r}{Q^1_n r}.
        \end{align*}
        For the error $\abs{\partial_i\psi-\partial_i\psi^{(n)}}$, we have
        \begin{align*}
            \abs{\partial_i\psi(\theta)- \partial_i\psi(\theta)^{(n)}} &= \abs{\partial_i \left[ \log\int_\Omega z(\tilde{x}) d\tilde{x} - \log Q^1_n r \right]} \\
            &= \abs{ \dfrac{\int_\Omega c_i(\phi(\tilde{x})) z(\tilde{x}) d\tilde{x}}{\int_\Omega z(\tilde{x}) d\tilde{x}}  - \dfrac{Q^1_n s_i}{ Q^1_n r} } \\
            &\leq \abs{  \dfrac{E_n s_i}{\int_\Omega z(\tilde{x}) d\tilde{x}} + \dfrac{Q^1_n s_i}{\int_\Omega z(\tilde{x}) d\tilde{x}}  - \dfrac{Q^1_n s_i}{ Q^1_n r} } \\
            &\leq \abs{  \dfrac{E_n s_i}{\int_\Omega z d\tilde{x}} + Q^1_n s_i \left[ \dfrac{Q^1_n r - \int_\Omega z(\tilde{x}) d\tilde{x}  }{Q^1_n r \int_\Omega z d\tilde{x}}  \right]} \\
            &\leq   \dfrac{1}{\exp(\psi)} E_n s_i + \abs{ \dfrac{Q^1_n s_i}{\exp(\psi)Q^1_n r }} E_n r.
        \end{align*}
        Due to the fact that
        \begin{align*}
            \partial_j \partial_i \psi &= \dfrac{\int_\Omega c_i(\phi(\tilde{x})) c_j(\phi(\tilde{x})) z(\tilde{x}) d\tilde{x}}{\exp(\psi)} - \dfrac{\int_\Omega c_j(\phi(\tilde{x})) z(\tilde{x}) d\tilde{x}}{\exp(\psi)^2},
        \end{align*}
        using a similar technique as for the second part above, we get
        \begin{align*}
            \abs{\partial_j\partial_i\psi(\theta)- \partial_j\partial_i\psi(\theta)^{(n)}} &\leq 
            \dfrac{1}{\exp(\psi)} E_n t_{j,i} + \abs{ \dfrac{Q^1_n t_{j,i}}{\exp(\psi)Q^1_n r }} E_n r  \nonumber\\ 
            &+ \dfrac{1}{\exp(\psi)} E_n s_j + \abs{ \dfrac{Q^1_n s_j}{\exp(\psi)Q^1_n r }} E_n r.
        \end{align*}
        \qed
    \end{proof}

    In essence, Theorem \ref{thm:error_estimation} guarantees that the approximation errors of the log-partition function as well as the first and the second derivatives of $\psi$ have upper bounds that are linear with respect to the approximation error of $\int_\Omega z(\tilde{x}) d\tilde{x}$. This is reasonable since the partial derivative with respect to the natural parameter $\theta_i$ is a linear operator. 
    \begin{rem}
        It is possible to have the error bounds above satisfy even stronger decaying inequalities if $r,s_k,t_{l,k}$ are differentiable or twice differentiable on $\bar{\Omega}$, at the cost of imposing stronger conditions on the bijection $\phi$ \citep{Chui1972}.
    \end{rem}

    \section{Calculating log-partition function in multidimensional case}
    \label{sec:log_partition_function_nd}
    In this section, we will introduce a sparse-grid integration method to calculate the log-partition function for multidimensional problems. For detailed exposition of sparse-grid integration methods, we refer the reader to \citet{Gerstner1998}. 

    As we mentioned earlier in Section \ref{sec:log_partition_function_1d}, to compute the log-partition function, we can use polynomial interpolation of order $n$, which will give exact integration results if the integrand is a polynomial of degree $2n-1$.  One of the main issues that arises in the Gauss--Chebyshev quadrature is that the quadrature nodes of the polynomial of order $n-1$, $\Gamma_{n-1}^1$, are not a subset of the quadrature nodes of the polynomials of order $n$, $\Gamma_{n}^1$, that is, the nodes of Gauss--Chebyshev quadrature are not nested. This means that if we need to increase the accuracy of the integration, we cannot reuse $\Gamma_{n-1}^1$ and add the remaining missing nodes, but instead, the whole set of nodes needs to be computed from scratch. In multidimensional quadratures, having a quadrature rule with nested integration nodes is critical for computational efficiency.

    Gauss--Chebyshev quadrature relies on the continuous orthogonality of the Chebyshev polynomials. There is also a discrete orthogonality property of the Chebyshev polynomials \citep[Section 4.6]{Mason2003}. The  quadrature rule that takes benefit of this property is known as the Clenshaw--Curtis quadrature. This quadrature rule has a polynomial degree of exactness of $n-1$ \citep{Gerstner1998} and its nodes are nested. If the inner product weight is set to identity, the quadrature nodes become the zeros of the Legendre polynomials and the weights can be calculated by integrating the associated polynomials. This Gauss--Legendre quadrature is not nested. \citet{Kronrod1966} proposed an extension of $n$-point Gauss quadrature by other $n+1$ points so that the quadrature nodes are nested. This quadrature has the polynomial degree of exactness equal to $3n+1$ \citep{Laurie1997,Gerstner1998}. Solving the quadrature nodes for this scheme is, however, difficult, and \citet{Patterson1968} proposed an iterated version of the Kronrod scheme.

    For the case $d_x>1$, we can use extensions of aforementioned quadratures to $d$-dimensional hypercube integration \citep{Novak1996}, see also \citet{Barthelmann2000,Bungartz2004,Gerstner1998,Judd2011}. This is done via what is known as Smolyak's construction \citep{Smolyak1963}. In Smolyak's approach, it is sufficient to have the means for solving the integration problem for unidimensional problem efficiently, since the algorithm for multidimensional case is fully determined in terms of the algorithm of the unidimensional one \citep{Wasilkowski1995}. Consider a one-dimensional quadrature formula, where the univariate integrand $f$ is interpolated with an $n$ degree polynomial 
    \begin{equation*}
        Q_n^1 f := \sum_i^{N_n^1} w_{i,n} f(x_{i,n}).
    \end{equation*}
    The differences of the quadrature formulas of different degrees are given by
    \begin{equation*}
        \Delta_{n}^1 f := ( Q_n^1 - Q_{n-1}^1) f,
    \end{equation*}
    where we define $Q_0^1 f = 0$. The differences are quadrature formulas defined on the union of the quadrature nodes $\Gamma_n^1 \bigcup \Gamma_{n-1}^1$, where for the case of the nested nodes, $\Gamma_n^1 = \Gamma_n^1 \bigcup \Gamma_{n-1}^1$. The Smolyak's construction for $d$-dimensional functions $f$ for $n\in \mathbb{N}$ and $\mathbf{k}\in \mathbb{N}^d$ is given by \citep{Gerstner1998, Delvos1982}
    \begin{align*}
        Q^d_n f &:= \sum_{\norm{\mathbf{k}}_1 \leq n + d - 1} \left( \Delta_{k_1}^1 \otimes \cdots  \otimes \Delta_{k_d}^1 \right) f.%
    \end{align*}
    Let us denote the number of quadrature points in the unidimensional polynomial interpolation of order $n$ as $N_n^1$. An important feature of Smolyak's construction is that instead of having ${(N_n^1)}^d$ quadrature points as the case with the tensor product of unidimensional grids, it has far fewer quadrature points to be evaluated. If $N_n^1 = \mathcal{O}(2^n)$, under Smolyak's construction, the order of $N_n^d$ is $\mathcal{O}(2^n n^{d-1})$, which is substantially lower than $\mathcal{O}(2^{nd})$ resulting from the product rules \citep{Novak1996,Gerstner1998}.

    As for the error bound, for a function $f \in \mathcal{C}^r$, we have $E_n^1 f = \mathcal{O}((N_n^1)^{-r})$, where $E_n^d f := \abs{\int_\Omega f dx - Q^d_n f}$. This bound holds for all interpolation-based quadrature formulas with positive weights, such as the Clenshaw--Curtis, Gauss--Patterson, and Gauss--Legendre schemes. When we use these quadrature formulas as the unidimensional basis, if $f \in \mathcal{W}^r_d$, where
    \begin{equation*}
        \mathcal{W}^r_d := \left\{ f: \Omega \rightarrow \mathbb{R}: \norm{\dfrac{\partial^{\norm{\mathbf{s}}_1}f}{\partial x_1^{s_1} \ldots \partial x_d^{s_d}}}_{\infty} < \infty, \norm{\mathbf{s}}_1\leq r \right\},
    \end{equation*}
    then we have $E_n^d f = \mathcal{O}(2^{-nr} n^{(d-1)(r+1)})$ \citep{Wasilkowski1995,Gerstner1998}. We may establish the error estimate for the multidimensional case using a similar approach used in the proof of Proposition \ref{prp:ExpectationByExtension}, but with caution, because the equivalent of $z$ \eqref{eq:z_in_tilde} for the multidimensional case may likewise have an unbounded total variation.

    Specifically, for the integration problem \eqref{eq:psi_in_tilde}, the density $p_\theta(x)$ at the points on the boundary of the canonical hypercube is almost zero since $\exp(\sum c_i(x)\theta_i)$ in \eqref{eq:psi} goes to zero as $\norm{x}$ goes to infinity. This implies that the grid points that lie on the boundary of the canonical cube can be neglected in the quadrature calculation. This also reduces the required number of grid points.

\section{Extending the projection filter to $d_x>1$}\label{sec:Extension_to_dx_greater_than_1}
The numerical implementation of the projection filter for the exponential family relies on computing the expected values of the natural statistics and the Fisher matrix. In the previous sections, we have shown how to calculate the log-partition function for both the unidimensional and multidimensional cases. The repetitive uses of numerical integrations for computing the Fisher information matrix and expectations of natural statistics can be avoided by using automatic differentiation and relations \eqref{eq:eta}. However, it still remains to show how to obtain the expectations of the remaining statistics in $\tilde{c}$. 

In unidimensional dynamics, it is well known that the Fisher information matrix and expectations of the natural statistics can be obtained efficiently via a recursion, given that we have calculated the first $m$ of them \citep{Brigo1995}. Unfortunately, this recursion does not hold for multidimensional problems. This can be concluded from Proposition~\ref{prp:Lemma2.3Extension}.

Therefore, in order to solve the projection filtering problems efficiently when $d_x\geq 2$, we show that it is possible to obtain the expectations of the expanded natural statistics $\tilde{c}$ via computing certain partial derivatives. This procedure is particularly efficient by using modern automatic differentiation tools. The main result is given in Proposition~\ref{prp:ExpectationByExtension}. 

Let us set some notations as follows. Let $c : \mathbb{R}^{d_x} \rightarrow \mathbb{R}^m$, where $m$ is the dimension of the parameter space $\Theta$. Define a multi-index $\bunderline{i}\in \mathbb{N}^{d_x}$ and
\begin{equation}
    x^{\bunderline{i}} := x_1^{\bunderline{i}(1)} x_2^{\bunderline{i}(2)} \cdots x_{d_x}^{\bunderline{i}(d_x)}.
\end{equation}
We can identify the elements of $c$ by the multi-index $\bunderline{i}$. As before, let $\eta_{\bunderline{i}} := \E_\theta  [c_{\bunderline{i}}]$. %
For two multi-indexes $\bunderline{i},\bunderline{j} \in \mathbb{N}^{d_x}$, we define an addition operation element-wise. For simplicity, we also write $\bunderline{j}=\bunderline{i}+a$ where $\bunderline{j}(k) =\bunderline{i}(k)+a$, for $k=1,\ldots,d_x$ and $a \in \mathbb{N}$.
Then we have the following result as an analog of Lemma 3.3 from \citet{Brigo1995}.
\begin{prop}\label{prp:Lemma2.3Extension}
    Let a set of exponential parametric densities be given by
    \begin{align*}
        S &= \left\{ p(\cdot, \theta), \theta \in \Theta \right\}, & &p(x,\theta) = \exp(c(x)^\top \theta - \psi),
    \end{align*}
    where $\Theta \subset \mathbb{R}^m$ is open. The log-partition function $\psi$ is infinitely differentiable on $\Theta$. Furthermore, the first, second, and $k$-th  moments of the natural expectations are given by 
    \begin{subequations}
        \begin{align}
            \E_\theta \mqty[c_{\bunderline{i}}] &= \partial_{\bunderline{i}}\psi(\theta) =: \eta_{\bunderline{i}}(\theta), \label{eq:Eta}\\
            \E_\theta \mqty[c_{\bunderline{i}}c_{\bunderline{j}}] &= \partial^2_{\bunderline{i},\bunderline{j}}\psi(\theta) + \eta_{\bunderline{i}}(\theta) \eta_{\bunderline{j}}(\theta),\label{eq:SecondMoment}\\
            \E_\theta \mqty[ c_{\bunderline{i}_1} \cdots c_{\bunderline{i}_k}] &= \exp(-\psi(\theta)) \partial^k_{\bunderline{i}_1\cdots \bunderline{i}_k}\exp(\psi(\theta))\label{eq:HighMoment}.
        \end{align}
    \end{subequations}
    The Fisher information matrix satisfies
    \begin{align}
        g_{\bunderline{i}\bunderline{j}} = \partial^2_{\bunderline{i}\bunderline{j}}\psi(\theta). \label{eq:Fisher}
    \end{align}
    In particular, if the natural statistics are a collection of monomials of $x$,
    \begin{align*}
        c = \mqty[x^{\bunderline{i}_1}, \ldots, x^{\bunderline{i}_m}],
    \end{align*}
    the Fisher information matrix satisfies
    \begin{align}   
        g_{\bunderline{i}\bunderline{j}} = \eta_{\bunderline{i}+\bunderline{j}}(\theta) - \eta_{\bunderline{i}}(\theta)\eta_{\bunderline{j}}(\theta), \label{eq:fisher_C_multi}
    \end{align}
    where $\eta_{\bunderline{i}+\bunderline{j}}(\theta) := \E_\theta[x^{\bunderline{i}+\bunderline{j}}]$ and for any $\bunderline{i}\in \mathbb{N}^{d_x}$
    \begin{align}
        \E_\theta \mqty[x^{\bunderline{i}}] &= (-1)^{d_x} \int_{\mathbb{R}^{d_x}} \dfrac{x^{\bunderline{i}+1}}{\pi(\bunderline{i}+1)} \dfrac{\partial^{d_x} p(x,\theta)}{\partial x_1\cdots \partial x_{d_x}} dx, \label{eq:eta_i_C_multi}
    \end{align}
    where
    \begin{align*}
        \pi(\bunderline{i}) &= \bunderline{i}(1)\times \cdots \times \bunderline{i}({d_x}).\label{eq:expected_x_power_i}
    \end{align*}
\end{prop}
\begin{proof}
    All results except \eqref{eq:fisher_C_multi} and \eqref{eq:eta_i_C_multi} follow directly from \citet[Lemma 2.3]{Brigo1995} and \citet{Amari1985} using multi-index notations. 
    By definition, $\E_\theta[c_{\bunderline{i}}c_{\bunderline{j}}] = \eta_{\bunderline{i}+\bunderline{j}}(\theta)$, hence using \eqref{eq:SecondMoment}, \eqref{eq:Eta}, and \eqref{eq:Fisher}, we obtain \eqref{eq:fisher_C_multi}.
    For \eqref{eq:eta_i_C_multi}, observe that
    \begin{align*}
        \E_\theta \mqty[x^{\bunderline{i}}] &= \idotsint_{\mathbb{R}^{d_x}} x_1^{\bunderline{i}(1)} x_2^{\bunderline{i}(2)} \cdots x_d^{\bunderline{i}(d_x)} p(x,\theta) dx_1\cdots dx_{d_x} \\
        &= \idotsint_{\mathbb{R}^{d_x-1}} x_1^{\bunderline{i}(1)} x_2^{\bunderline{i}(2)}\cdots x_{d_x-1}^{\bunderline{i}(d_x-1)} \left[\int_{-\infty}^\infty x_{d_x}^{\bunderline{i}(d_x)} p(x,\theta) dx_{d_x}\right] dx_1\cdots dx_{d_x-1} \\
        &= \idotsint_{\mathbb{R}^{d_x-1}} x_1^{\bunderline{i}(1)} x_2^{\bunderline{i}(2)} \cdots x_{d_x-1}^{\bunderline{i}(d_x-1)} \\
        & \times \left[ \left. \dfrac{x_{d_x}^{\bunderline{i}(d_x)+1}}{\bunderline{i}(d_x)+1} p(x,\theta) \right|^\infty_{-\infty} - 
            \int_{-\infty}^\infty \dfrac{x_{d_x}^{\bunderline{i}(d_x)+1}}{\bunderline{i}(d_x)+1} \pdv{p(x,\theta)}{x_{d_x}}  dx_{d_x}\right] dx_1\cdots dx_{d_x-1}.
    \end{align*}
    Doing a similar integration for the remaining integration variables $x_1,\ldots,x_{d_x-1}$ and since $p(x,\theta)$ vanishes as $\norm{x}\rightarrow \infty$ faster than any polynomial, we end up with \eqref{eq:eta_i_C_multi}. \qed
\end{proof}

In unidimensional problems \eqref{eq:eta_i_C_multi} admits a recursion, where we can express $\eta_{m+i}$ as a linear combination of $\eta_i,\ldots,\eta_{m+i-1}$ \citep{Brigo1995}. However, the recursion does not hold for $d_x>1$. This can be verified by evaluating \eqref{eq:expected_x_power_i} for the case $d_x=2$ and natural statistics $\left\{ x^{\bunderline{i}} \right\}$, where $\bunderline{i} \in \left\{ (0,1),(0,2),(1,0),(1,1),(2,0) \right\}$. Instead of relying on a recursion to compute the higher moments, we can use a simple procedure which is given in the following proposition.
\begin{prop}\label{prp:ExpectationByExtension}
    Let $s(x): \mathbb{R}^d \rightarrow \mathbb{R}$ be another statistic, linearly independent of the natural statistics $c(x)$ with the corresponding natural parameters $\theta$ (see Proposition \ref{prp:Lemma2.3Extension}). If there exists an open neighborhood of zero $\Theta_0$, such that $\tilde{\psi}(\tilde{\theta}) := \log(\int_\mathcal{X} \exp(\tilde{c}^\top \tilde{\theta}) dx)< \infty$ for any $\tilde{\theta} \in \Theta \times \Theta_0$ where $\tilde{c} = [c^\top(x) \; s(x)]^\top$, then the expectation of $s(x)$ under $p_\theta$ is given by
    \begin{align*}
        \mathbb{E}_\theta [s(x)] &= \left.\pdv{\tilde{\psi}(\tilde{\theta})}{\tilde{\theta}_{m+1}}\right|_{\tilde{\theta} = \theta_\ast}, & \theta_\ast = \mqty[\theta^\top & 0]^\top.
    \end{align*}
\end{prop}
\begin{proof}
    Define a density as $\bar{p}(x,\tilde{\theta}) := \exp(\tilde{c}^\top \tilde{\theta}- \tilde{\psi}(\tilde{\theta}))$, where $\bar{p}(x,\theta_\ast) = p(x,\theta)$. Then, by Proposition \ref{prp:Lemma2.3Extension}:
    \begin{align*}
        \E_\theta [s(x)] &= \int_{\mathcal{X}} p(x,\theta)s(x) dx = \int_{\mathcal{X}} \bar{p}(x,\theta_\ast)s(x) dx =  \left.\pdv{\tilde{\psi}(\tilde{\theta})}{\tilde{\theta}_{m+1}}\right|_{\tilde{\theta}= \theta_\ast}.
    \end{align*}
    \qed
\end{proof}
The above proposition shows that by using an expanded exponential family, the higher moments can be computed via partial differentiation. This removes the difficulty of obtaining the higher moments which previously could only be cheaply obtained for unidimensional problems via recursion \citep{Brigo1995}. The linear independence condition in this assumption is required since otherwise, the expectation of the statistic $s(x)$ can be obtained by a linear combination of the expectations of the elements of $c(x)$.

\section{Numerical experiments}
\label{sec:NumericalExamples}
\begin{figure*}[!h]
\centering
\subfloat[density evolution $\phi=\text{arctanh}$]{\label{fig:arctanh_wireframe}\includegraphics[trim={3.5cm 1.cm 3.5cm 1.cm},clip,width=0.5\linewidth]{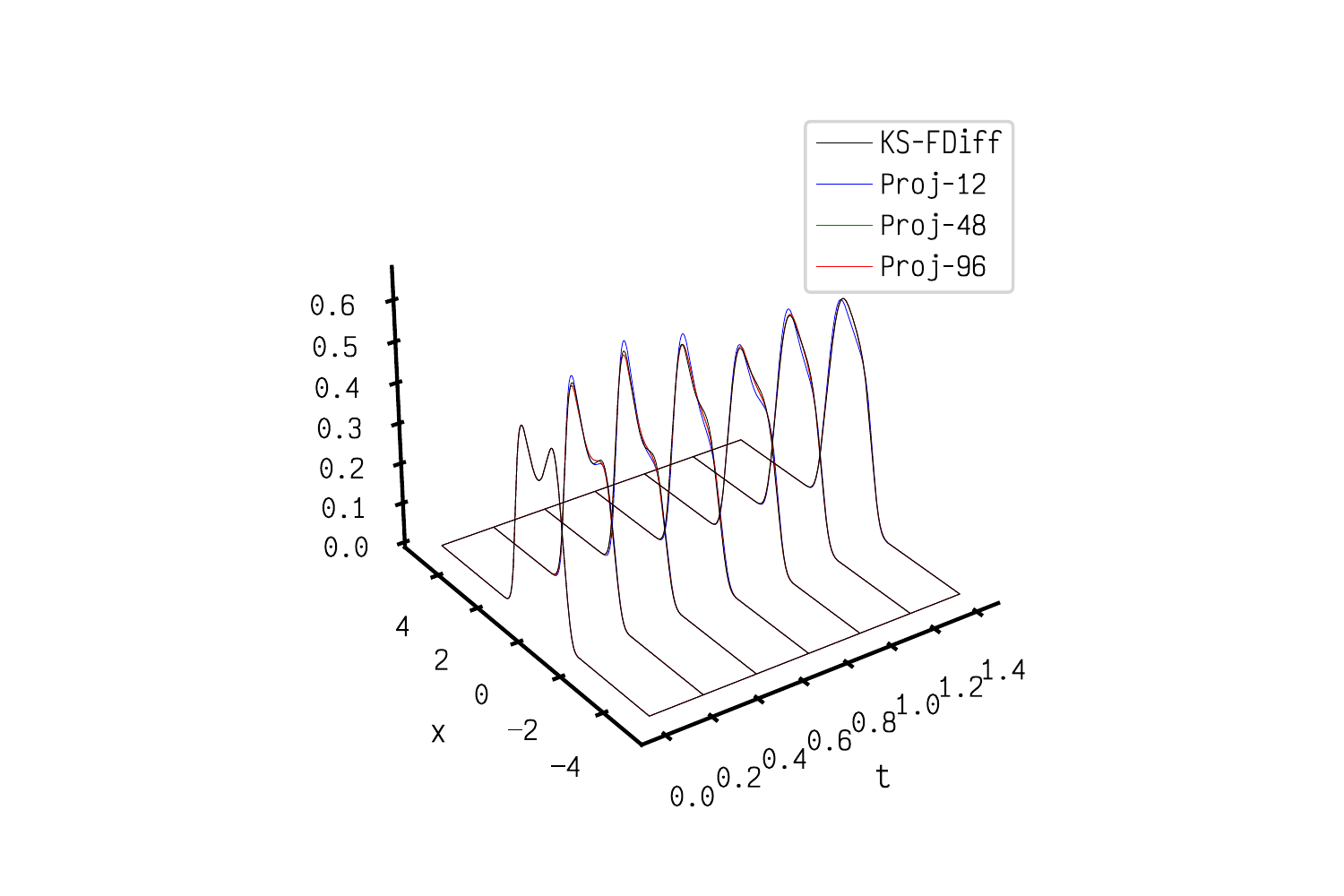}}
\subfloat[density evolution $\phi=\frac{\tilde{x}}{1-\tilde{x}^2}$]{\label{fig:boyd_wireframe}\includegraphics[trim={3.5cm 1.cm 3.5cm 1.cm},clip,width=0.5\linewidth]{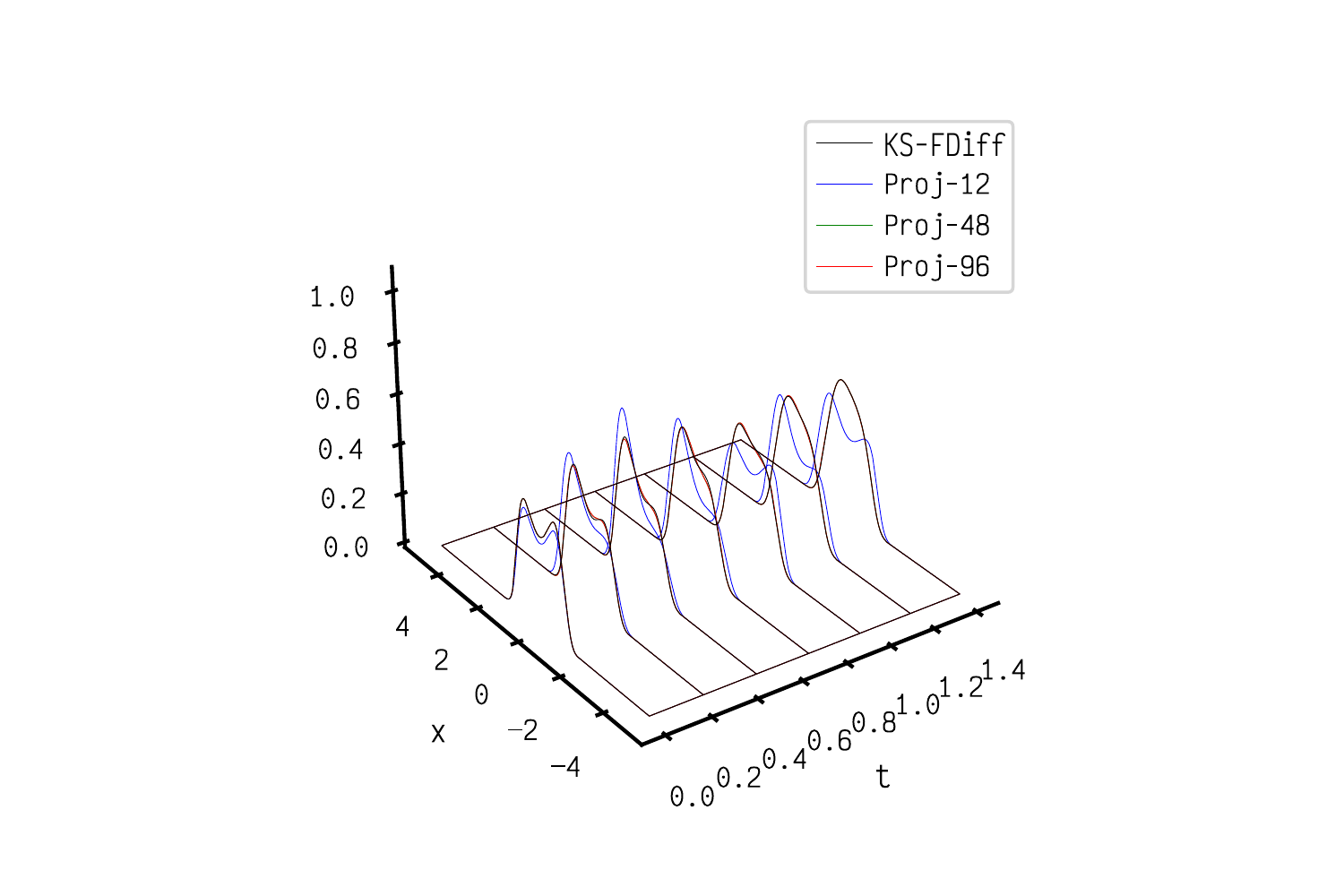}}\\
\subfloat[Hellinger distance $\phi=\text{arctanh}$]{\label{fig:arctanh_hellinger}\includegraphics[trim={0.2cm 0.0cm 1.5cm 1.0cm},clip,width=0.5\linewidth]{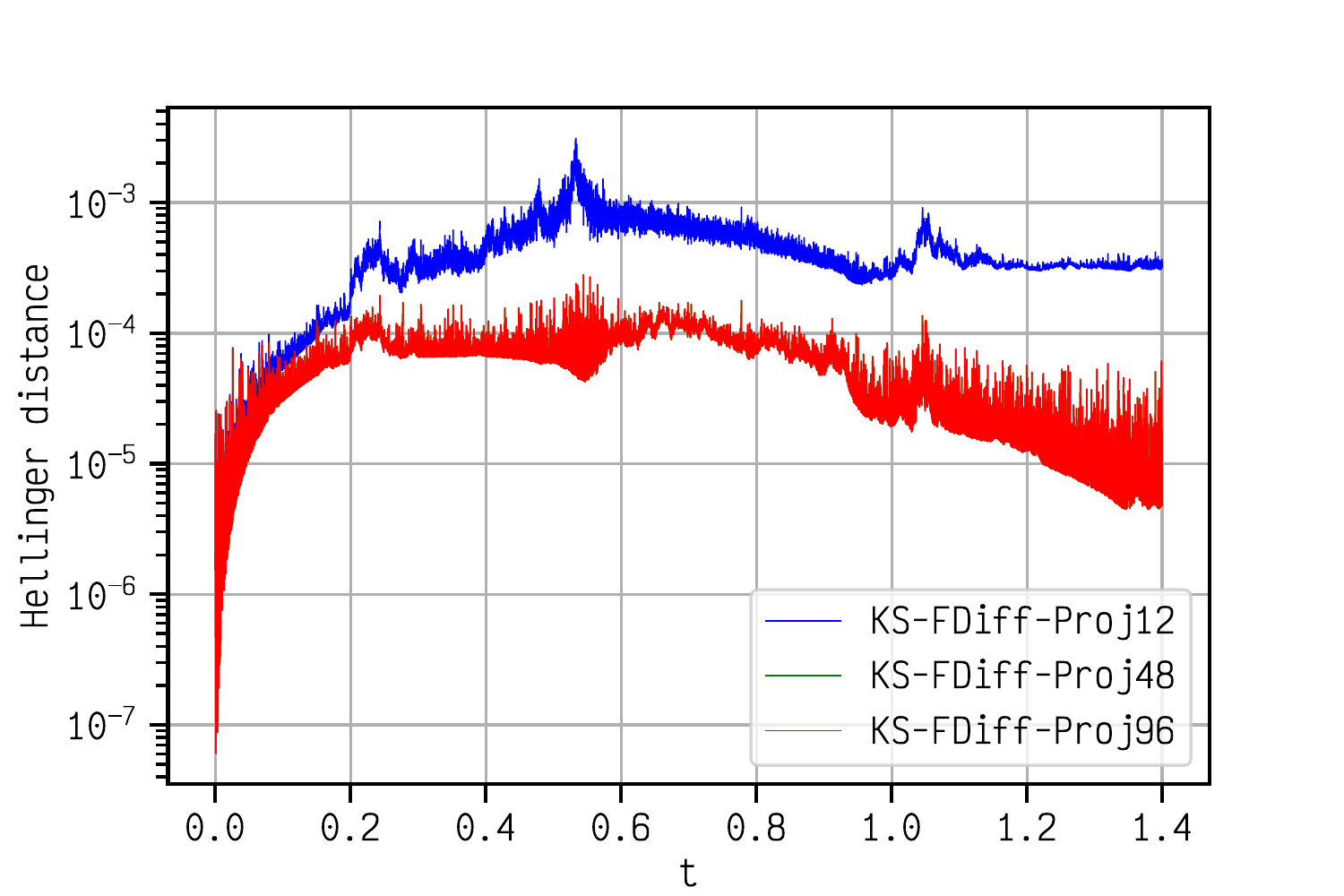}}
\subfloat[Hellinger distance $\phi=\frac{\tilde{x}}{1-\tilde{x}^2}$]{\label{fig:boyd_hellinger}\includegraphics[trim={0.2cm 0.0cm 1.5cm 1.0cm},clip,width=0.5\linewidth]{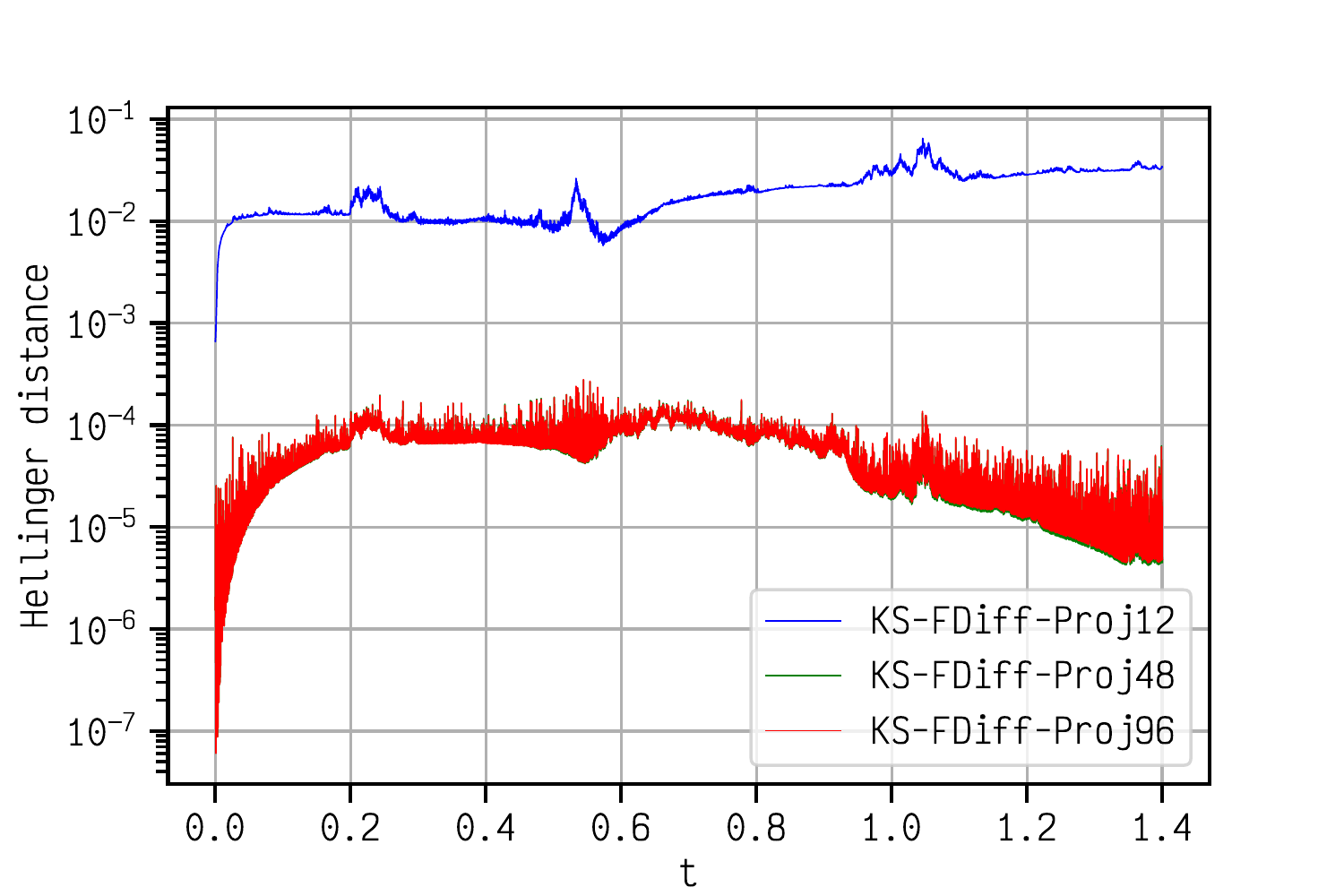}}\\

\caption{Figures \ref{fig:arctanh_wireframe} and \ref{fig:arctanh_wireframe} show the evolution of densities solved by approximating the Kushner--Stratonovich equation via a finite difference scheme (KS-FDiff), and the projection filter using different bijections and numbers of Chebyshev nodes (blue, green and red for 12, 48, and 96 quadrature nodes respectively), while Figures \ref{fig:arctanh_hellinger} and \ref{fig:boyd_hellinger} show their respective Hellinger distances. It can be seen that the Hellinger distances of the projection filter and the solution of the finite difference scheme using 48 and 96 quadratures are practically indistinguishable.}

\label{fig:comparison_densities}
\end{figure*}

In this section, we demonstrate the application of the proposed implementation of the projection filter to selected optimal filtering problems. We compare the projection filter densities against the densities obtained by solving the Kushner--Stratonovich equation via the central finite difference approach \citep{Bain2009} and kernel density estimate from a bootstrap particle filter with systematic resampling \citep{Chopin2020}. All numerical simulations in this section are implemented using JAX \citep{jax2018github} where the automatic differentiation is supported by default. The sparse-grid implementation is taken from the TASMANIAN library \citep{stoyanov2015tasmanian}. 

For the unidimensional problem below, we use a simple central finite-difference scheme to approximate the solution of the Kushner--Stratonovich equation. This numerical scheme is stable enough for the simulation period that we are interested in. For the nonlinear two-dimensional problem, we use the Crank--Nicolson scheme since for this problem the central finite-difference scheme is unstable in the grid space and simulation time that we are interested in. We set the ratio $\Delta t/(\Delta x^2)$ to be significantly below $0.5$ according to the von Neumann stability analysis \citep{Smith1985}.

\subsection{Unidimensional example}
The first example is the following scalar system dynamic with a nonlinear measurement model:
\begin{subequations}
    \begin{align}
        dx_t &=  \sigma dW_t,\\
        dy_t &=  \beta x^3_t dt +  dV_t,
        \end{align}            
\end{subequations}
with independent standard Brownian motions $\{ W_t , t \geq 0\}$  and  $\{ V_t , t \geq 0\}$, and where $\sigma=0.4, \beta = 0.8$ are constants. We generate one measurement trajectory realization from the model with $x_0 = 1$. The simulation time step is set to be $10^{-4}$. 

We use the exponential manifold with $c_i \in \{x, x^2, x^3, x^4\}$ and choose the initial condition to be $\theta_0 = [0,1,0,-1]$. We assume that this initial condition is exactly the initial density of the dynamical system which corresponds to a double-mode non-Gaussian density with peaks at $-1$ and $1$. We compare two choices of the bijections, the first is $\text{arctanh}(\tilde{x})$ and the second is $\frac{\tilde{x}}{1-\tilde{x}^2}$. Furthermore, in the Chebyshev quadrature \eqref{eq:Gauss_Chebyshev_quadrature} we compare three choices for the number of Chebyshev nodes, where we set them to be $\Chebyshevbasislow$, $\Chebyshevbasismed$, and $\Chebyshevbasishigh$. We also solve the Kushner--Stratonovich SPDE with finite differences on an equidistant grid $(-5,5)$ with one thousand points. 

The evolution of the densities provided by the Kushner--Stratonovich equation approach and projection filter using the different settings can be seen in Figures \ref{fig:arctanh_wireframe} and \ref{fig:boyd_wireframe}. The Hellinger distances between the projection filter densities and the numerical solutions of the Kushner--Stratonovich equation can be seen in Figures \ref{fig:arctanh_hellinger} and \ref{fig:boyd_hellinger}. These figures show that for this particular filtering problem when using a proper selection of a bijection function, the projection filter solution using a few nodes is very close to the one using a very high number of nodes. In particular, when using $\text{arctanh}$ as the bijection, with the node number as low as $\Chebyshevbasislow$, the Hellinger distance of solution of the projection filter to the Kushner--Stratonovich solution until $t=1.4$ is below $10^{-3}$. This is not the case when we use $\frac{\tilde{x}}{1-\tilde{x}^2}$, as its Hellinger distance is significantly larger with lower than with higher number of nodes. The interpolation using the latter bijection with $\Chebyshevbasislow$ nodes is less accurate since the initial condition as can be seen in the Figure \ref{fig:boyd_wireframe}. This inaccuracy leads to a false double-peaked density that appears from time $1.0$

We can also see that the projection filter solution for the case of $\Chebyshevbasismed$ nodes is almost indistinguishable from that of $\Chebyshevbasishigh$, regardless of the choice of the bijection, with the Hellinger distance to the Kushner--Stratonovich solution being kept below $10^{-4}$ for most of the time until $1.4$. This shows one numerical benefit of this approach for this particular example: the filter equation can be propagated  efficiently using a very low number of integration nodes.

\subsection{Two-dimensional example}
Here, we show that the projection filter solved using the proposed method can perform well in many interesting problems. In these experiments we use the Gauss--Patterson scheme with the canonical domain $(-1,1)^{\otimes d_x}$.

\subsubsection{Comparison to Kalman--Bucy filter}
In this experiment, we consider the following stochastic estimation problem on a two-dimensional linear system:
\begin{subequations}
    \begin{align}
        d\mqty[x_{1,t}\\x_{2,t}] &= -\mqty[x_{1,t}\\x_{2,t}] dt + \sigma_w \mqty[dW_{1,t}\\dW_{2,t}],\\
        d\mqty[y_{1,t}\\y_{2,t}] &= -\mqty[x_{1,t}\\x_{2,t}] dt + \sigma_v \mqty[dV_{1,t}\\dV_{2,t}].
    \end{align}\label{eqs:SDE_all_n_d_linear}    
\end{subequations}

We set $dt=10^{-3}$, $\sigma_v=10^{-1}$, and $\sigma_w=1$a, and simulate the SDE with $n_t=1000$ Euler--Maruyama steps and record the measurement trajectory. To obtain a reference result, we implement the Kalman--Bucy filter using the Euler scheme. For the projection filter, we compare the two integration procedures:  the sparse-grid methods and the quasi Monte Carlo (qMC) \citep{Leobacher_2014}. We choose the $\text{arctanh}$ function as the bijection function based on our previous examination on the unidimensional problem. We use the Gauss--Patterson scheme for unidimensional sparse grid where we compare various integration levels. For the qMC, we choose the Halton low discrepancy sequence. The number of quadrature points for qMC is chosen to be equal to the selected sparse grid quadrature points.
We use the Gaussian family in our projection filter. This means that we set the natural statistics to be $\left\{ x^{\bunderline{i}} \right\}$, where $\bunderline{i} \in \left\{ (0,1),(0,2),(1,0),(1,1),(2,0) \right\}$. 

Figures \ref{fig:Hellinger_distance_Kalman_1} and \ref{fig:Hellinger_distance_Kalman_2} illustrate the Hellinger distances between the approximated solution of the projection filter and the solution of the Kalman--Bucy filter, whereas Figures \ref{fig:Kalman_Bucy_realizations_x_1} and \ref{fig:Kalman_Bucy_realizations_x_2} show the estimation outcomes for both states. Using 49 quadrature points, which is comparable to level 3 in sparse-grid integration, both integration methods provide estimation results that are quite similar. The sparse-grid based projection filter approximates the standard Kalman--Bucy filter solution more accurately than qMC. As more quadrature nodes are added, the Hellinger distances rapidly decrease, as can be seen in Figures \ref{fig:Hellinger_distance_Kalman_1} and \ref{fig:Hellinger_distance_Kalman_2}. With the sparse-grid level set to 5, the sparse-grid based projection filter generates a practically negligible Hellinger distance to the Kalman-Bucy solution after  $t=0.4$. In contrast, the Hellinger distances of the densities determined from qMC integration continue to oscillate around $10^{-4}$.

\newcommand{\lefttrim}{0.0}
\newcommand{\righttrim}{1.5}
\newcommand{\verticaltrim}{0.0}

\begin{figure}[!h]
\centering
\subfloat[SPG-level 3]{\label{fig:Hell_pf_vs_kf_spg_49}\includegraphics[trim={\lefttrim cm \verticaltrim cm \righttrim cm \verticaltrim cm},clip,width=0.5\linewidth]{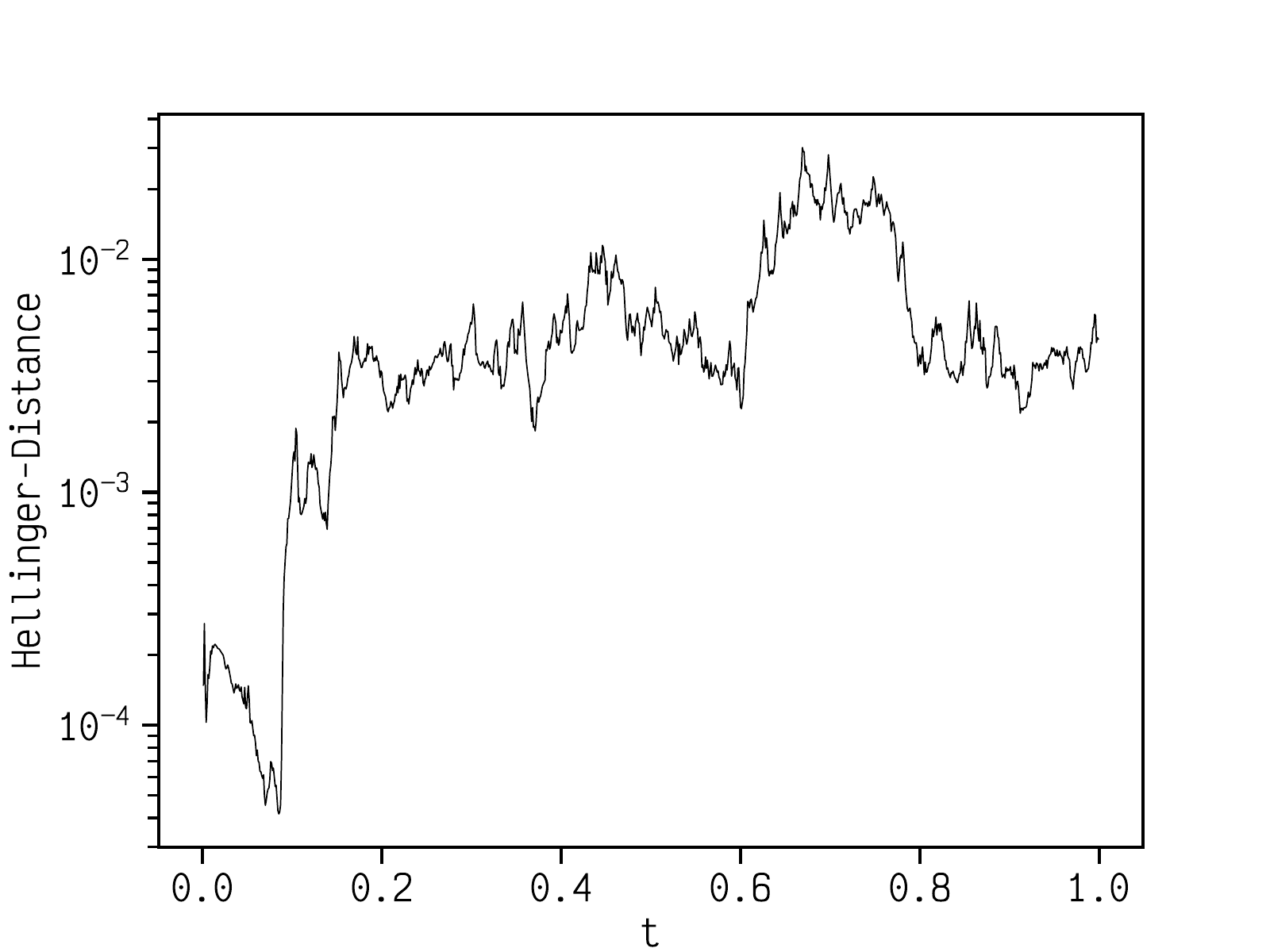}}
\subfloat[qMC-49]{\label{fig:Hell_pf_vs_kf_qmc_49}\includegraphics[trim={\lefttrim cm \verticaltrim cm \righttrim cm \verticaltrim cm},clip,width=0.5\linewidth]{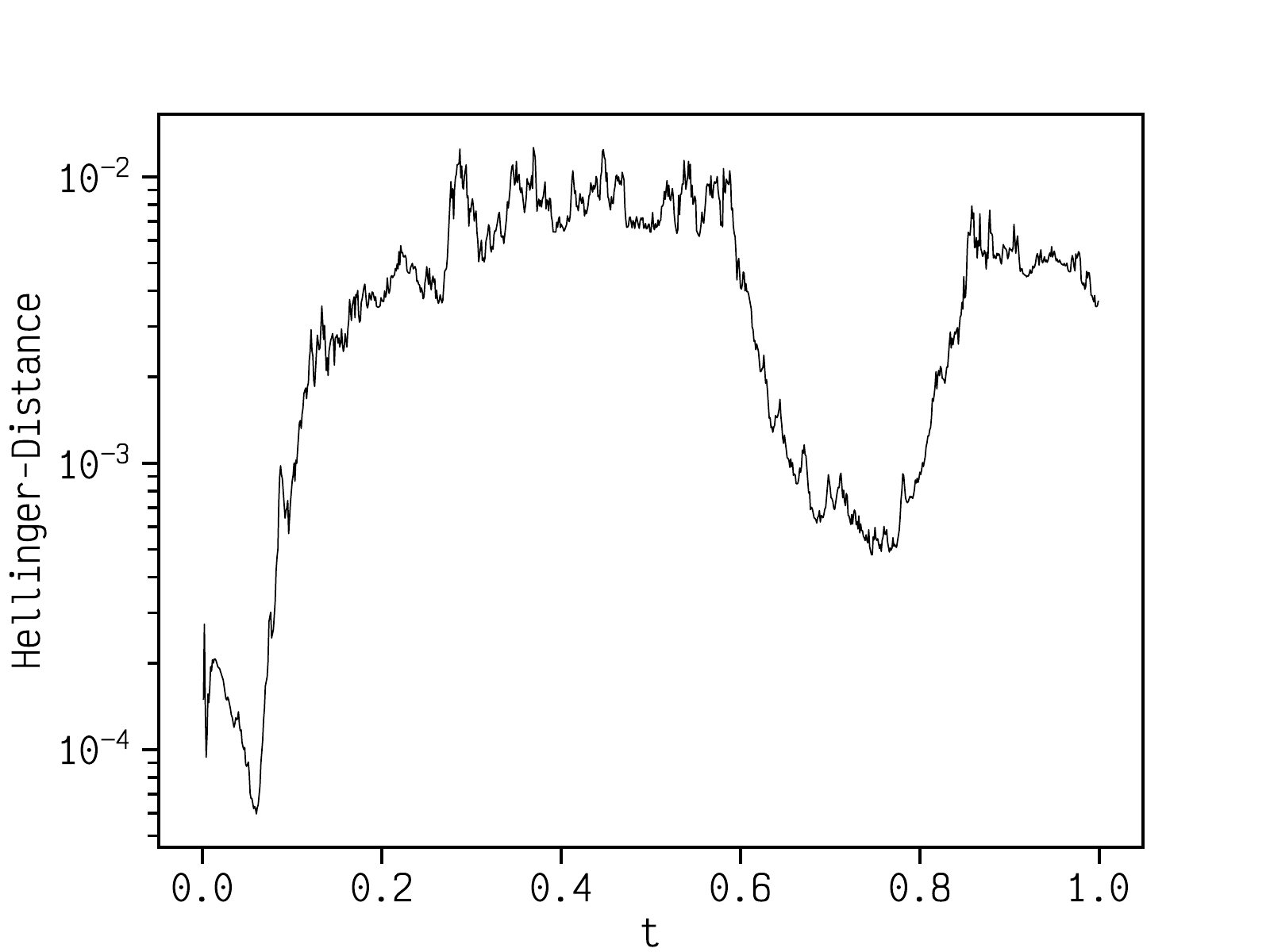}}\\
\subfloat[SPG-level 4]{\label{fig:Hell_pf_vs_kf_spg_129}\includegraphics[trim={\lefttrim cm \verticaltrim cm \righttrim cm \verticaltrim cm},clip,width=0.5\linewidth]{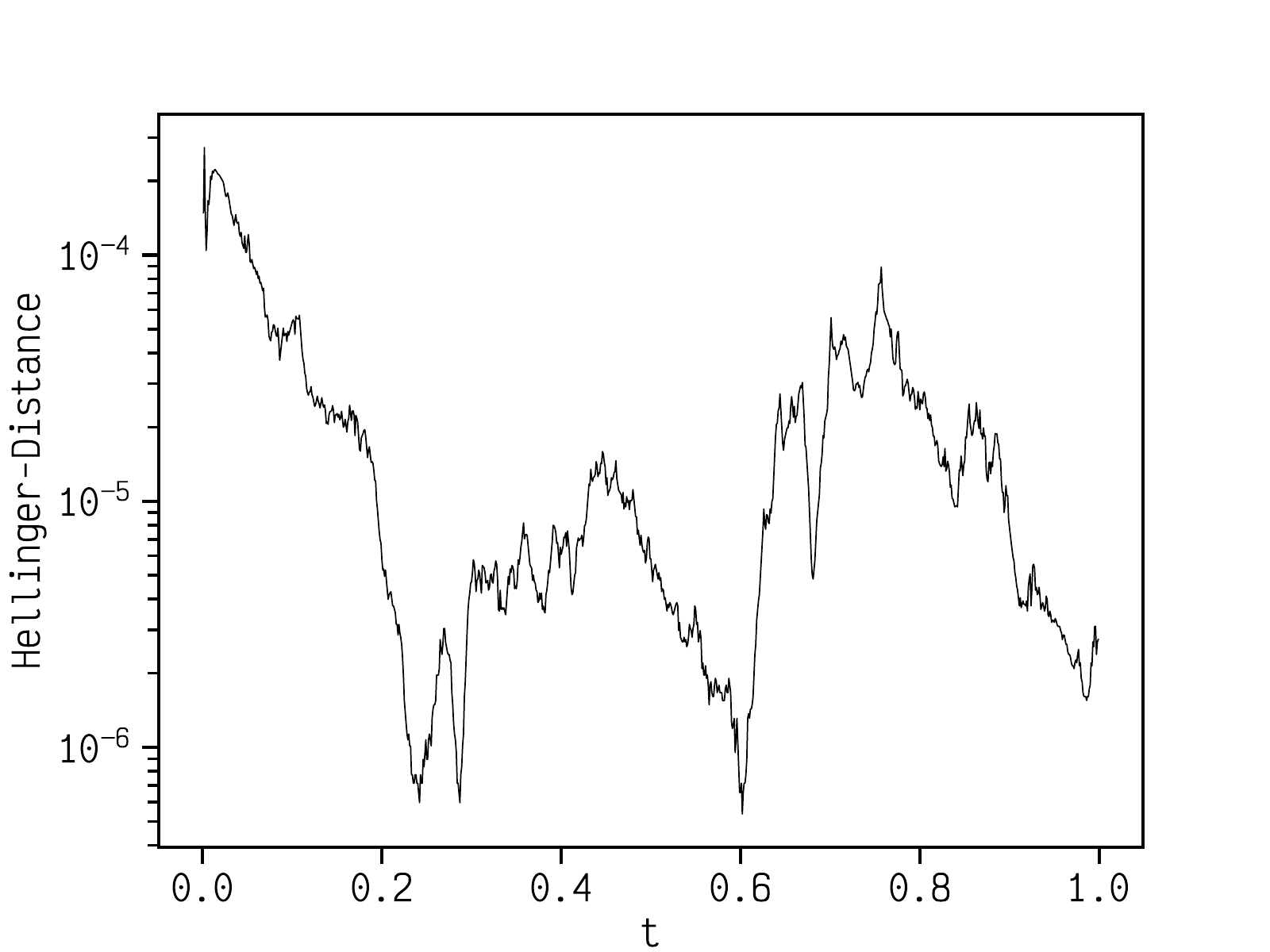}}
\subfloat[qMC-129]{\label{fig:Hell_pf_vs_kf_qmc_129}\includegraphics[trim={\lefttrim cm \verticaltrim cm \righttrim cm \verticaltrim cm},clip,width=0.5\linewidth]{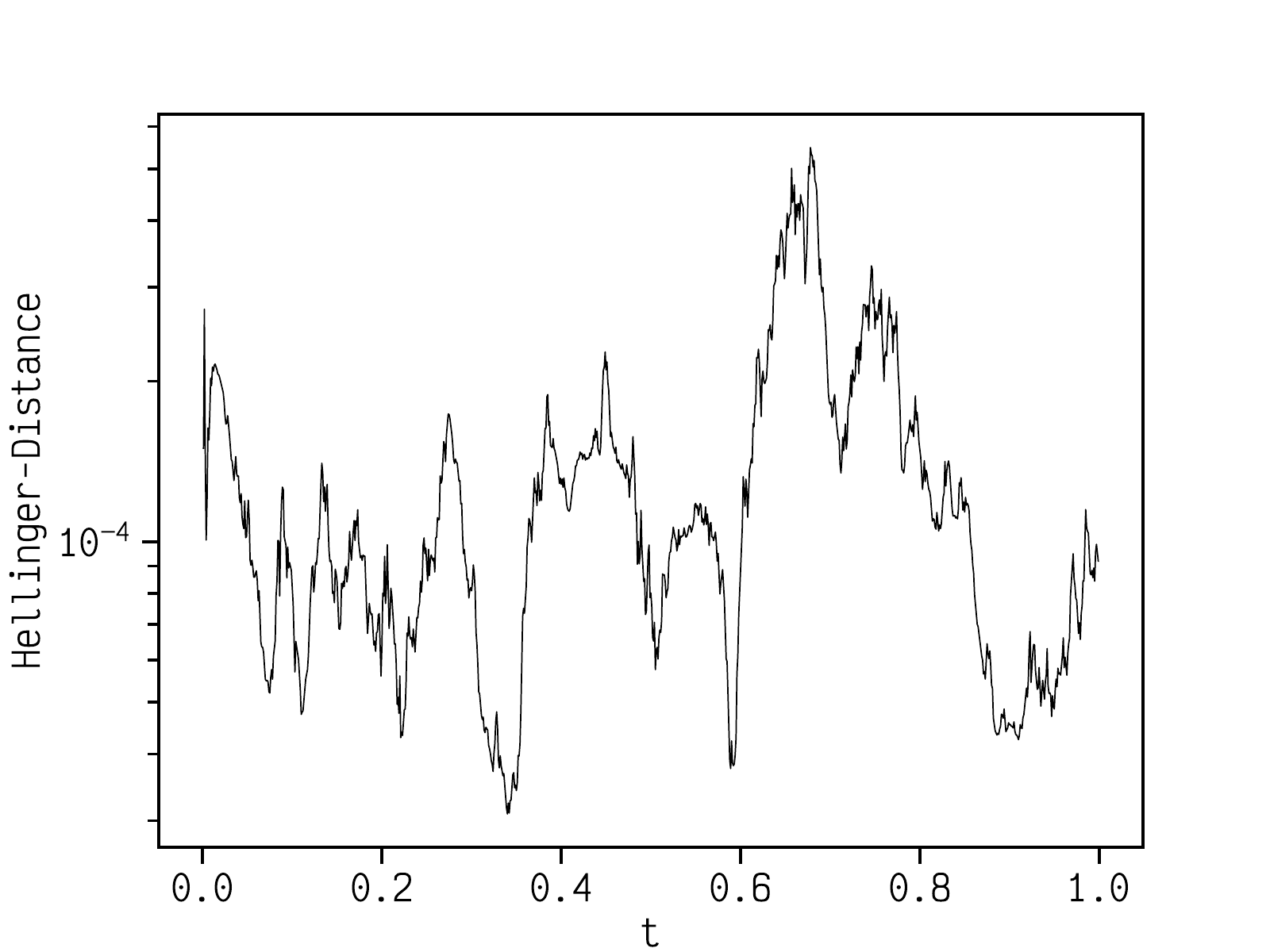}}

\caption{The Hellinger distances between the densities obtained by solving the Kalman--Bucy equation and those obtained by projection filter, solved with different integrators. The number of points in the quasi Monte Carlo integrators are set to be equal to the respective number of sparse grid points. \label{fig:Hellinger_distance_Kalman_1}}
\end{figure}

\begin{figure}[!h]
\centering
\subfloat[SPG-level 5]{\label{fig:Hell_pf_vs_kf_spg_321}\includegraphics[trim={\lefttrim cm \verticaltrim cm \righttrim cm \verticaltrim cm},clip,width=0.5\linewidth]{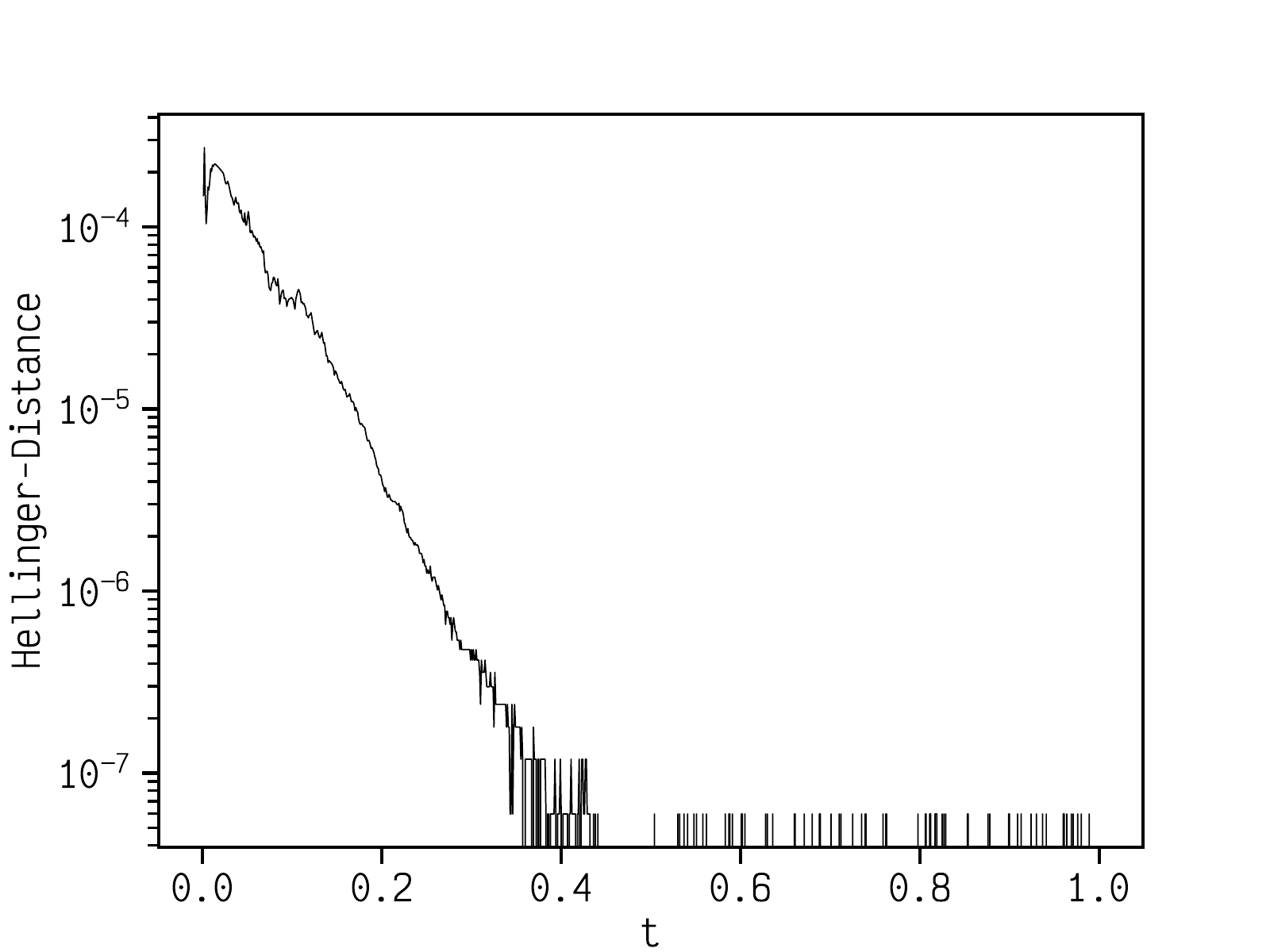}}
\subfloat[qMC-321]{\label{fig:Hell_pf_vs_kf_qmc_321}\includegraphics[trim={\lefttrim cm \verticaltrim cm \righttrim cm \verticaltrim cm},clip,width=0.5\linewidth]{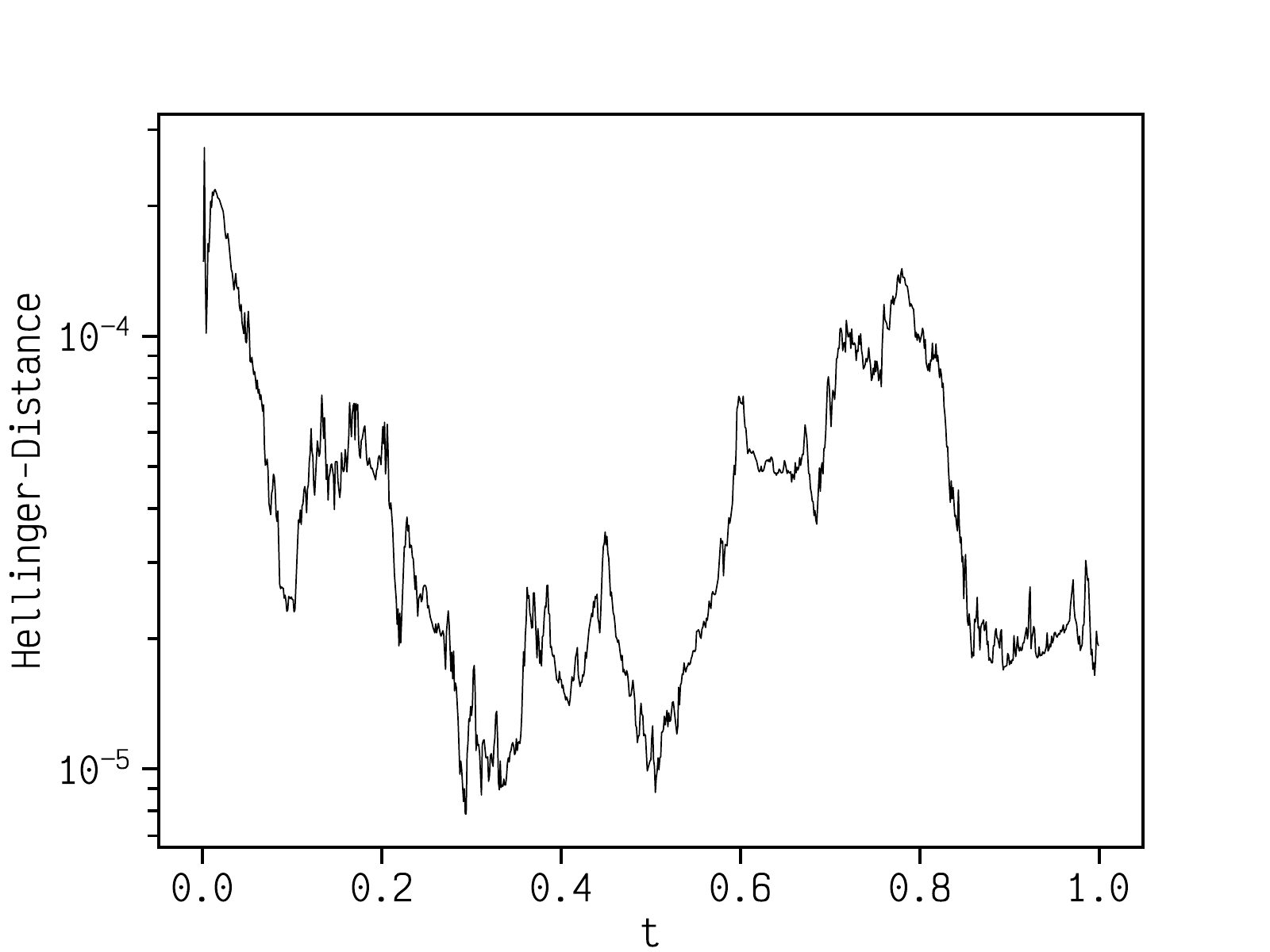}}\\
\subfloat[SPG-level 6]{\label{fig:Hell_pf_vs_kf_spg_769}\includegraphics[trim={\lefttrim cm \verticaltrim cm \righttrim cm \verticaltrim cm},clip,width=0.5\linewidth]{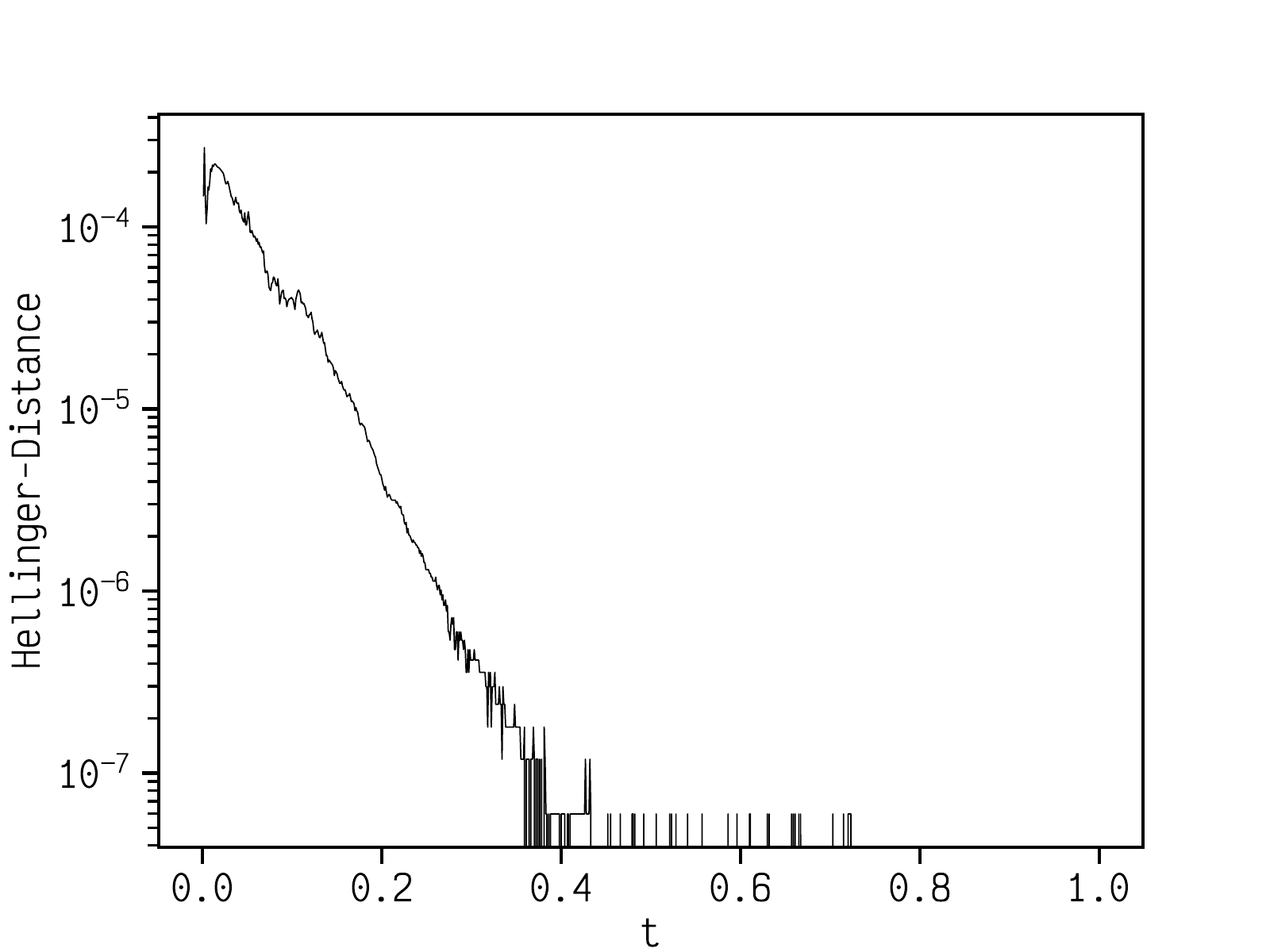}}
\subfloat[qMC-769]{\label{fig:Hell_pf_vs_kf_qmc_145}\includegraphics[trim={\lefttrim cm \verticaltrim cm \righttrim cm \verticaltrim cm},clip,width=0.5\linewidth]{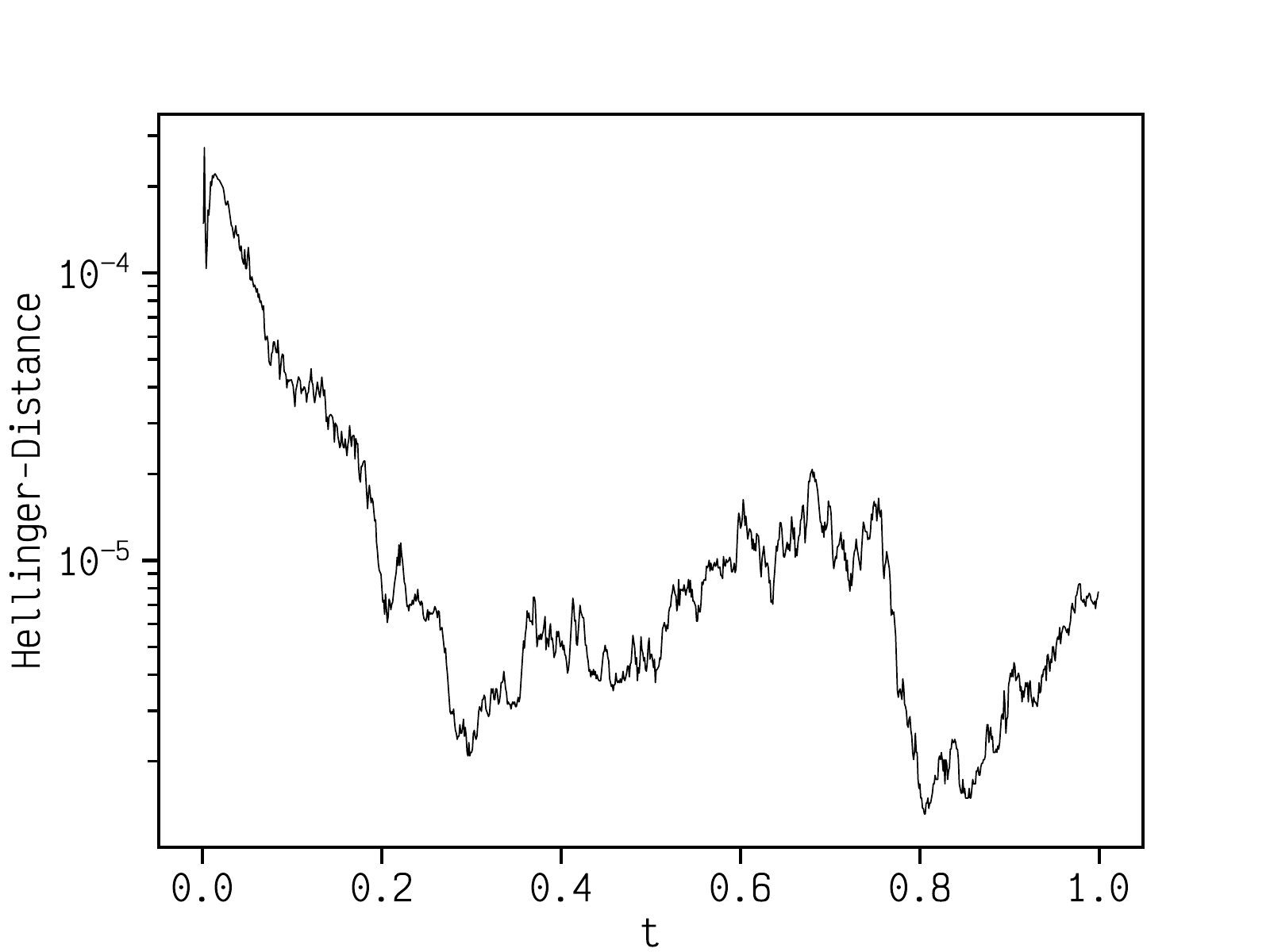}}

\caption{Continuation of Figure \ref{fig:Hellinger_distance_Kalman_1} for different
levels of sparse grid integrators and quasi Monte Carlo points.\label{fig:Hellinger_distance_Kalman_2}}
\end{figure}

\begin{figure}[!h]
\centering
\subfloat[SPG-level 3]{\label{fig:x_0_pf_vs_kf_spg_145}\includegraphics[trim={\lefttrim cm \verticaltrim cm \righttrim cm \verticaltrim cm},clip,width=0.5\linewidth]{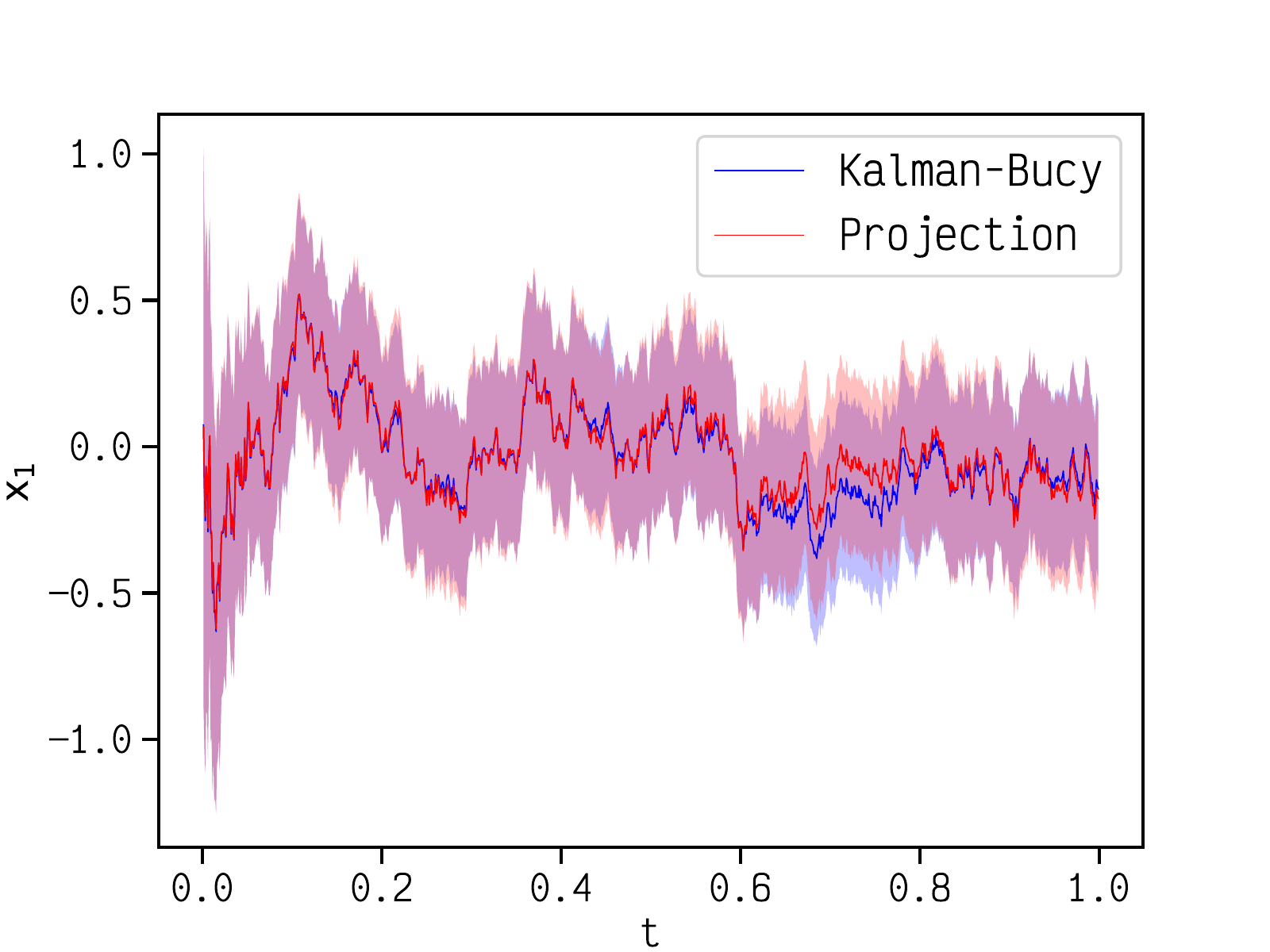}}
\subfloat[qMC-49]{\label{fig:x_0_pf_vs_kf_qmc_145}\includegraphics[trim={\lefttrim cm \verticaltrim cm \righttrim cm \verticaltrim cm},clip,width=0.5\linewidth]{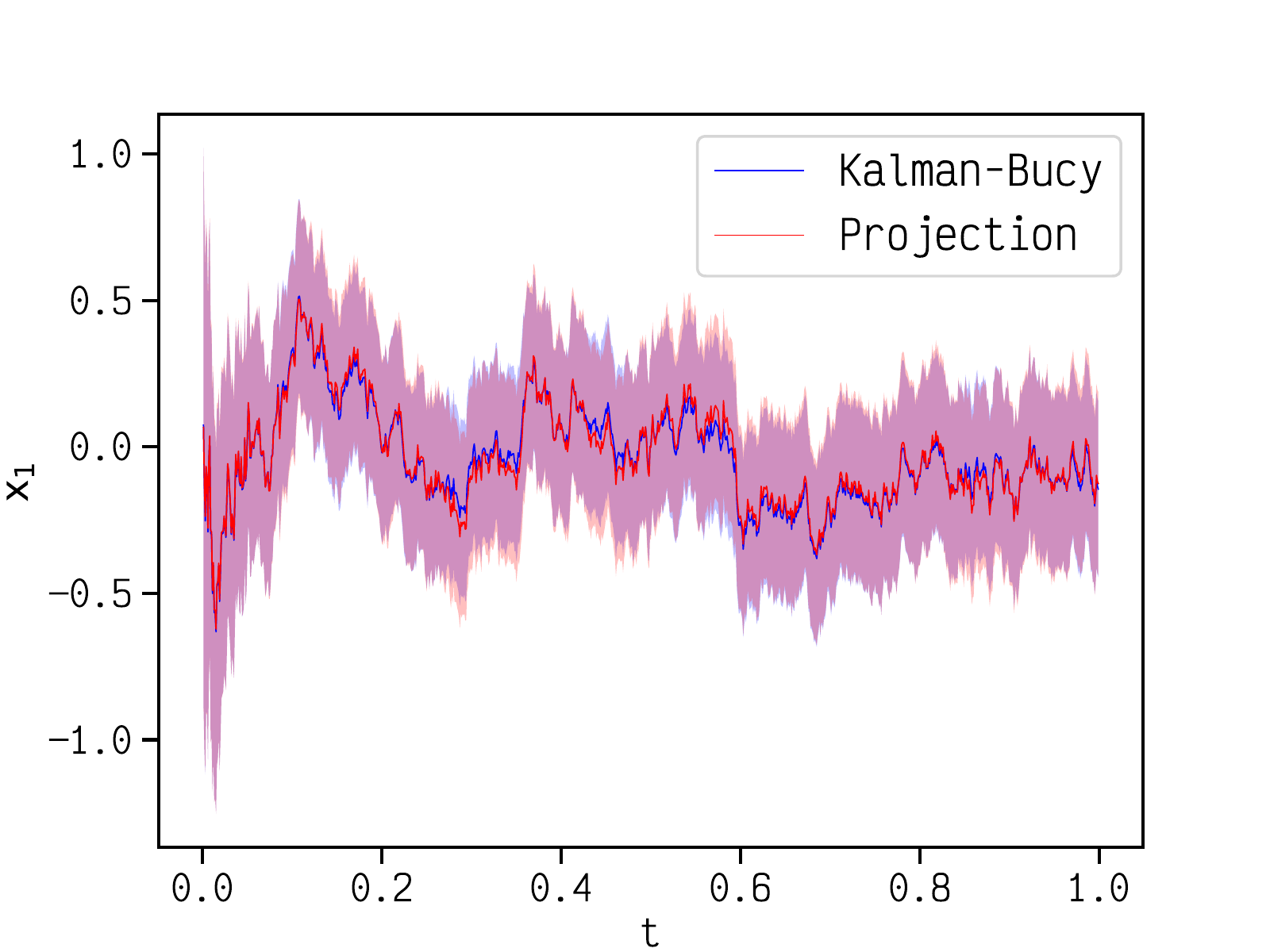}}\\
\subfloat[SPG-level 6]{\label{fig:x_0_pf_vs_kf_spg_1537}\includegraphics[trim={\lefttrim cm \verticaltrim cm \righttrim cm \verticaltrim cm},clip,width=0.5\linewidth]{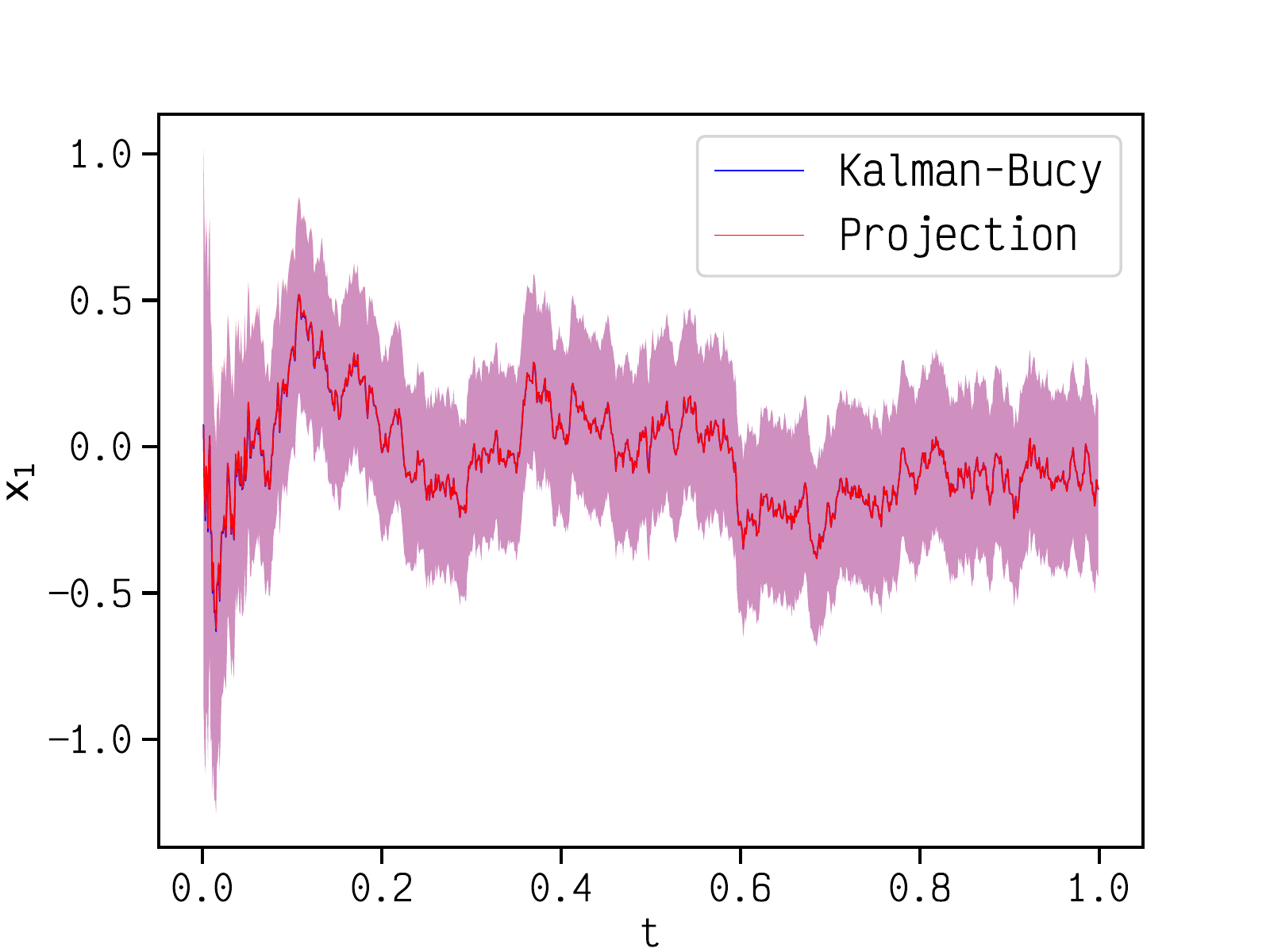}}
\subfloat[qMC-769]{\label{fig:x_0_pf_vs_kf_qmc_1537}\includegraphics[trim={\lefttrim cm \verticaltrim cm \righttrim cm \verticaltrim cm},clip,width=0.5\linewidth]{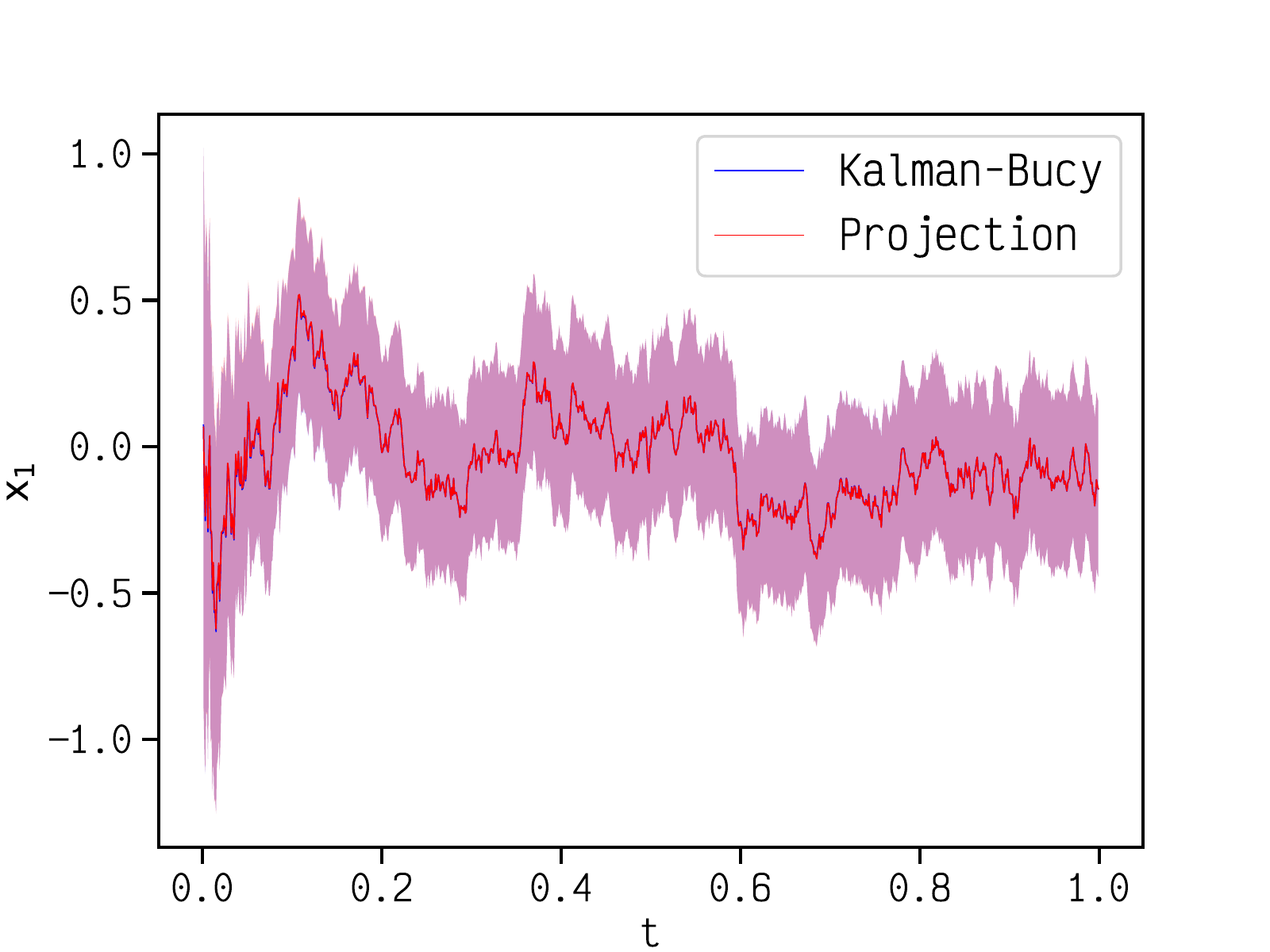}}

\caption{Comparison of mean and one standard deviation confidence intervals for $x_{1,t}$ obtained with the Kalman--Bucy filter and the projection filter solved via different integration setups. \label{fig:Kalman_Bucy_realizations_x_1}}
\end{figure}

\begin{figure}[!h]
\centering
\subfloat[SPG-level 3]{\label{fig:x_1_pf_vs_kf_spg_145}\includegraphics[trim={\lefttrim cm \verticaltrim cm \righttrim cm \verticaltrim cm},clip,width=0.5\linewidth]{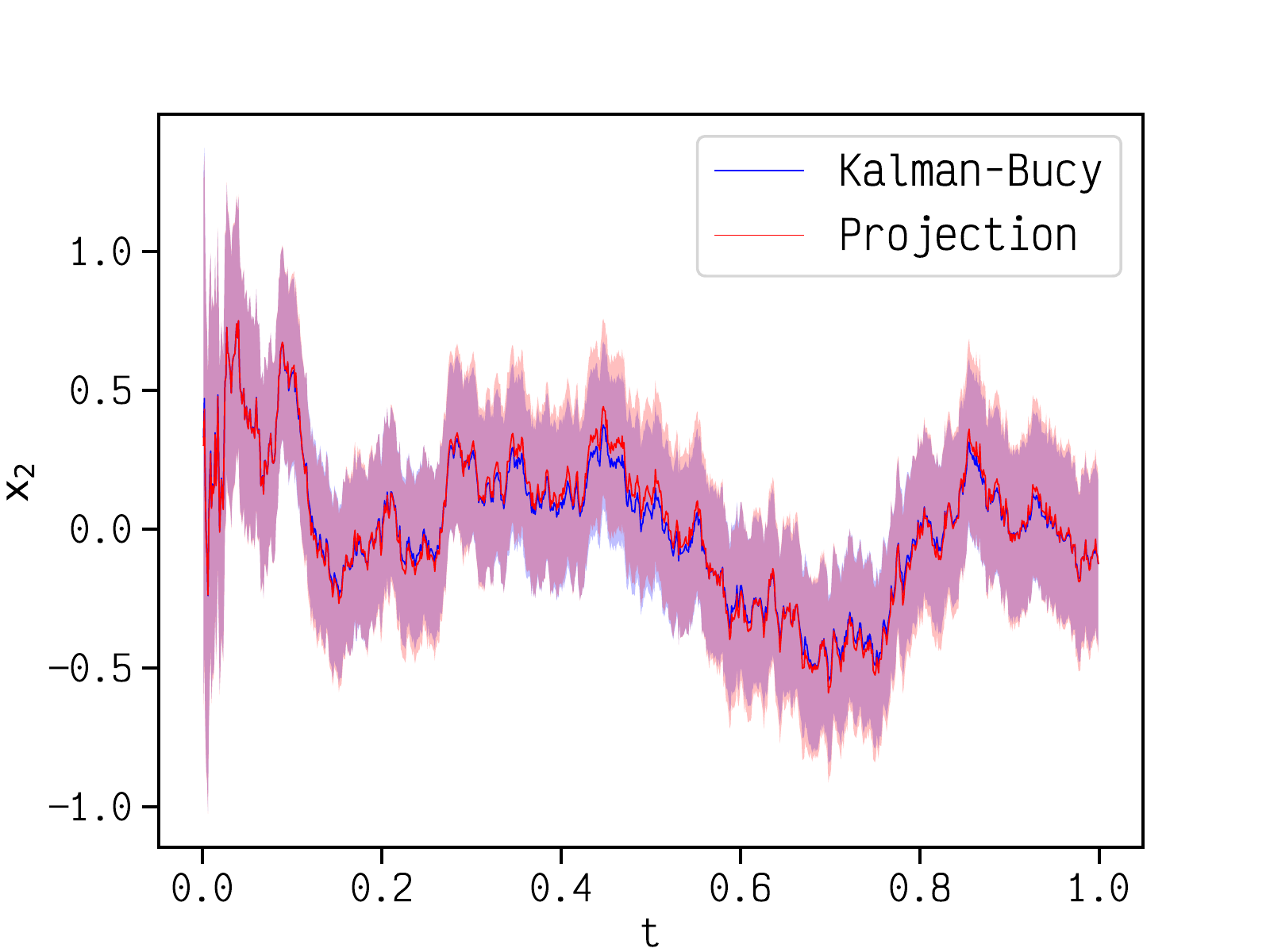}}
\subfloat[qMC-49]{\label{fig:x_1_pf_vs_kf_qmc_145}\includegraphics[trim={\lefttrim cm \verticaltrim cm \righttrim cm \verticaltrim cm},clip,width=0.5\linewidth]{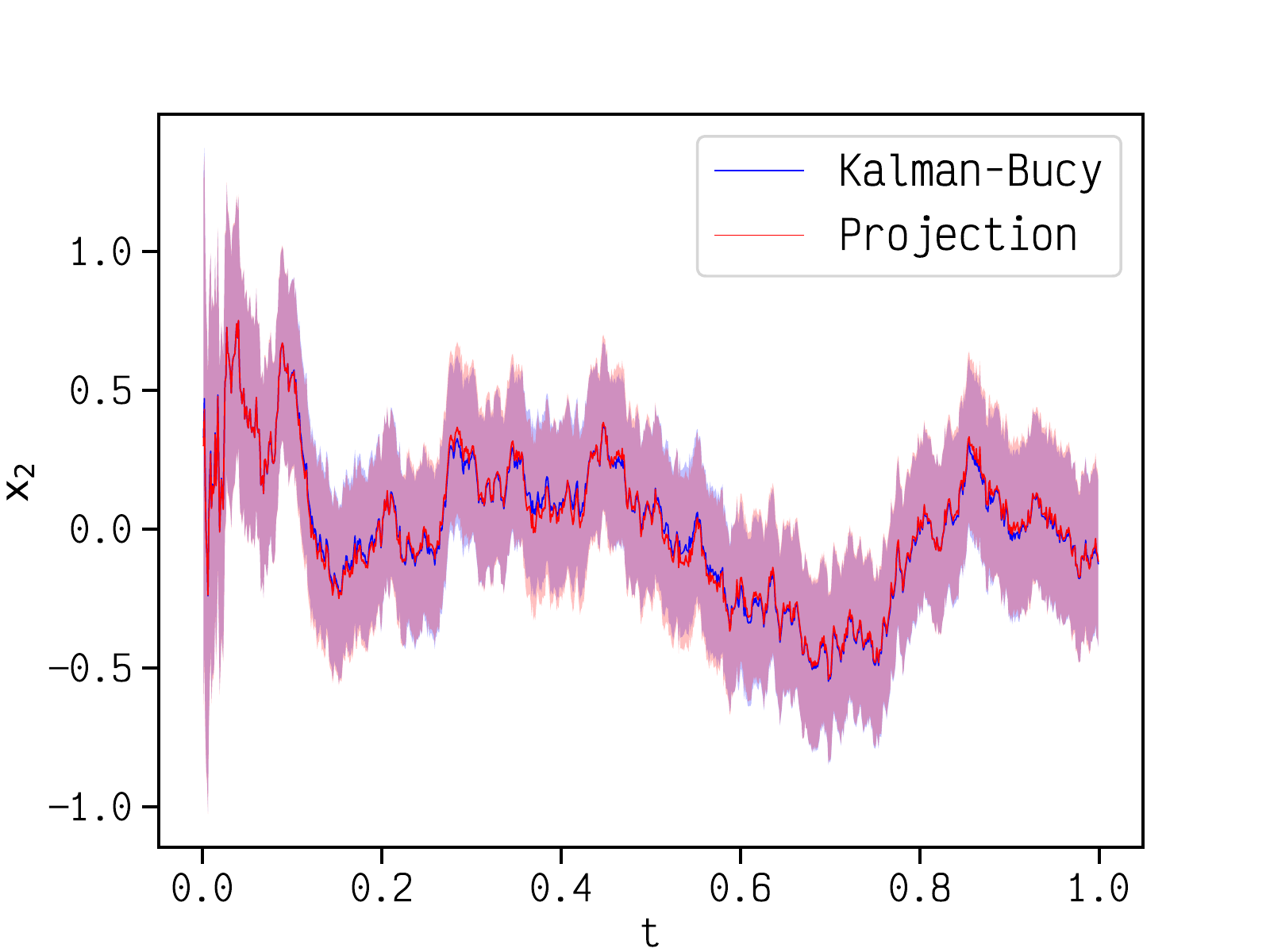}}\\
\subfloat[SPG-level 6]{\label{fig:x_1_pf_vs_kf_spg_1537}\includegraphics[trim={\lefttrim cm \verticaltrim cm \righttrim cm \verticaltrim cm},clip,width=0.5\linewidth]{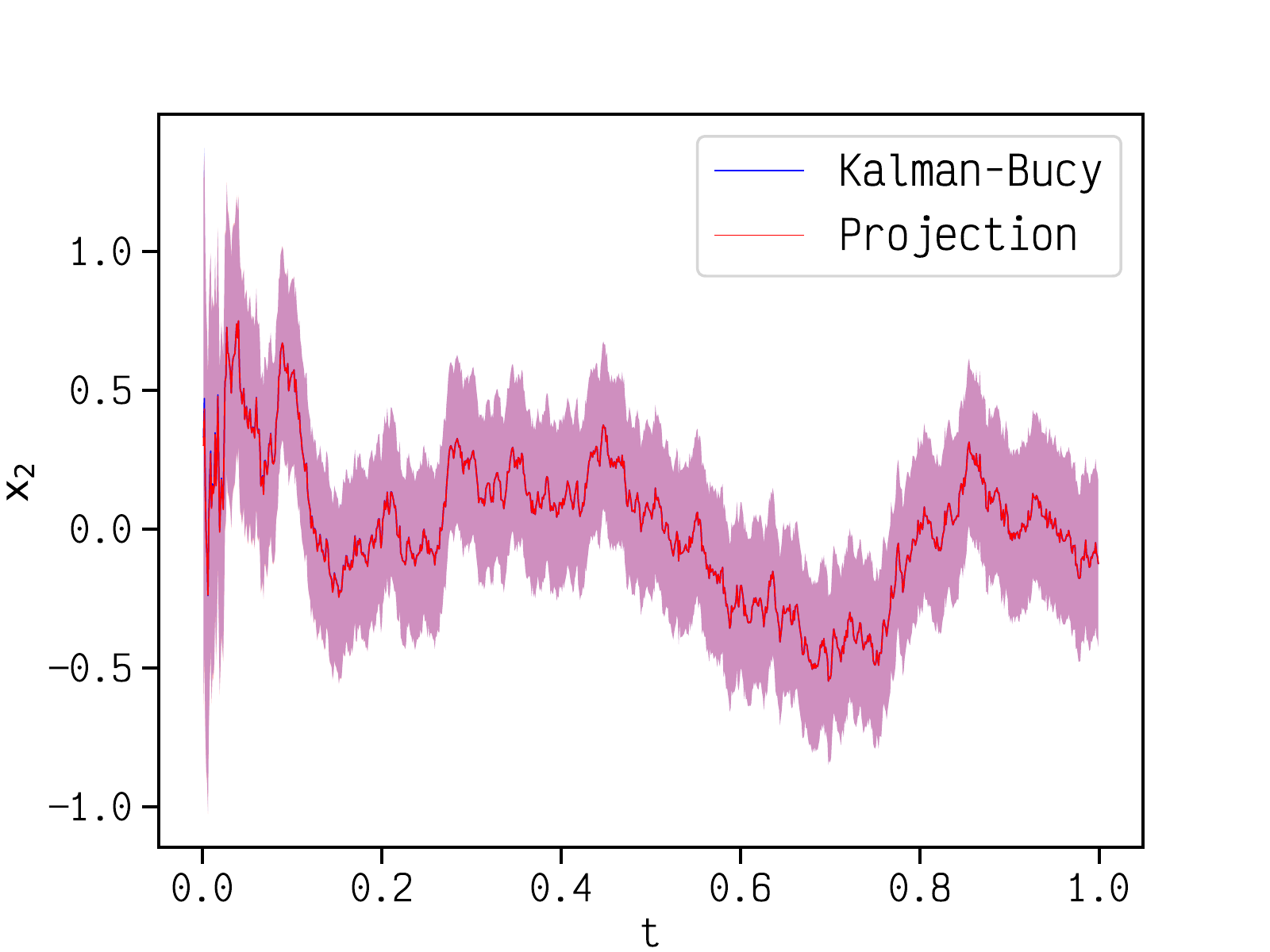}}
\subfloat[qMC-769]{\label{fig:x_1_pf_vs_kf_qmc_1537}\includegraphics[trim={\lefttrim cm \verticaltrim cm \righttrim cm \verticaltrim cm},clip,width=0.5\linewidth]{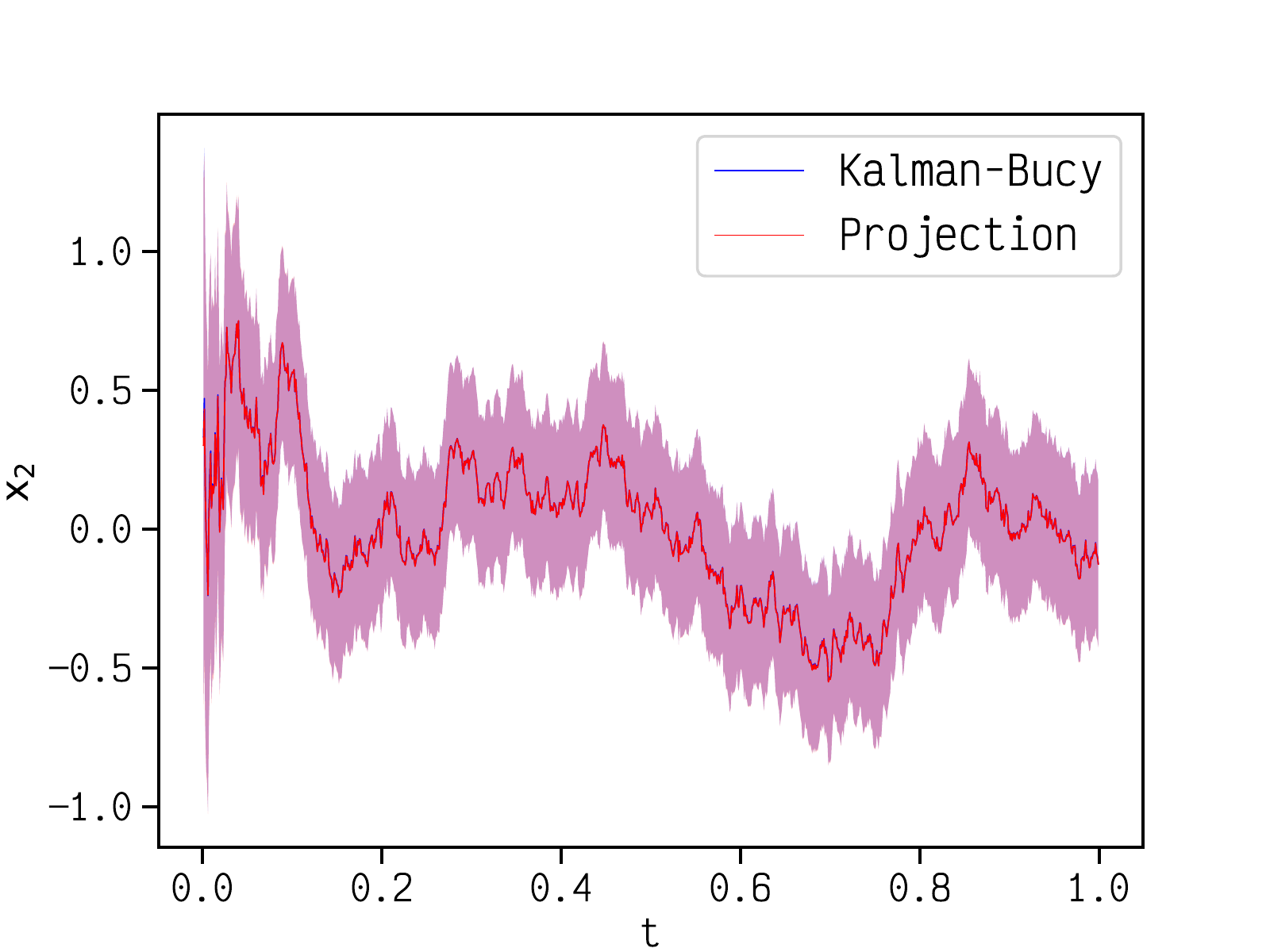}}

\caption{Comparison of mean and one standard deviation confidence intervals for $x_{2,t}$ obtained with the Kalman--Bucy filter and the projection filter solved via different integration setups. \label{fig:Kalman_Bucy_realizations_x_2}}
\end{figure}

\subsubsection{Van der Pol oscillator}

In this section, we compare the projection filter, the numerical solution of the Kushner--Stratonovich equation obtained via a Crank--Nicolson scheme, and a bootstrap particle filter with systematic resampling \citep{Chopin2020}. The filtering model considered is the partially observed the Van der Pol oscillator:
\begin{subequations}
    \begin{align}
        d\mqty[x_{1,t}\\x_{2,t}] &= \mqty[x_{2,t}\\ \mu (1 - x_{1,t}^2)x_{2,t} - x_{1,t}] dt + \mqty[0\\ \sigma_w] dW_t,\\
        dY &= x_{1,t} dt + \sigma_v dV_t.
    \end{align}\label{eqs:SDE_VDP_n_d_linear}    
\end{subequations}
For simulation, we set $\mu = 0.3$ and $\sigma_v=\sigma_w = 1$. We also set $dt=2.5\times 10^{-4}$. We discretize the dynamic model \eqref{eqs:SDE_VDP_n_d_linear} using the Euler--Maruyama scheme for both the particle filter and the simulation. For the particle filter, we use $4\times 10^4$ samples at each time. The time step of the finite-difference scheme is set to be $dt=1\times 10^{-3}$. The two-dimensional grid is set to be $(-6,6)^2$ with $150^2$ equidistant points. In the projection filter we use a sparse-grid integration scheme where we set the level to 8, which corresponds to 4097 quadrature points, and we set the bijection to $\text{arctanh}$.   
The natural statistics in the projection filter are set to be $\left\{ x^{\bunderline{i}} \right\}$, where $\bunderline{i} \in \{ (0,1),(0,2),(0,3),(0,4),(1,0),(1,1),$ $(1,2),(1,3),(2,0),(2,1),(2,2),(3,0),(3,1)$ $,(4,0) \}$. Further, the initial density is set to be a Gaussian density, with variance equal to $5\times 10^1 \sigma_w^2 dt I$ and mean at the origin. 

\begin{figure}[!h]
    \centering
    \includegraphics[trim={0.cm 0cm \righttrim cm 0cm},clip,width=0.60\linewidth]{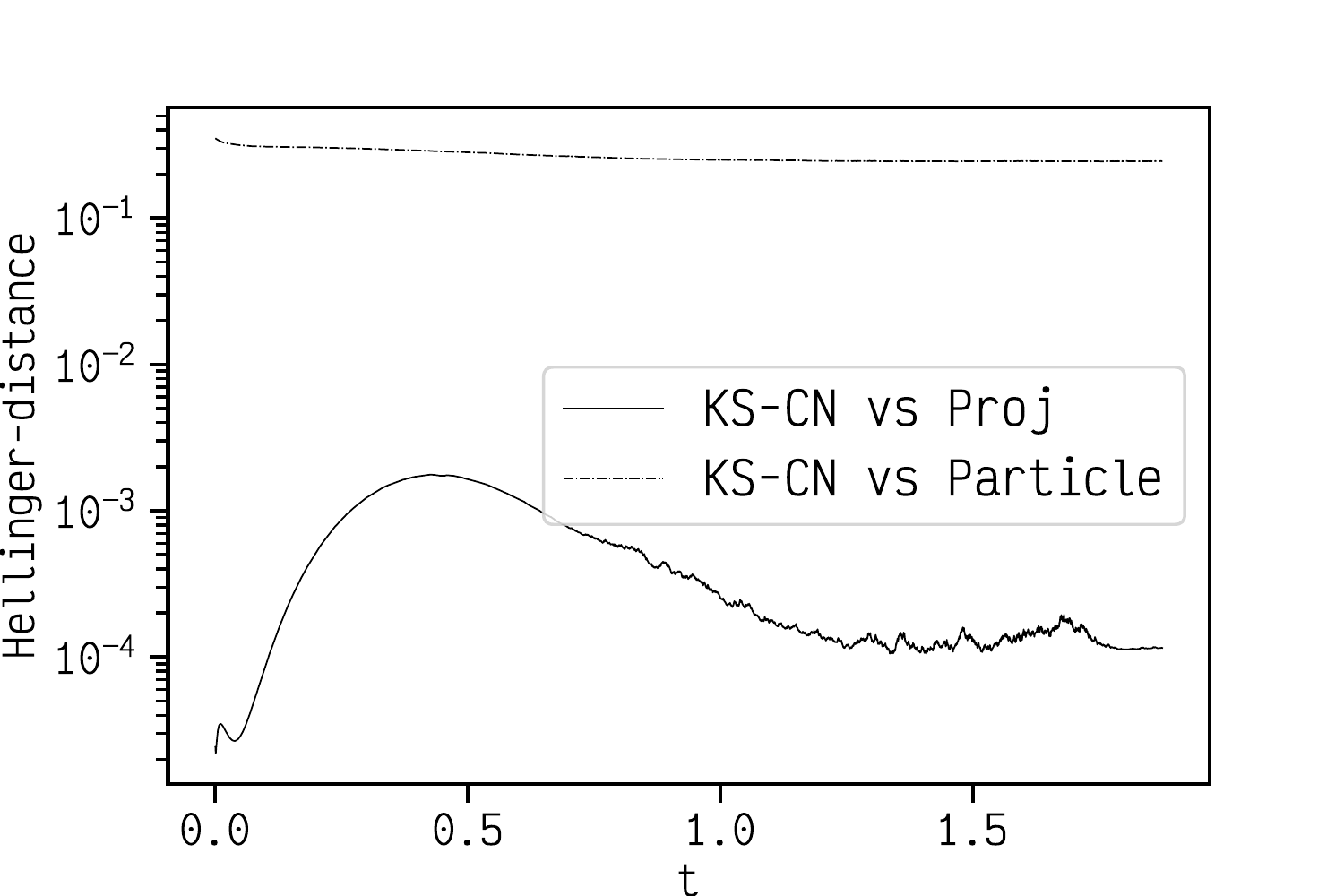}
    \caption{Hellinger distance between the projection filter densities for the Van der Pol oscillator and the densities obtained from the Crank--Nicolson solution to the Kushner--Stratonovich equation (solid black) along with the Hellinger distances between the Gaussian kernel density estimates computed from the particle filter samples and the densities obtained by the Crank--Nicolson solution (dashed black). \label{fig:Hell_vdp}}
\end{figure}

Figure~\ref{fig:Hell_vdp} shows time of evolution of the Hellinger distance between the Crank--Nicolson solution to the Kushner--Stratonovich SPDE and the projection filter, and the Hellinger distance between the Crank--Nicolson solution and the particle filter (estimated using a Gaussian kernel density estimate). It can be seen that the Hellinger distances of the projection filter are significantly smaller than those of the particle filter. Figures \ref{fig:Zakai_vs_GaussianKDE_vs_projection_filter} show the evolution of the approximate densities obtained by the Crank--Nicolson method, particle filter, and the projection filter. There figures confirm that the projection filter approximates the Kushner--Stratonovich solution better than the particle filter. We note though that the Crank--Nicolson solution to the Kushner--Stratonovich equation is only stable up to around 1.75s. After that, a high frequency oscillation starts to appear (not shown). Thus we believe that the Crank--Nicolson solution of the Kushner--Stratonovich equation is no longer accurate after 1.75s. This numerical stability issue of the Crank--Nicolson scheme for the Kushner--Stratonovich equation might also contribute to the Hellinger distance between the finite-difference approximation of the Kushner--Stratonovich equation and the density of the projection filter in Figure~\ref{fig:Hell_vdp}.

A comparison of the different moment approximations is shown in Figures 
\ref{fig:Error_comparison_first_part}, and  \ref{fig:Error_comparison_second_part}. 
The figures show the absolute difference between these expected values in log scale, where $\E_a, \E_b, $ and $\E_c$ are the expectation operator where the densities are obtained using Crank--Nicolson scheme, projection filter, and  the particle filter, respectively. Figures \ref{fig:Error_comparison_first_part}, and  \ref{fig:Error_comparison_second_part} clearly demonstrate that in terms of expectations of the natural statistics and their respective absolute errors, in the present setup, the projection filter offers a much better approximation to the problem than the particle filter.

\newcommand{\lefttrimDensity}{0.45}
\newcommand{\righttrimDensity}{0.25}
\newcommand{\verticaltrimDensity}{0.0}
\begin{figure}[!h]
\centering
\makebox[\textwidth][c]{
\includegraphics[trim={\lefttrimDensity cm \verticaltrimDensity cm \righttrimDensity cm \verticaltrimDensity cm},clip,width=0.98
\textwidth]{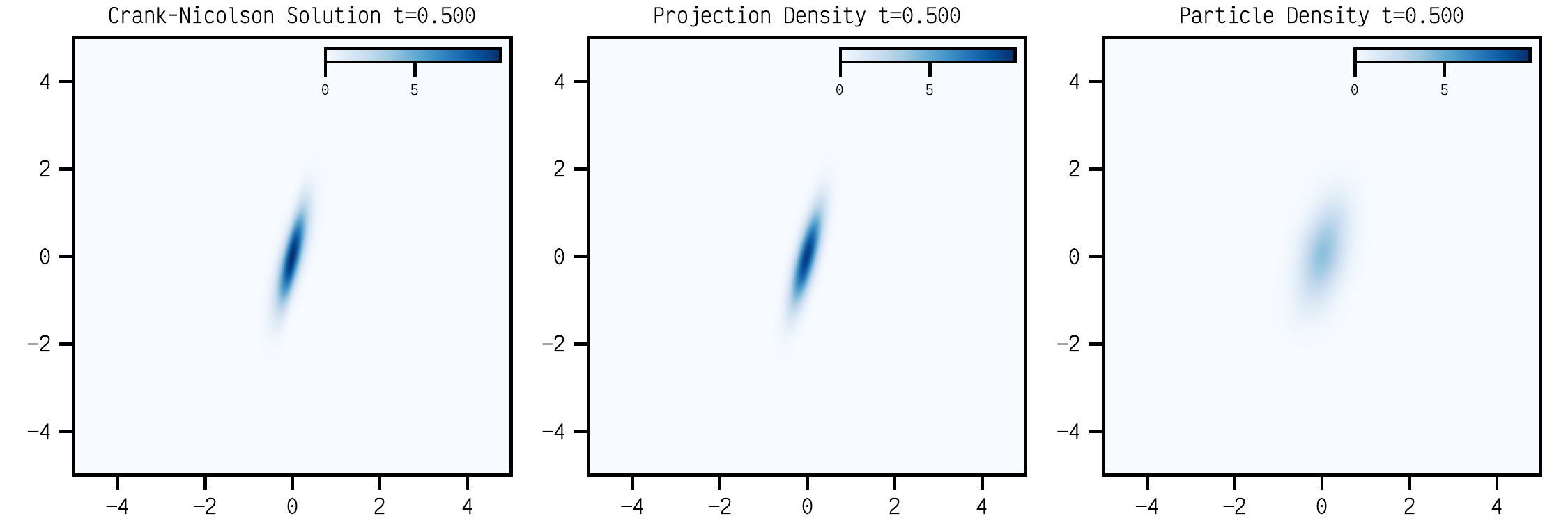}
}
\makebox[\textwidth][c]{
\includegraphics[trim={\lefttrimDensity cm \verticaltrimDensity cm \righttrimDensity cm \verticaltrimDensity cm},clip,width=0.98
\textwidth]{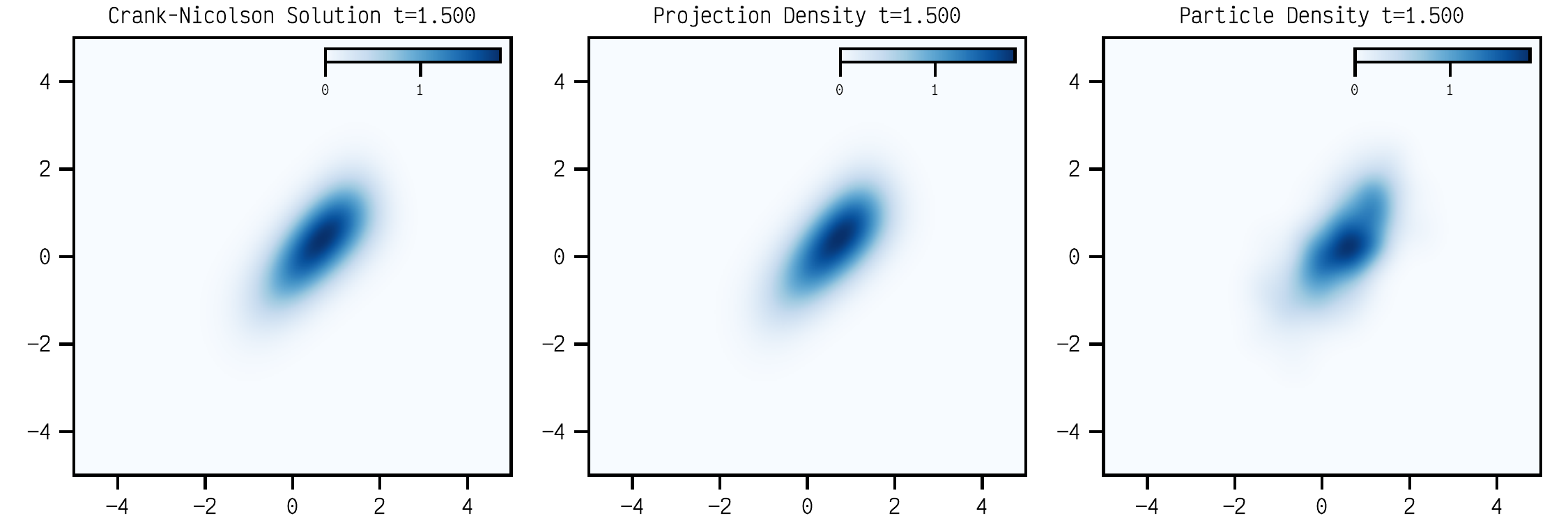}
}
\makebox[\textwidth][c]{
\includegraphics[trim={\lefttrimDensity cm \verticaltrimDensity cm \righttrimDensity cm \verticaltrimDensity cm},clip,width=0.98
\textwidth]{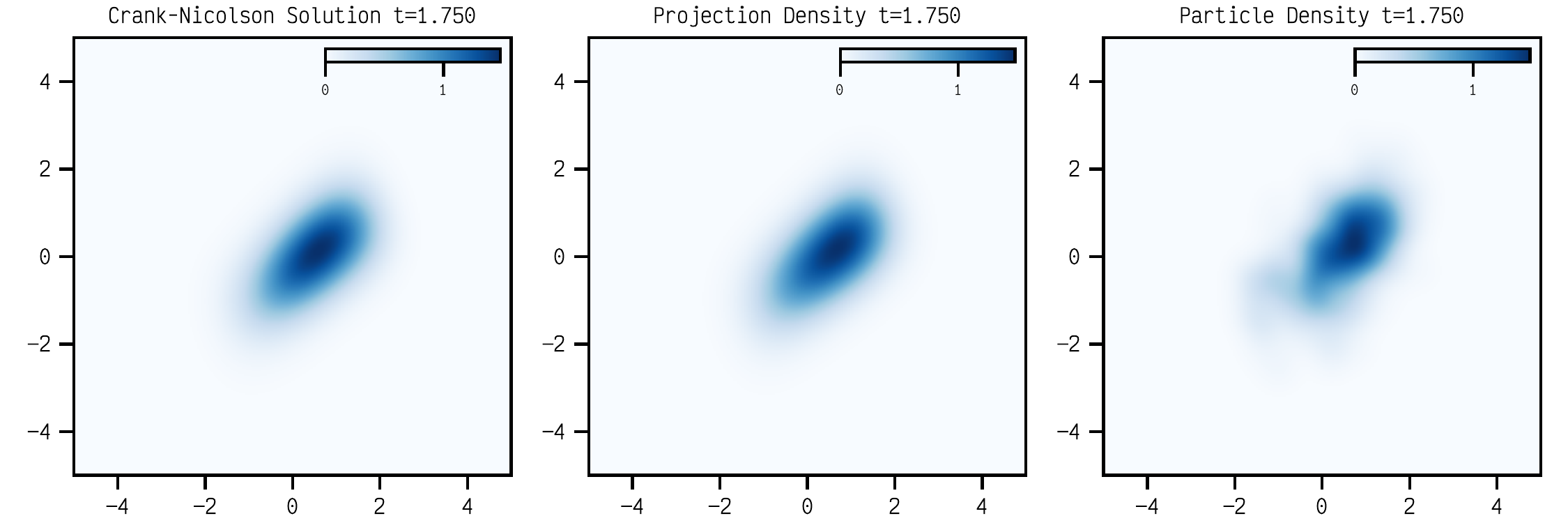}
}
\caption{Comparison of solution of the Crank--Nicolson solution to the Kushner--Stratonovich equation, the densities of the projection filter for Van der Pol dynamics \eqref{eqs:SDE_VDP_n_d_linear}, and Gaussian kernel density estimates from the solution of the particle filter. The Gaussian kernel density estimate smoothing coefficient is set to be proportional to the Cholesky factor of the sample covariance matrix. 
\label{fig:Zakai_vs_GaussianKDE_vs_projection_filter}}
\end{figure}

\newcommand{\lefttrimMoment}{0.45}
\newcommand{\righttrimMoment}{0.25}
\newcommand{\verticaltrimMoment}{0.45}
\newcommand{\figcolumnwidth}{0.5\textwidth}
\newcommand{\figurewidth}{0.49\textwidth}

\begin{figure}[!h]
\centering
\makebox[\textwidth][c]{
\begin{tabular}{m{\figcolumnwidth}m{\figcolumnwidth}m{\figcolumnwidth}}
\includegraphics[trim={\lefttrimMoment cm \verticaltrimMoment cm \righttrimMoment cm \verticaltrimMoment cm},clip,width=\figurewidth]{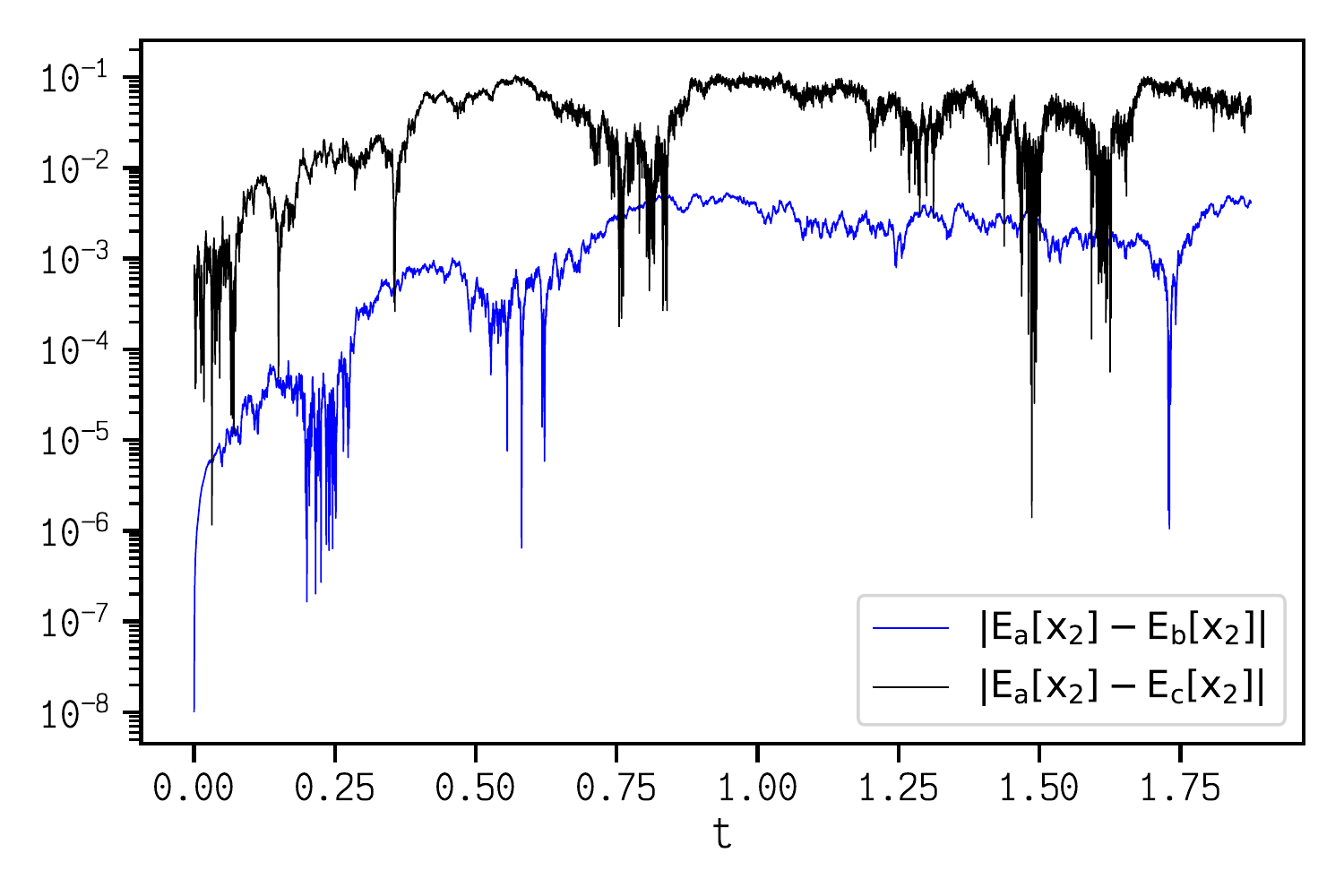}&
\includegraphics[trim={\lefttrimMoment cm \verticaltrimMoment cm \righttrimMoment cm \verticaltrimMoment cm},clip,width=\figurewidth]{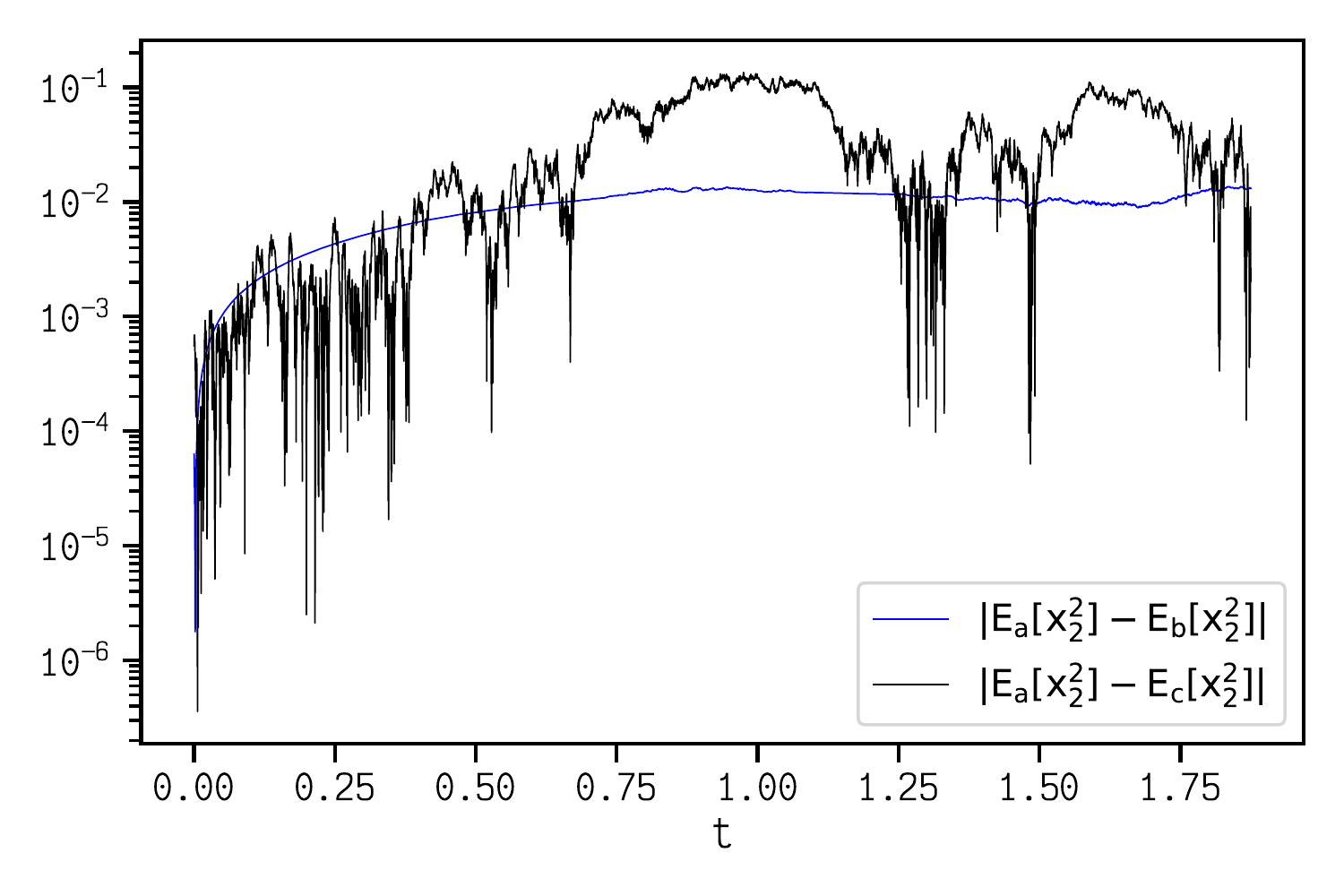}\\
\includegraphics[trim={\lefttrimMoment cm \verticaltrimMoment cm \righttrimMoment cm \verticaltrimMoment cm},clip,width=\figurewidth]{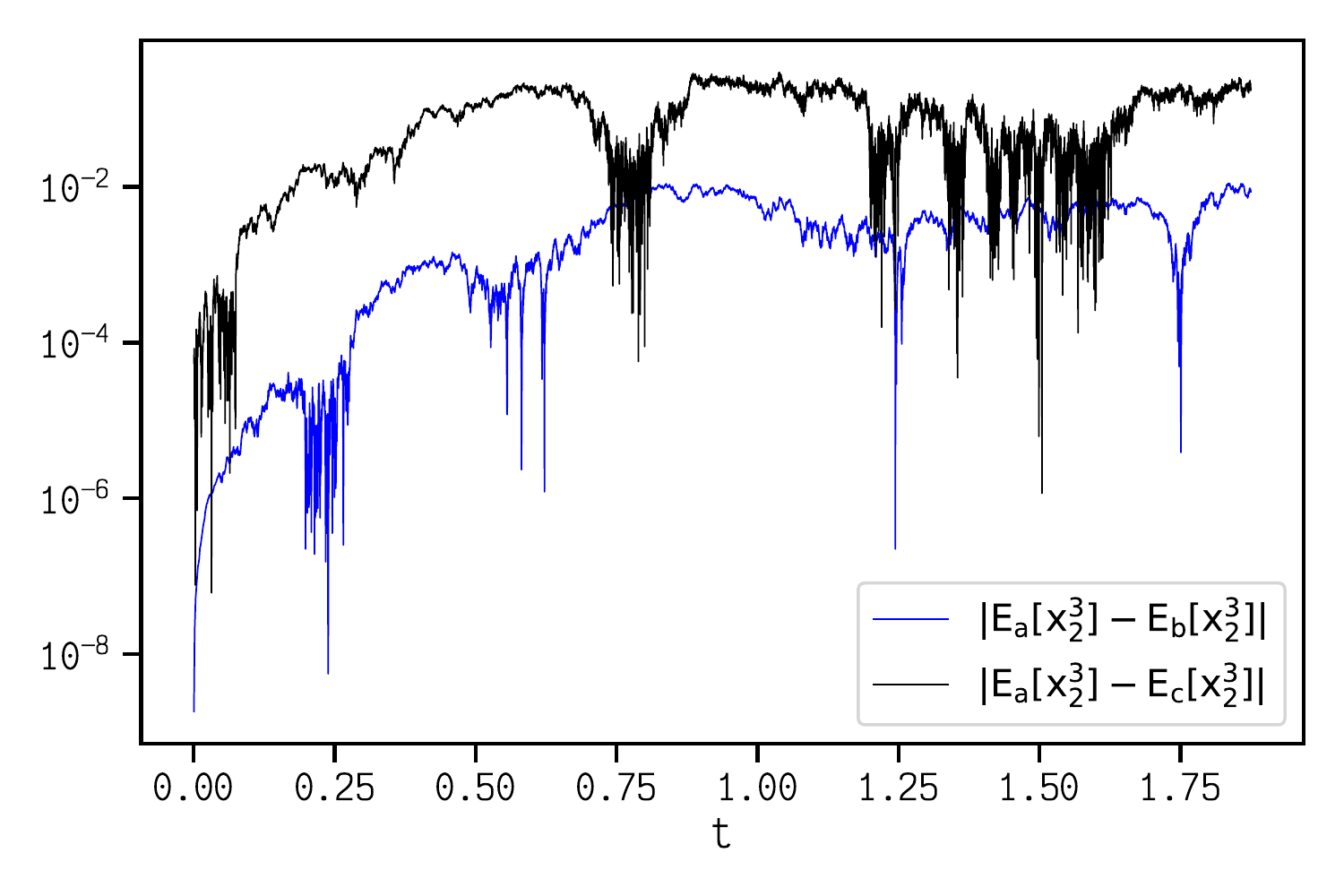}&
\includegraphics[trim={\lefttrimMoment cm \verticaltrimMoment cm \righttrimMoment cm \verticaltrimMoment cm},clip,width=\figurewidth]{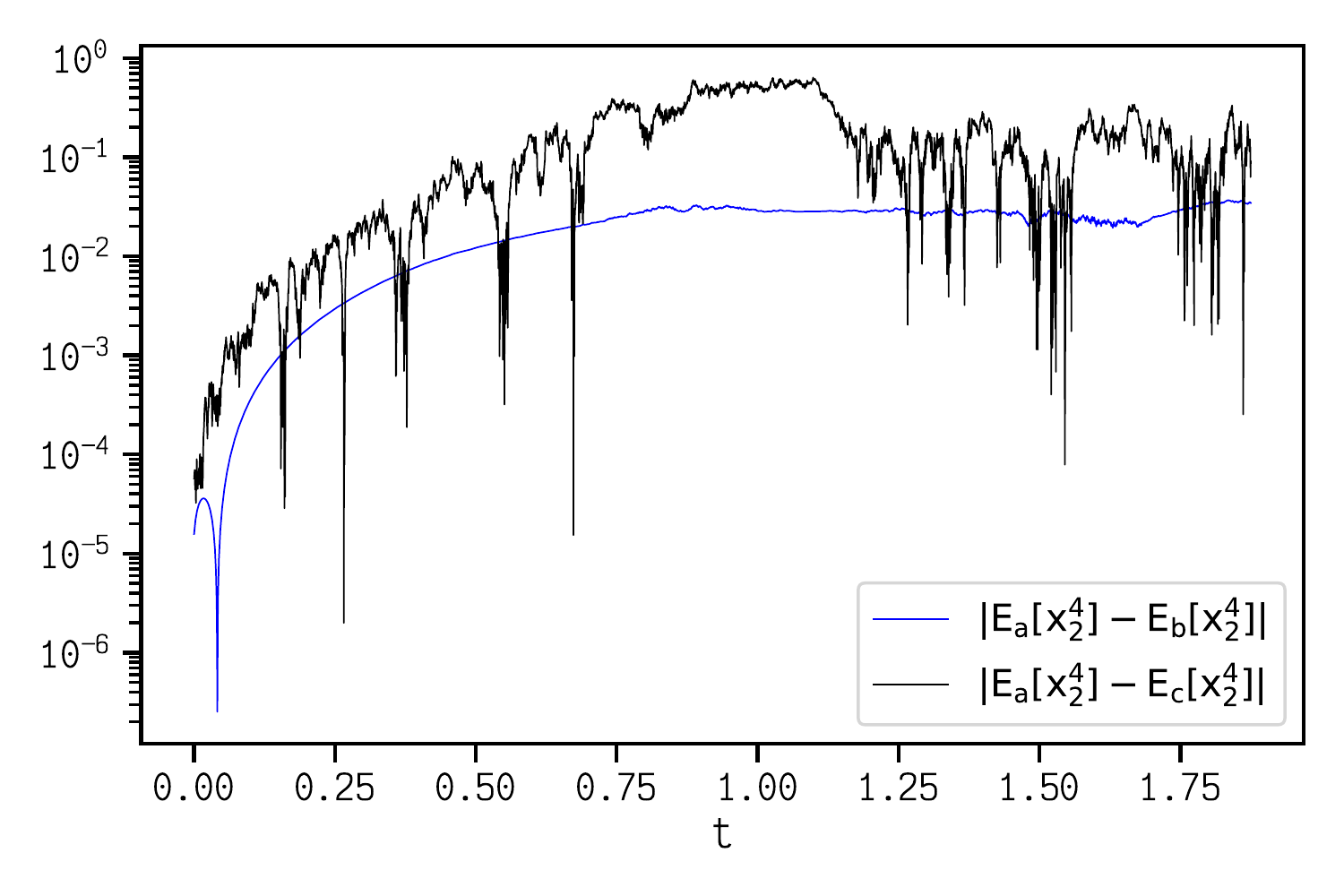}\\
\includegraphics[trim={\lefttrimMoment cm \verticaltrimMoment cm \righttrimMoment cm \verticaltrimMoment cm},clip,width=\figurewidth]{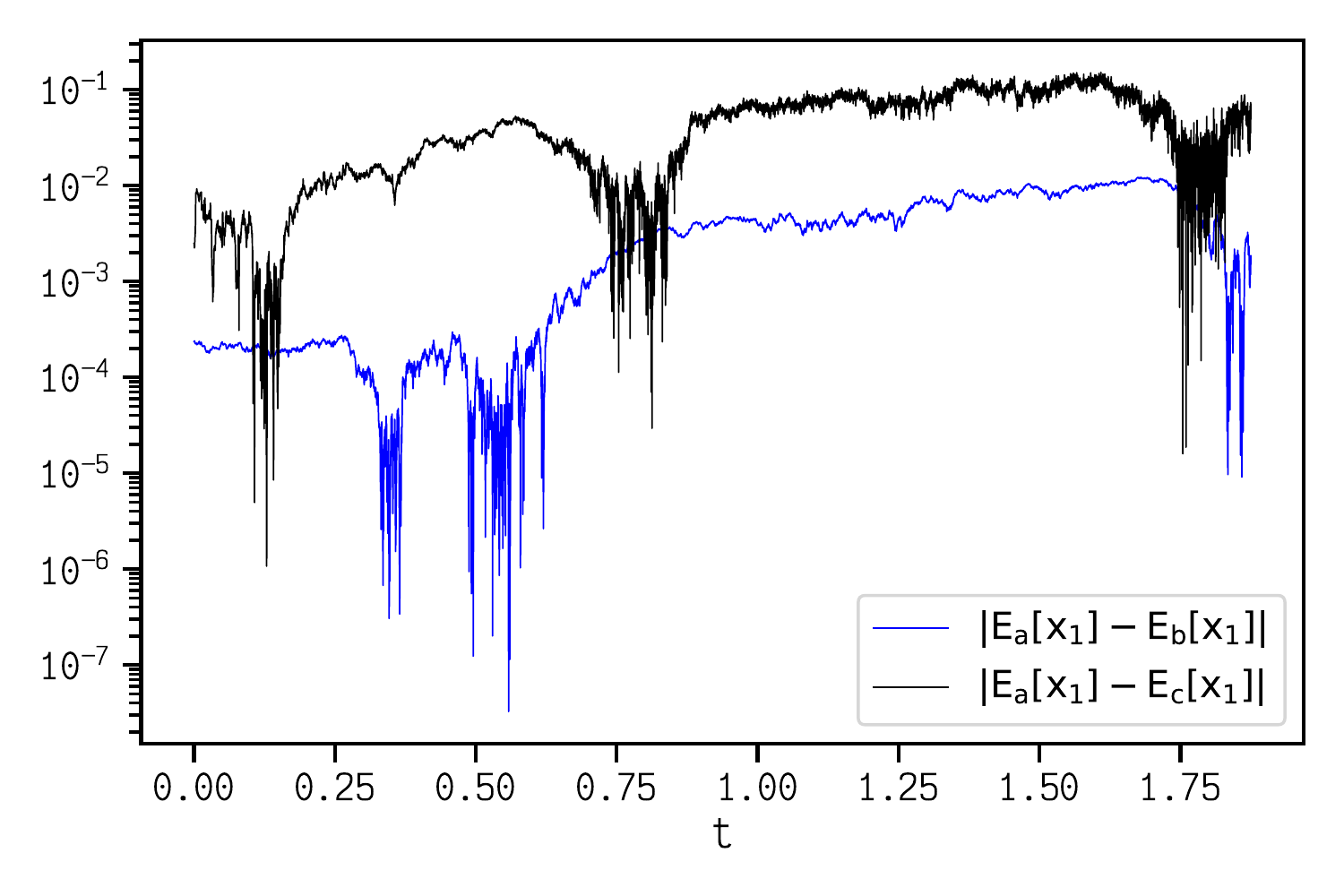}&
\includegraphics[trim={\lefttrimMoment cm \verticaltrimMoment cm \righttrimMoment cm \verticaltrimMoment cm},clip,width=\figurewidth]{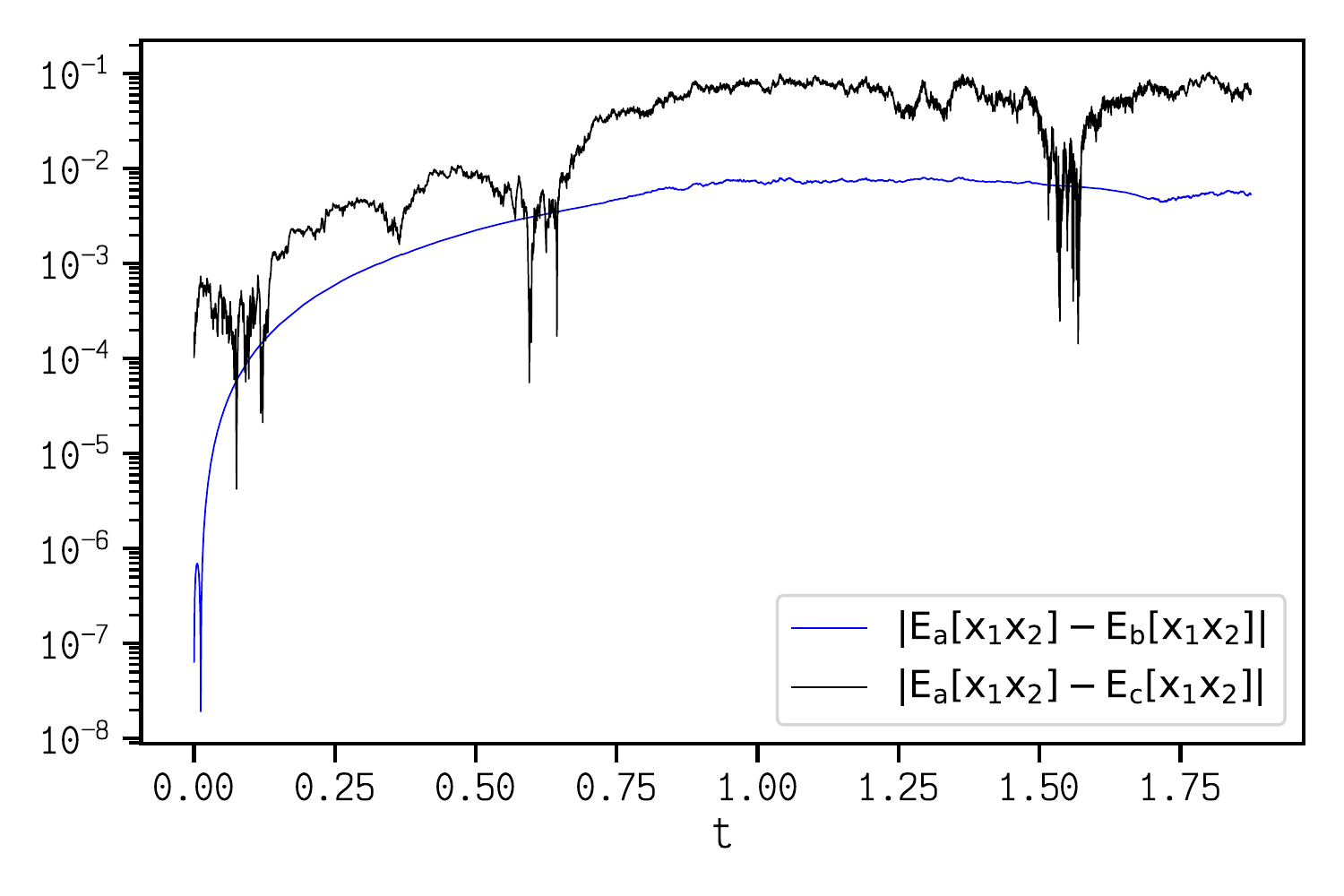}
\end{tabular}
}
\caption{Absolute differences of the approximate expected values of the natural statistics. In these figures $\E_a$ is the expectation operator where the densities are obtained using Crank--Nicolson scheme, $\E_b$ uses densities obtained using projection filter,  and $\E_c$ is the expectation operator where the densities are the empirical densities of the particle filter, respectively.
\label{fig:Error_comparison_first_part}}
\end{figure}

\begin{figure}[!h]
\centering
\makebox[\textwidth][c]{
\begin{tabular}{m{\figcolumnwidth}m{\figcolumnwidth}m{\figcolumnwidth}}
\includegraphics[trim={\lefttrimMoment cm \verticaltrimMoment cm \righttrimMoment cm \verticaltrimMoment cm},clip,width=\figurewidth]{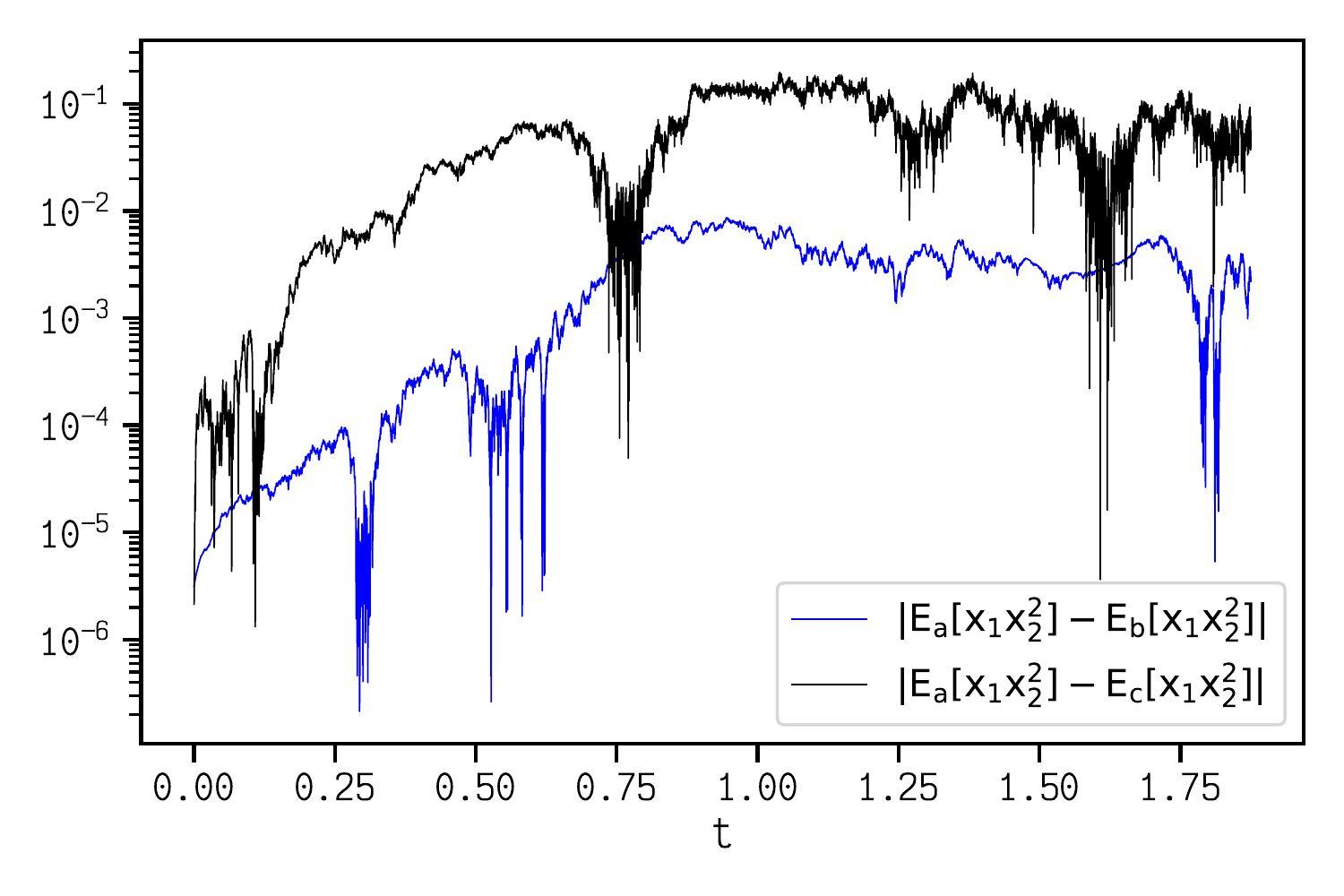}&
\includegraphics[trim={\lefttrimMoment cm \verticaltrimMoment cm \righttrimMoment cm \verticaltrimMoment cm},clip,width=\figurewidth]{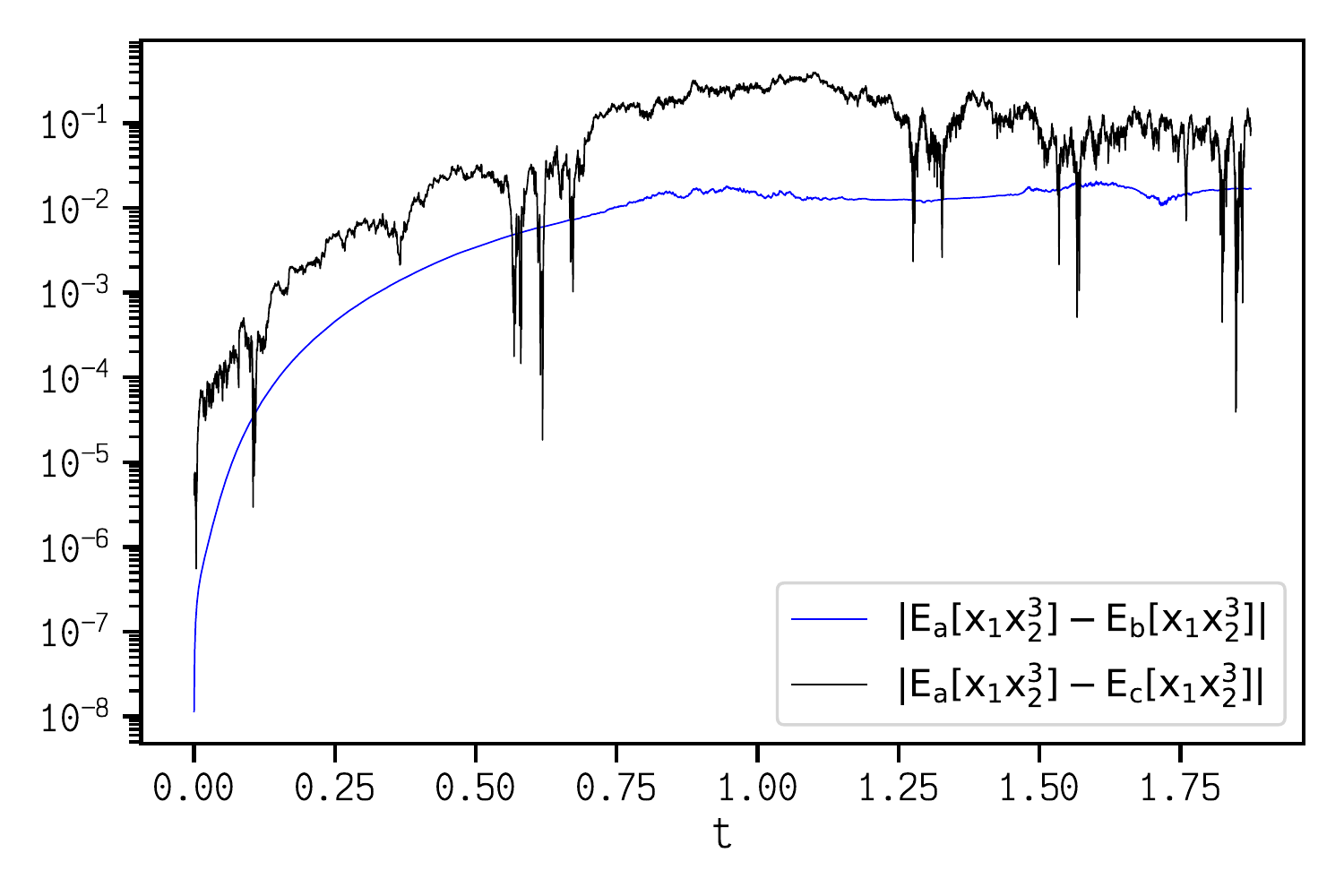}\\
\includegraphics[trim={\lefttrimMoment cm \verticaltrimMoment cm \righttrimMoment cm \verticaltrimMoment cm},clip,width=\figurewidth]{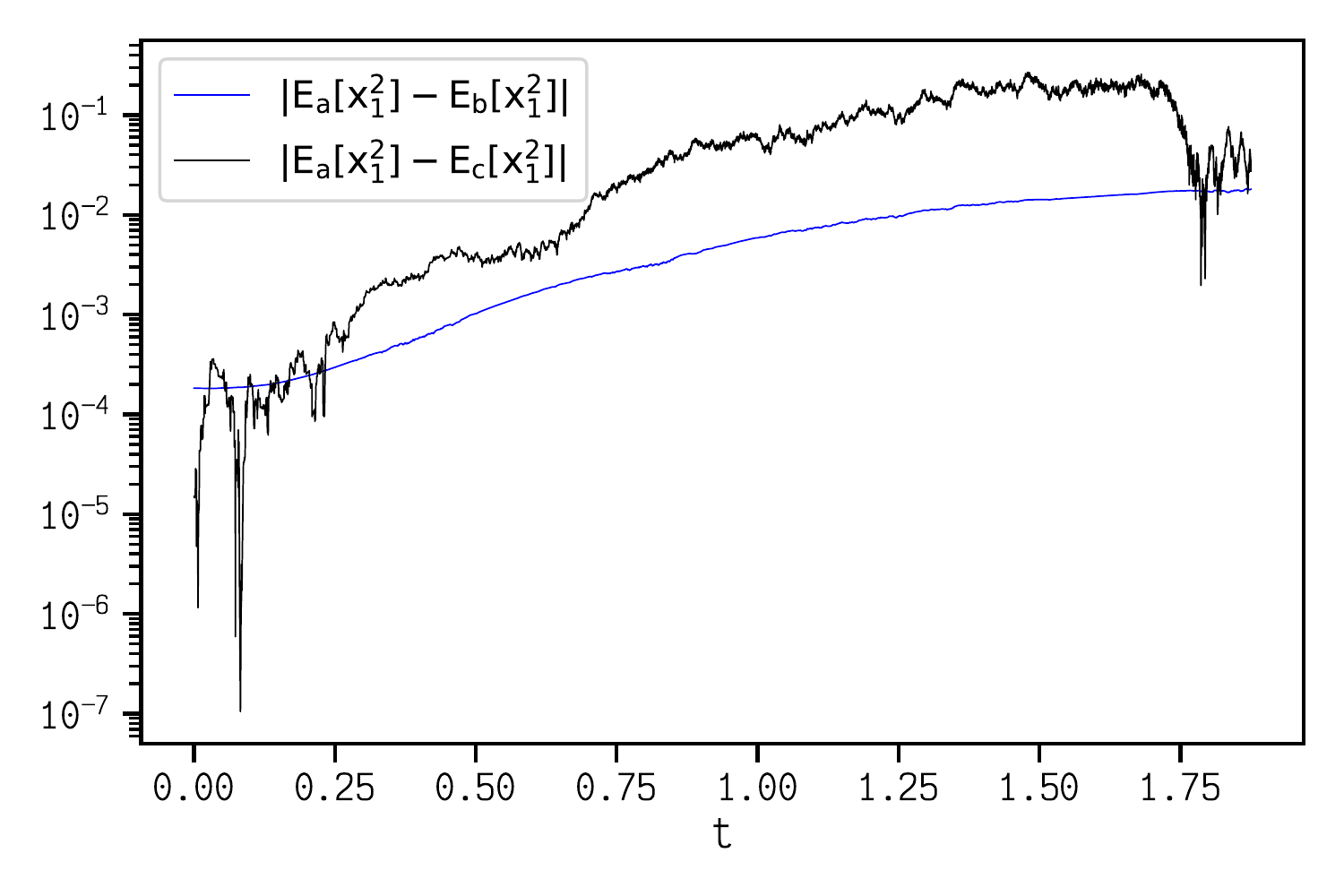}&
\includegraphics[trim={\lefttrimMoment cm \verticaltrimMoment cm \righttrimMoment cm \verticaltrimMoment cm},clip,width=\figurewidth]{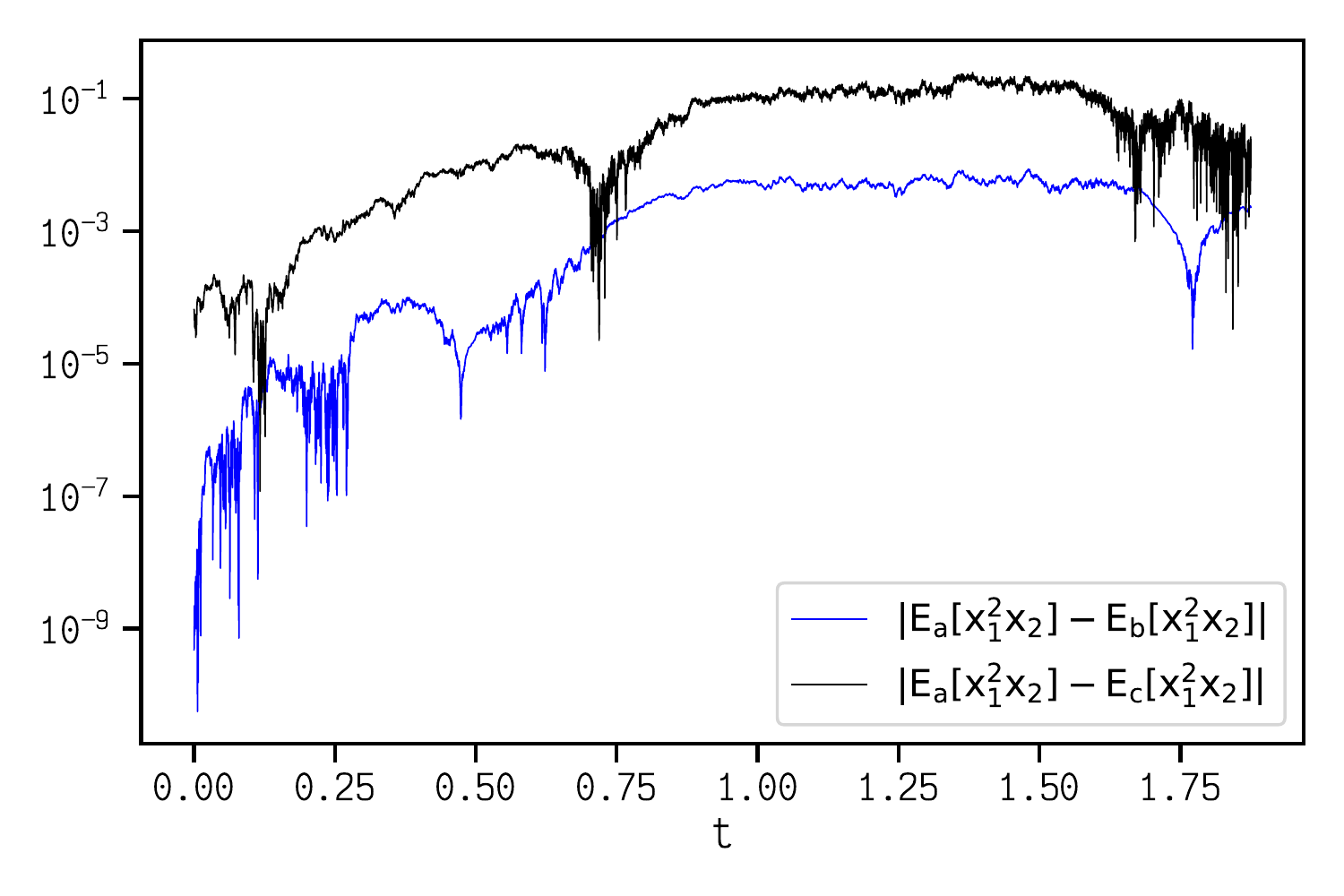}\\
\includegraphics[trim={\lefttrimMoment cm \verticaltrimMoment cm \righttrimMoment cm \verticaltrimMoment cm},clip,width=\figurewidth]{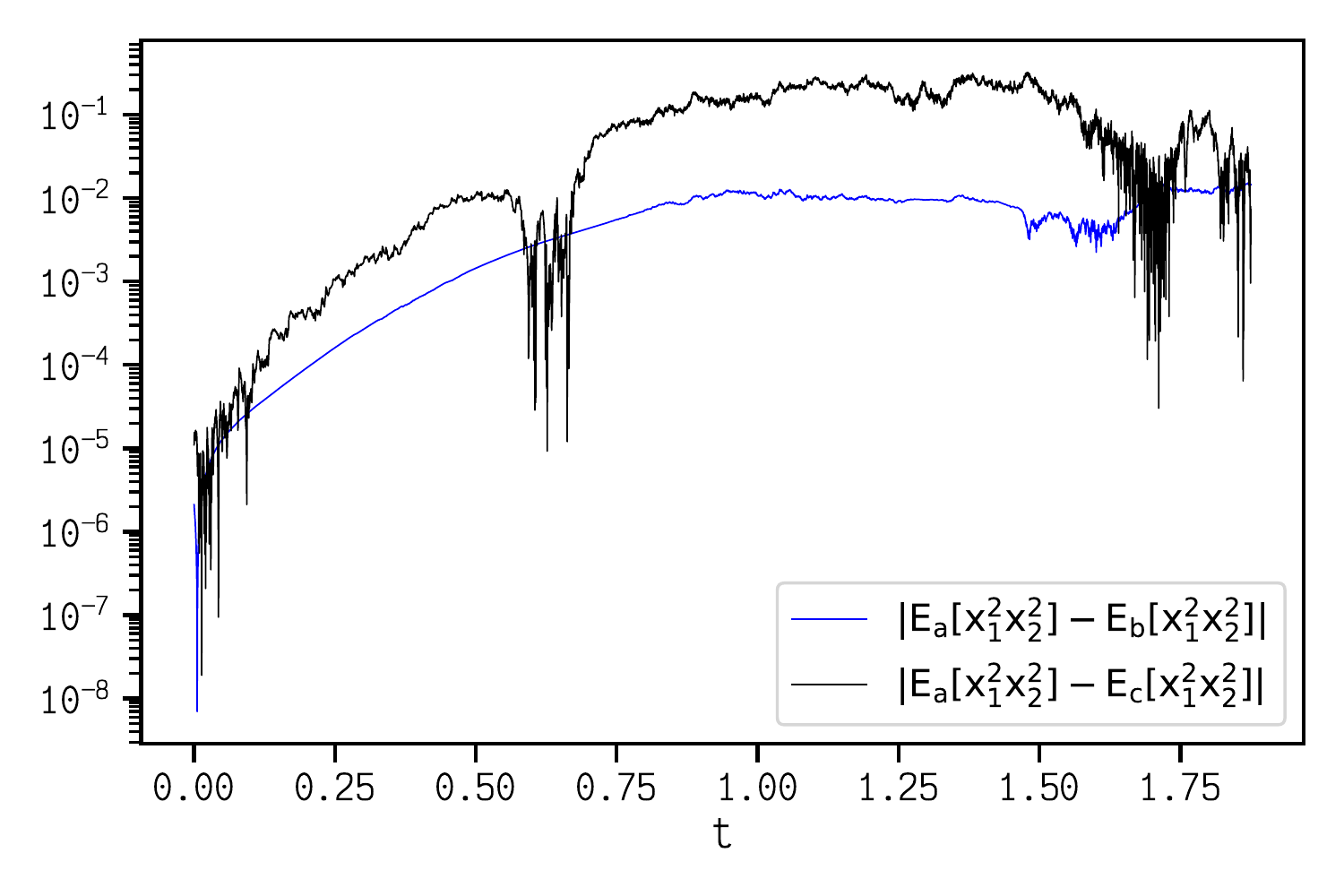}&
\includegraphics[trim={\lefttrimMoment cm \verticaltrimMoment cm \righttrimMoment cm \verticaltrimMoment cm},clip,width=\figurewidth]{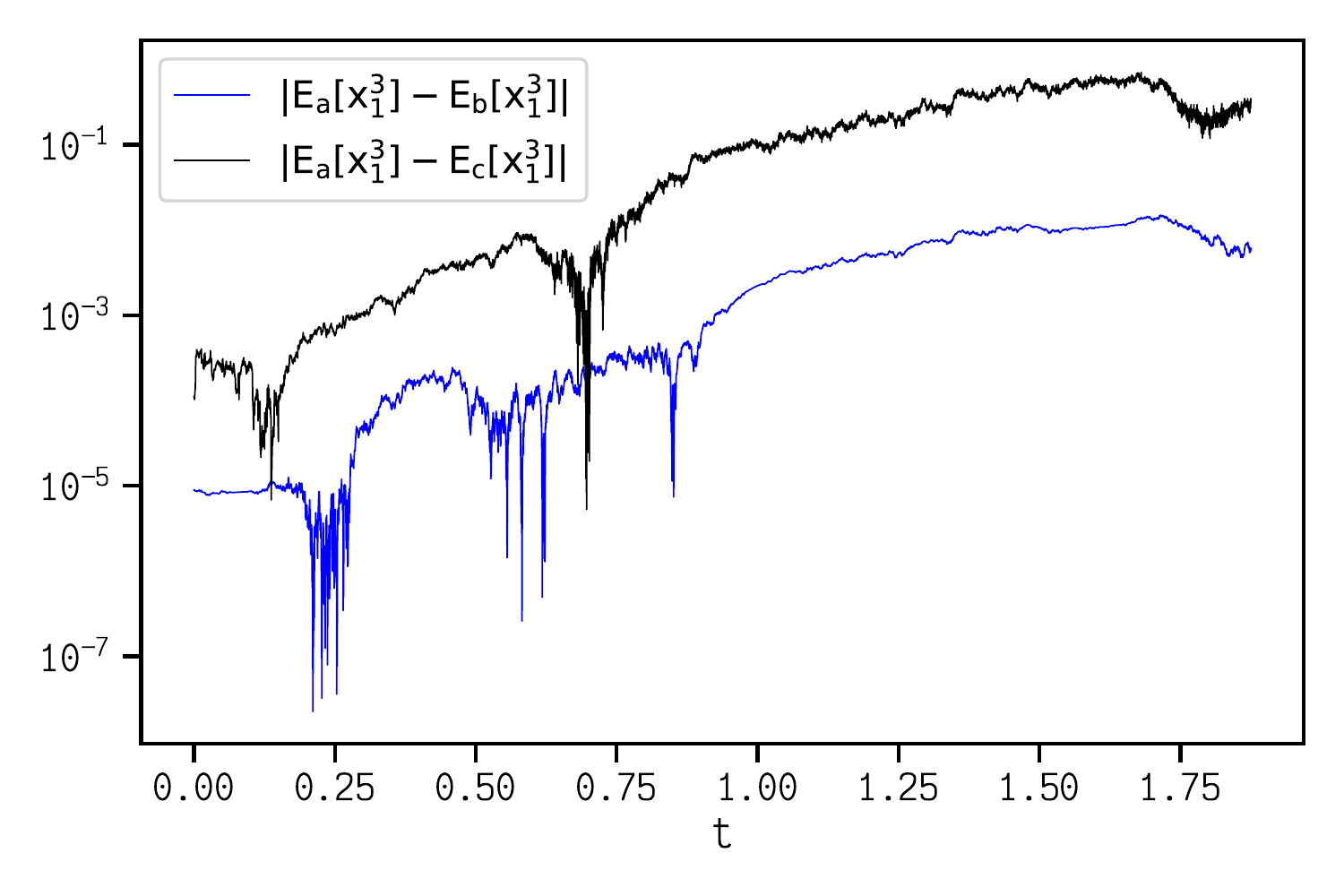}\\
\includegraphics[trim={\lefttrimMoment cm \verticaltrimMoment cm \righttrimMoment cm \verticaltrimMoment cm},clip,width=\figurewidth]{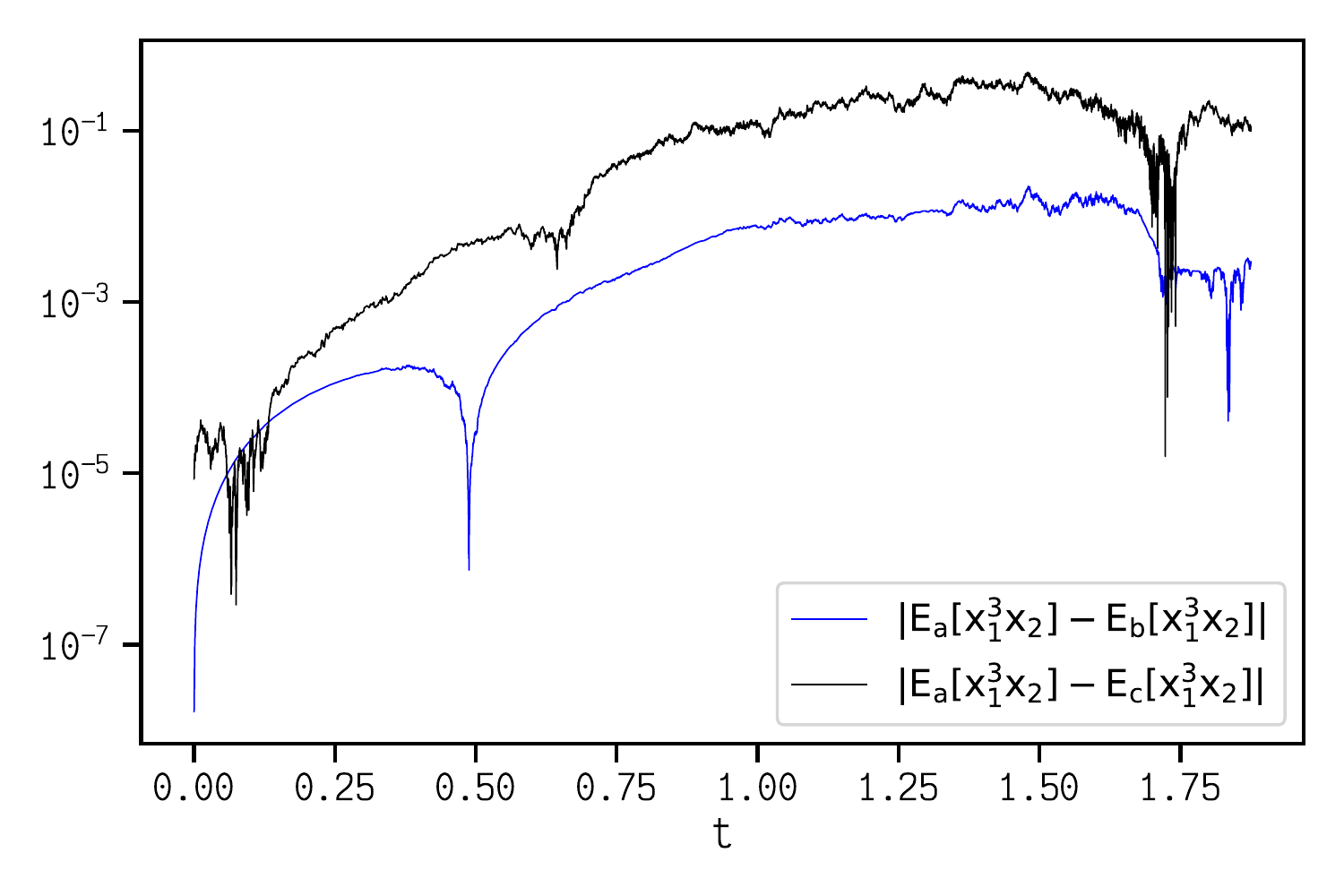}&
\includegraphics[trim={\lefttrimMoment cm \verticaltrimMoment cm \righttrimMoment cm \verticaltrimMoment cm},clip,width=\figurewidth]{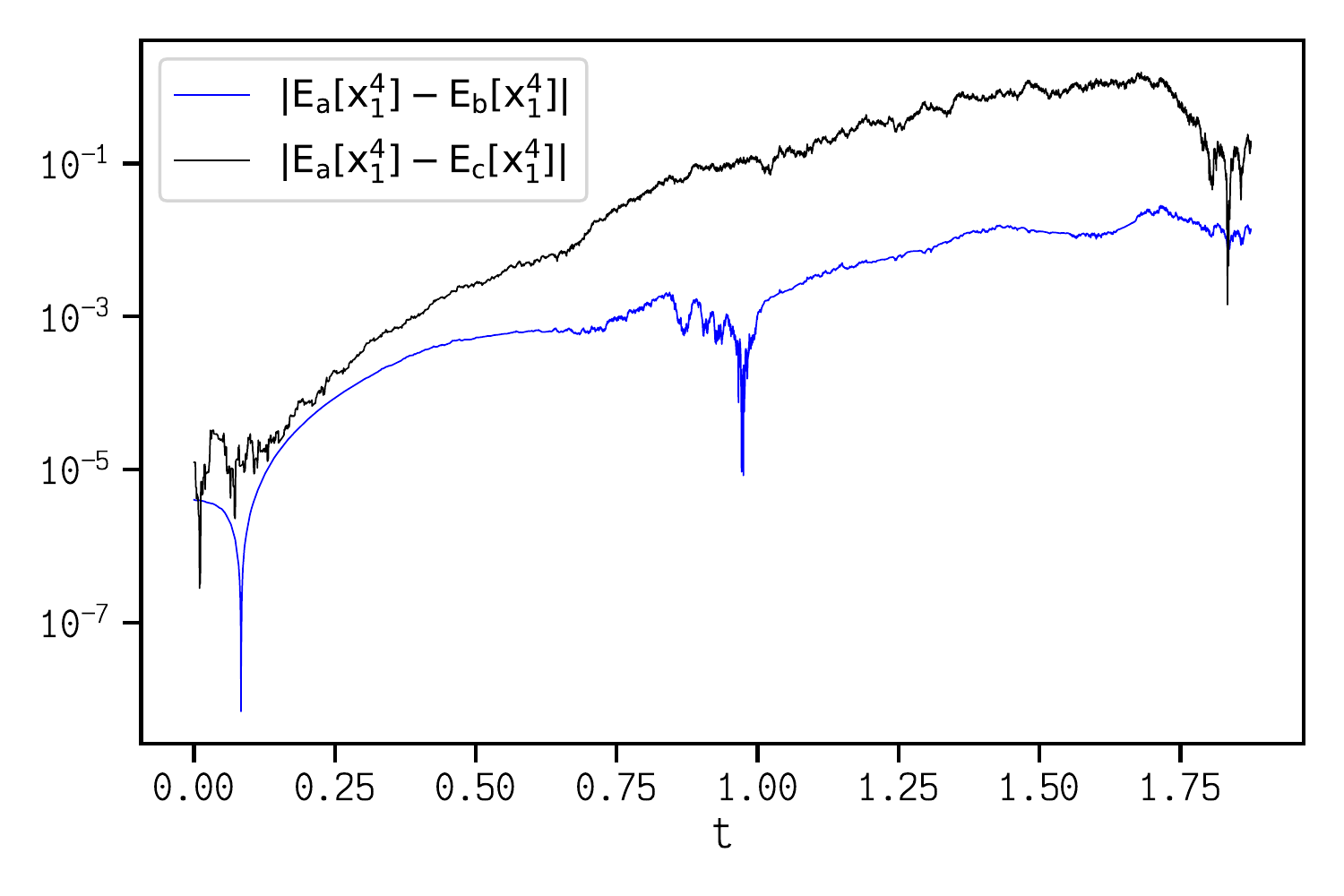}
\end{tabular}
}
\caption{Continuation of Figure \ref{fig:Error_comparison_first_part} for different sets of natural statistics. \label{fig:Error_comparison_second_part} }
\end{figure}

\section{Conclusions and discussion}\label{sec:conclusion}

We have proposed a novel implementation of the projection filter for a class of dynamical systems. In this new approach, we have successfully removed some difficulties in projection filter implementation by calculating all the moments and the Fisher information matrix using automatic differentiation knowing that they can be obtained as partial derivatives of the log partition function. We calculate the log-partition function via Gauss--Chebyshev quadrature in unidimensional problems and via sparse-grid integration in multidimensional problems. We also replace a recursion procedure typically used to obtain the higher moments -- which is valid only for unidimensional problems -- with a simple partial derivative evaluation of the log-partition function that works for multidimensional problems. We have shown mathematically that the resulting approximation error of the log-partition function and its first and second derivatives have linear relations to the integration error of partition functions and converge to zero as the number of integration grid points approaches infinity. This result enables us to focus only on choosing integration grid such that the integration error of the partition function is acceptable.

We have shown the effectiveness of the proposed method by comparing it to a finite-difference approximation of the Kushner--Stratonovich equation and a bootstrap particle filter. In particular, for a unidimensional problem we have shown that our projection filter implementation can be constructed using as low as $\Chebyshevbasislow$ integration nodes, while still being very close to the finite-difference solution to the Kushner--Stratonovich equation. 

We also compared sparse-grid and quasi Monte Carlo (qMC) implementations of the projection filter to the exact Kalman--Bucy filter solution in a two dimensional linear dynamic model. We found that the sparse-grid based projection filter approximates the Kalman--Bucy filter solution more accurately than qMC. With the sparse-grid level set to 5, the sparse-grid based projection filter has a practically negligible Hellinger distance to the Kalman-Bucy solution after $t=0.4$, while the Hellinger distances of the densities obtained with qMC integration continue to oscillate around $10^{-4}$.

We have also applied the projection filter to non-linear/non-Gaussian partially observed two-dimensional Van der Pol oscillator model. 
For the Van der Pol oscillator filtering problem, we have shown that the projection filter consructed with a level 8 sparse grid gives posterior density approximations that are closer to the Crank--Nicolson solution to the Kushner--Stratonovich equation than a particle filter with $4\times 10^4$ particles. In addition to being lightweight, our projection filter implementation does not suffer from numerical stability issues as the finite-difference or Crank--Nicolson approximations to the Kushner--Stratonovich equation.

Thanks to the sparse-grid integration, our projection filter algorithm is suitable for high-dimensional problems. In addition, if an increased accuracy is needed, it would be possible to employ adaptive sparse-grid techniques for more precise integration of the log partition function. In the future, we will investigate the possibility of using parametric bijections from the canonical cube to $\mathbb{R}^d$ to further reduce the number of quadrature nodes while maintaining the same accuracy level.

\section{Acknowledgements}
Muhammad Emzir would like to express his gratitude to the KFUPM Dean of Research Oversight and Coordination for the SR211015 grant. 

\bibliographystyle{abbrvnat}
\bibliography{reference.bib}
\end{document}